\documentclass{article}

\usepackage{PRIMEarxiv}
\usepackage[numbers,sort&compress]{natbib}
\usepackage[utf8]{inputenc} 
\usepackage[T1]{fontenc}    
\usepackage{hyperref}       
\usepackage{url}            
\usepackage{booktabs}       
\usepackage{amsfonts}       
\usepackage{nicefrac}       
\usepackage{microtype}      
\usepackage{lipsum}
\usepackage{fancyhdr}       
\usepackage{graphicx}       
\usepackage{amsmath}
\usepackage{lyu}
\usepackage{amsthm}
\usepackage{amssymb}
\usepackage{xcolor}
\usepackage{subfigure}
\newtheorem{theorem}{Theorem}[section]

\newtheorem{lemma}[theorem]{Lemma}
\newtheorem{proposition}[theorem]{Proposition}
\theoremstyle{definition}

\newtheorem{assumption}{Assumption}
\newtheorem{remark}[theorem]{Remark}
\definecolor{commentrule}{HTML}{6B7280}
\definecolor{commenttext}{HTML}{1F2937}
\definecolor{responseblue}{HTML}{003B73}
\definecolor{revone}{HTML}{003B73}
\definecolor{revtwo}{HTML}{B22222}
\definecolor{revboth}{HTML}{6A1B9A}
\definecolor{revedit}{HTML}{00695C}

\usepackage{lineno}
\newcommand*\patchAmsMathEnvironmentForLineno[1]{%
  \expandafter\let\csname old#1\expandafter\endcsname\csname #1\endcsname
  \expandafter\let\csname oldend#1\expandafter\endcsname\csname end#1\endcsname
  \renewenvironment{#1}%
     {\linenomath\csname old#1\endcsname}%
     {\csname oldend#1\endcsname\endlinenomath}}%
\newcommand*\patchBothAmsMathEnvironmentsForLineno[1]{%
  \patchAmsMathEnvironmentForLineno{#1}%
  \patchAmsMathEnvironmentForLineno{#1*}}%
\AtBeginDocument{%
\patchBothAmsMathEnvironmentsForLineno{equation}%
\patchBothAmsMathEnvironmentsForLineno{align}%
\patchBothAmsMathEnvironmentsForLineno{flalign}%
\patchBothAmsMathEnvironmentsForLineno{alignat}%
\patchBothAmsMathEnvironmentsForLineno{gather}%
\patchBothAmsMathEnvironmentsForLineno{multline}%
}

\title{
MVNN: A Measure-Valued Neural Network for Learning McKean-Vlasov Dynamics from Particle Data
}

\author{
  Liyao Lyu\thanks{Equal contribution.}\\
  Department of Mathematics \\
  University of California, Los Angeles \\
  Los Angeles, CA 90024, USA\\
  \texttt{lyuliyao@math.ucla.edu} \\
  \AND
   Xinyue Yu\footnotemark[1] \\
  Department of Mathematics \\
  University of California, Los Angeles \\
  Los Angeles, CA 90024, USA\\
  \texttt{tracy@math.ucla.edu} \\
  \AND
   Hayden Schaeffer  \\
  Department of Mathematics \\
  University of California, Los Angeles \\
  Los Angeles, CA 90024, USA\\
  \texttt{hayden@math.ucla.edu} \\
}

\begin{document}
\maketitle
\setcounter{footnote}{0}

\begin{abstract}
Collective behaviors that emerge from interactions are fundamental to numerous biological systems. To learn such interacting forces from observations, we introduce a measure-valued neural network that infers measure-dependent interaction (drift) terms directly from particle-trajectory observations. The proposed architecture generalizes standard neural networks to operate on probability measures by learning cylindrical features, using an embedding network that produces scalable distribution-to-vector representations. On the theory side, we establish well-posedness of the resulting dynamics and prove propagation-of-chaos for the associated interacting-particle system. We further show universal approximation and quantitative approximation rates under a low-dimensional measure-dependence assumption. Numerical experiments on first- and second-order systems, including deterministic and stochastic Motsch-Tadmor dynamics, two-dimensional attraction-repulsion aggregation, Cucker-Smale dynamics, and a hierarchical multi-group system, demonstrate accurate prediction and generalization to initial distributions not encountered during training.
\end{abstract}


\section{Introduction}
Interacting particle and agent systems are widely used to model collective dynamics in physics, biology, and the social sciences~\citep{carrillo2017review, albi2019vehicular, kolokolnikov2011stability, vicsek2012collective, shoham2008multiagent}. Since first-principles approaches to modeling interacting forces are often challenging, a key problem is to infer the underlying interaction mechanisms from trajectory observations. Recent data-driven methods address this by assuming a binary-interaction ansatz and estimating a pairwise interaction kernel via least-squares fitting of observed (finite-difference) velocities as a function of pairwise distances~\citep{lu2019nonparametric, liu2023random, lang2022learning, gu2023data, chu2024inference, maggioni2021learning}. However, the pairwise assumption is often insufficient for complex, realistic systems. In many biological, social, and physical scenarios, the effective dynamics are governed by a drift of the mean field form, a nonlinear functional of the population distribution, rather than a superposition of two-body forces. Popular examples include crowd dynamics governed by local density constraints rather than individual repulsion~\citep{maury2010macroscopic}, vehicular traffic where speed is a nonlinear functional of the flow density~\citep{lighthill1955kinematic}, and cell migration influenced by chemical concentration fields~\citep{keller1970initiation}. Restricting the hypothesis class to pairwise interaction kernels can therefore limit the ability to capture emergent collective behaviors.
Beyond modeling fidelity, computational efficiency is another major difficulty since direct simulation of interacting particle systems typically scales as \(O(N^2)\) in the number of particles. Mean-field limits and related numerical techniques, such as the random batch method~\citep{jin2020random,gan2025random}, provide scalable approximations for large populations~\citep{kadanoff2009more,golse2003mean,spohn2012large} and have been used in aggregation and swarming, Motsch-Tadmor dynamics, and optimization~\citep{bresch2023mean,carrillo2010particle,carrillo2010asymptotic,albi2014boltzmann,motsch2014heterophilious,carrillo2018analytical,carrillo2021consensus,huang2022mean,lyu2025consensus}. This motivates the development of drift models whose evaluation scales linearly with \(N\), by utilizing the symmetry arising from agent indistinguishability.

In recent years, machine learning and optimization-based methods have attracted growing attention for discovering and modeling governing equations of ODEs~\citep{schmidt2009distilling, chen2018neural,brunton2016discovering, schaeffer2017sparse, 
schaeffer2018extracting}, SDEs~\citep{xie2024ab,lyu2023construction,ge2024data,chen2024learning,liu2025training,yeo2025model}, and PDEs~\citep{zheng2025sindy,lu2021learning,lifourier,schaeffer2013sparse,long2018pde,churchill2023dnn, sun2020neupde, schaeffer2017learning, schaeffer2020extracting, messenger2021weak}; see also~\cite{chen2025due}. Existing approaches to inference for interacting particle systems differ
substantially in the structure assumed to be known. When the dynamics are
specified up to finitely many parameters, likelihood-based estimators
\citep{sharrock2023online,jasra2025bayesian}, neural calibration methods
\citep{gaskin2023neural}, and invariant-moment estimators
\citep{pavliotis2024method} provide complementary tools for parameter
inference. Nonparametric alternatives include adaptive kernel estimation of
the evolving density and drift
\citep{della2022nonparametric} and weak-form sparse discovery of
mean-field equations from a prescribed candidate library
\citep{messenger2022learning}. MVNN addresses a different target, specifically, it learns
the nonlinear map \(\bb(\bx,\mu)\) directly from particle trajectories without
prescribing a finite-dimensional parameterization, a pairwise interaction
form, or a candidate PDE library. This flexibility comes with different
assumptions, where the present framework uses fully observed particle states and a
known constant or vanishing diffusion coefficient and does not provide
parameter posteriors or minimax statistical recovery guarantees.
Extending these techniques to mean-field limits remains challenging because the drift is an operator on the space of probability measures rather than a function of finite-dimensional states. Neural operators such as DeepONet~\citep{lu2021learning} and FNO~\citep{lifourier} provide powerful tools for learning mappings between function spaces. More broadly, recent developments include graph neural operators~\citep{li2020multipole}, kernel-based operator learning frameworks~\citep{batlle2024kernel, liao2025cauchy, yu2026regularized}, discretization-free operators~\citep{zhang2023belnet}, and attention-based operator models~\citep{sun2025towards}, further advancing operator learning in infinite-dimensional settings. However, these approaches are typically designed for function-valued inputs defined on structured domains, and are not inherently suited to the Wasserstein setting considered here, where the input is a discrete empirical measure represented as an unordered point cloud.
Emerging PDE foundation models have been proposed to learn unified latent representations across multiple families of PDE datasets in a single, shared framework \citep{zhu2025pi, weihs2025deep, hmida2026compno, negrini2025multimodal, ye2024pdeformer, liu2024prose, cao2024vicon, morel2025disco, yang2023context, jollie2024time, liu2024prosefd, sun2025towards, zhou2024unisolver,ye2025pdeformer}. By capturing common structure across a wide range of systems, these models enable improved generalization to unseen regimes and more robust transfer across tasks. While these methods are applicable to a wide range of PDEs \citep{sun2025towards}, their utilization for particle dynamics is limited. To the best of our knowledge, there are currently no end-to-end frameworks that can infer such measure-dependent drifts directly from particle trajectories while providing theoretical guarantees. To address these challenges, we propose a data-driven framework to learn the measure-dependent drift in mean-field dynamics directly from particle observation.

Unlike ordinary or stochastic differential equations, where the evolution of each particle depends explicitly on its relative positions or pairwise distances, the evolution of particles (or agents) in the mean-field limit depends only on the distribution of the system. The resulting dynamics take the form of a McKean–Vlasov stochastic differential equation for particles $(\bX_t)$:
\[
    \intd \bX_t = \bb(\bX_t,f_t) \intd t + \sigma(\bX_t,f_t)\intd \mathbf{B}_t, 
\]
where $f_t = \mathcal L(\bX_t)$ describes the law of $\bX_t$ and $\mathbf{B}_t$ is a $d\mhyphen$dimensional Wiener process.
Numerically, the McKean–Vlasov dynamics can be approximated by a system of interacting particles 
 $(\bX_t^{i,N})_{1\leq i\leq N}$ with a sufficiently large number of particles $N$. Formally, the initial conditions \(
(\bX_0^{i,N})_{1\leq i \leq N} 
\) are independent and identically distributed with law $\mu_0$ and each particle evolves according to: \[
\intd \bX^{i,N}_t = \bb(\bX^{i,N}_t,\mu^N_t)\intd t + \sigma(\bX^{i,N}_t,\mu^N_t)\intd \mathbf{B}^{i,N}_t,
\]
where $\mathbf{B}_t^{i,N}$ are  $d\mhyphen$dimensional Wiener processes, \(\mu_t^N = \frac{1}{N} \sum_{j=1}^N \delta_{\bX^{j,N}_t}\) denotes the empirical measure of the particle system, and $\delta$ is the Dirac measure.
For simplicity, we only focus on the drift term $\bb:\mathbb R ^d \times \mathcal P(\mathbb R^d) \to \mathbb R^d$ in this work, where $\mathcal P (\mathbb R^d)$ is the space of probability measures on $\mathbb R^d$. The diffusion term $\sigma$ is treated as a known constant or omitted in the present work. Its extension to state- and measure-dependent cases will be explored in future work.

To represent such measure-dependent functions, we propose a \emph{measure-valued neural network} (MVNN), which extends conventional neural networks to take probability measures as inputs in a permutation-invariant manner. The architecture is motivated by the cylindrical functional framework~\citep{guo2023ito,Protter2005}, which characterizes functionals on spaces of probability measures via finitely many test-function integrals. We generalize this idea by learning the test functions and their subsequent interaction map with neural networks, yielding a scalable model that preserves permutation invariance and naturally extends to multi-group interactions.

We summarize our contributions as follows:
\begin{itemize}
    \item  
    We propose MVNN, a permutation-invariant architecture for learning measure-dependent drifts $\bb(\bx,\mu)$ directly from particle-trajectory data. In this setting, permutation-invariant means that the representation depends only on the empirical measure, i.e., on the unordered collection of particle states, so reordering the particles does not change the aggregated feature or the resulting drift.
    \item We establish well-posedness of the neural network induced McKean-Vlasov dynamics and prove propagation-of-chaos guarantees for the corresponding learned particle system.
    \item We prove a universal approximation result for MVNNs on $\mathbb R^d \times \mathcal P_2(\mathbb R^d)$ and establish approximation rates under a low-dimensional (order-parameter) measure-dependence assumption.
    \item 
    
    We demonstrate accurate forward simulations on first- and second-order
    systems. The accurate predictions for initial-distribution geometries not encountered
    during training provide empirical evidence that the trained model has learned
    the underlying measure-dependent dynamics rather than memorizing the training configurations.
\end{itemize}

\section{Learning McKean–Vlasov Drifts from Particle Data}
We consider a standard interacting particle/agent system as a motivating example. Note that the learning framework developed here does not rely on the pairwise interaction kernel and can be applied to more complex interactions. Consider $N$ agents with states
$\bX_t^i\in\mathbb R^d$ and the (possibly stochastic) dynamics:
\begin{equation}\label{eq:micro_sde}
\mathrm d \bX_t^i
= \frac{1}{N}\sum_{j=1}^N \phi\!\left(\|\bX_t^j-\bX_t^i\|\right)\big(\bX_t^j-\bX_t^i\big)\,\mathrm dt
+ \sigma\, \mathrm d\mathbf B_t^i,\qquad i=1,\ldots,N,
\end{equation}
where $\|\cdot\|$ is the Euclidean norm, $\phi:\mathbb R_+\to\mathbb R$ is the pairwise interaction kernel,
and $\{\mathbf B_t^i\}_{i=1}^N$ are independent $d$-dimensional Wiener processes (the deterministic case corresponds to $\sigma=0$).
Let $\bX_t=(\bX_t^1,\ldots,\bX_t^N)$ and denote the empirical measure by:
\[
\mu_t^N := \frac{1}{N}\sum_{j=1}^N \delta_{\bX_t^j}.
\]
In the mean-field limit, such interactions can be represented through a drift $\bb(\bx,\mu)$ acting on a single state $\bx$
and the population distribution $\mu$ ~\citep{sznitman2006topics}. For instance, following Section 2.2.1 in ~\cite{chaintron2021propagation}, for the pairwise model in \eqref{eq:micro_sde}, the corresponding drift is
\[
\bb(\bx,\mu)=\int \phi(\|\by-\bx\|)(\by-\bx)\,\mu(\intd \by).
\]
In order to infer the drift, rather than deriving it from explicit particle-level interactions, we instead rely on observations of the agents, which are typically more accessible in practical settings.
We observe $M$ independent trajectories, at discrete time instances $0 = t_0 < t_1 < \cdots < t_L = T$:
\[
\bX_{\mathrm{tr}} := \{\bX_{t_\ell,m}^{i}\}_{i=1,\ldots,N;\,\ell=0,\ldots,L;\,m=1,\ldots,M},
\]
and calculate the corresponding velocities by first-order finite differences:
\[
\bV_{t_\ell,m}^{i} := \frac{\bX_{t_{\ell+1},m}^{i}-\bX_{t_\ell,m}^{i}}{t_{\ell+1}-t_\ell},
\qquad \ell = 0,\ldots,L-1.
\]
Each trajectory, indexed by \(m\), is initialized by independent and identically distributed samples drawn from an (a priori unknown) initial distribution \(\mu_{0,m}\).
Our goal is to infer a measure-dependent drift
\(
\bb:\mathbb R ^d \times \mathcal P_2(\mathbb R^d) \to \mathbb R^d
\)
directly from trajectory data, where $\mathcal P_2(\mathbb R^d)$ denotes the set of probability measures with finite
second moment.

\subsection{Measure-Valued Neural Network}
One way to represent functionals on Wasserstein space is via cylindrical dependence, i.e., a functional depending on a measure only through finitely many test-function integrals. Concretely, a cylindrical functional has the following form:
\[
\mu \ \longmapsto\  f\bigl(\langle g_1,\mu\rangle,\ldots,\langle g_n,\mu\rangle\bigr),
\]
for some test functions $g_i$ and a finite-dimensional map $f$ (e.g.\ polynomial/smooth); see, e.g.,~\cite{guo2023ito,Protter2005}.
On compact subsets of Wasserstein space, such cylindrical classes are dense in appropriate smoothness spaces; see Lemma~3.12 and
Definition~2.4 in~\cite{guo2023ito} for a representative result.
The key computational advantage is that $\langle g,\mu_t^N\rangle$ can be evaluated by averaging over particles, which is permutation invariant and scales linearly in $N$. Motivated by this formulation, we approximate the drift by a composition of two neural networks:
(1) an embedding network and (2) an interaction network. We define the MVNN drift as
\begin{equation}\label{eq:mvnn-drift}
\bb_\theta(\bx, \mu)
:=
\varphi_\mathrm{int}\!\left(\bx,\ \big\langle \varphi_\mathrm{emb}(\cdot; \theta_\mathrm{emb}),\, \mu \big\rangle;\ \theta_\mathrm{int}\right),
\end{equation}
where $\mu \in \mathcal{P}_2(\mathbb{R}^d)$,
\(
\varphi_\mathrm{emb}(\cdot; \theta_\mathrm{emb}): \mathbb{R}^d \to \mathbb{R}^k
\)
extracts feature representations, and
\(
\varphi_\mathrm{int}(\cdot,\cdot;\theta_\mathrm{int}): \mathbb{R}^d \times \mathbb{R}^k \to \mathbb{R}^d
\)
maps the local state and the global (measure) feature to the drift.
Here $\langle\cdot,\cdot\rangle$ denotes integration with respect to $\mu$:
\[
\langle \varphi_{\mathrm{emb}},\mu\rangle
:= \int_{\mathbb{R}^d} \varphi_{\mathrm{emb}}\left(\bx;\theta_{\mathrm{emb}}\right)\,\mu(\mathrm d \bx).
\]
When $\mu=\mu_t^N=\frac{1}{N} \sum_{j=1}^N \delta_{\bX^{j,N}_t}$, we have:
\[
\langle \varphi_{\mathrm{emb}},\mu_t^N\rangle
= 
\frac{1}{N}\sum_{j=1}^N \varphi_{\mathrm{emb}}\left(\bX^{j,N}_t;\theta_{\mathrm{emb}}\right),
\]
and thus:
\begin{equation}\label{eq:mvnn-empirical}
\bb_\theta(\bX,\mu_t^N)
=
\varphi_{\mathrm{int}}
\left(
\bX,\;
\frac{1}{N}\sum_{j=1}^N \varphi_{\mathrm{emb}}\left(\bX^{j,N}_t;\theta_{\mathrm{emb}}\right);\,
\theta_{\mathrm{int}}
\right).
\end{equation}
Evaluating \eqref{eq:mvnn-empirical} requires a single forward pass of $\varphi_{\mathrm{emb}}$ per particle and a single forward pass of $\varphi_{\mathrm{int}}$ per particle. Consequently, the computational complexity of evaluating the learned drift scales linearly with the number of particles, i.e., $O(N)$. This is in contrast to explicit pairwise interaction models, which incur a quadratic computational cost of order $O(N^2)$.

The model jointly learns (i) a finite collection of test functions
$\varphi_{\mathrm{emb}}$ that extract informative features from the measure, and (ii) an interaction map
$\varphi_{\mathrm{int}}$ that aggregates these features with the local state to approximate the drift. 
The embedding part can be interpreted as learning an optimal, data-driven test function basis on the space of measures. The embedding network $\varphi_\mathrm{emb}$ plays the role of the functions $g_i$ in the cylindrical
functional representation, but unlike polynomial or finite-element bases, it is learned from data instead
of fixed a priori. Given the MVNN drift \eqref{eq:mvnn-drift}, the associated particle system with $N$ agents is governed by:
\begin{equation}\label{equ:learned_dynamics}\begin{aligned}
\intd \bX^{\theta,i,N}_t = \bb_\theta\left(\bX^{\theta,i,N}_t,\mu^{\theta,N}_t\right)\intd t + \sigma\intd \mathbf{B}^{i,N}_t,
\end{aligned} 
\end{equation}
where \(\mu^{\theta,N}_t = \frac{1}{N} \sum_{i=1}^N \delta_{\bX^{\theta,i,N}_t}\) is the empirical law of the interacting particles.
Formally, the limiting McKean-Vlasov dynamics for a single representative particle $\bX^\theta_t$ is given by:
\begin{equation}\label{equ:learned_mean_field}
      \begin{aligned}
    \intd \bX^\theta_t =& \bb_\theta\left(\bX^\theta_t, f^\theta_t\right) \intd t + \sigma \intd \mathbf{B}_t
    \\
    = & \varphi_\mathrm{int}\left(\bX^\theta_t,\int  \varphi_\mathrm{emb}(\bx ) f^\theta_t(\bx)\intd\bx\right)\intd t + \sigma \intd \mathbf{B}_t,
\end{aligned} 
\end{equation}
where $f^\theta_t = \mathrm{Law}(\bX^\theta_t)$ is the law of the random variable $\bX^\theta_t$ at time $t$. The corresponding
Fokker-Planck equation for $(f^\theta_t)_{t\ge0}$ in weak form is
\begin{equation}\label{equ:weak_form_fk}
\begin{aligned}
   \frac{\intd }{\intd t} \left\langle f^\theta_t,\psi \right\rangle =& 
    \left\langle f^\theta_t, \bb_\theta(\bx,f_t^\theta) \cdot \nabla \psi + \frac{1}{2}\sigma^2 \Delta \psi\right\rangle
\\
=& \left\langle f^\theta_t, \varphi_\mathrm{int}\left(\bx,\int  \varphi_\mathrm{emb}(\bx ) f^\theta_t(\bx)\intd\bx\right)\cdot \nabla \psi + \frac{1}{2}\sigma^2 \Delta \psi\right\rangle, 
\end{aligned}
\end{equation}
for all smooth test functions $\psi:\mathbb{R}^d\to\mathbb{R}$ with compact support~\citep{mckean1967propagation}.
\begin{proposition}[Well-Posedness of MVNN-Induced McKean--Vlasov Dynamics]\label{prop:well-possness}
    Assume that $\varphi_{\mathrm{int}}$ and $\varphi_{\mathrm{emb}}$ are globally Lipschitz: there exist $C_i,C_e>0$, such that for all $\bx,\bx'\in \mathbb R^d , \by,\by'\in \mathbb R^{k}$, it holds that:
    \[
    \begin{aligned}
        \left\|\varphi_{\mathrm{emb}}(\bx)-\varphi_{\mathrm{emb}}(\bx')\right\|&\leq C_e\left\|\bx-\bx'\right\|,\\ 
        \left\|\varphi_{\mathrm{int}}(\bx,\by)-\varphi_{\mathrm{int}}(\bx',\by')\right\|&\leq C_i\left(\|\bx-\bx'\|+\left\|\by-\by'\right\|\right).
    \end{aligned}
    \]
    Assume that $f_0\in \mathcal P_2(\mathbb R^d)$. Then for any $T>0$, the SDE \eqref{equ:learned_mean_field} has a unique strong solution on $[0,T]$ and consequently, its law is the unique weak solution to the Fokker-Planck equation \eqref{equ:weak_form_fk}.
\end{proposition} 
The proof of Proposition~\ref{prop:well-possness} is provided in Appendix~\ref{app:proof-wellposedness}.

Although Proposition~\ref{prop:well-possness} is a standard consequence of Lipschitz continuity for McKean-Vlasov SDEs,
it guarantees that our learned MVNN drift induces a well-posed nonlinear Fokker-Planck evolution at the continuum level. This establishes our method as a rigorous framework for learning continuous governing equations directly from discretized particle-level observations, without requiring any smoothing or filtering procedures, in contrast to existing approaches \citep{messenger2022learning,cohen2024physics}.
However, to complete the theoretical justification, we must demonstrate that the learned $N$-particle system \eqref{equ:learned_dynamics} converges to this identified mean-field model \eqref{equ:learned_mean_field} as $N \to \infty$.

\begin{proposition}[Mean-Field Convergence and Propagation of Chaos for the Learned Particle System]\label{prop:propogation_of_chaos}
    Let the assumptions of Proposition \ref{prop:well-possness} hold. Let $(\bX_t^{\theta,i,N})_{i=1}^N$ be the $N$-particle system solving \eqref{equ:learned_dynamics} with $f_0$-chaotic initial data $\bX_0^{\theta,i,N} \sim f_0^\theta$. Let $f^\theta_t$ be the unique solution to the mean-field Fokker-Planck equation \eqref{equ:weak_form_fk} with initial condition $f_0^\theta = f_0$. Then, the $N$-particle system \eqref{equ:learned_dynamics} converges to the mean-field model \eqref{equ:learned_mean_field} as $N \to \infty$. That is, for any $T > 0$, the $N$-particle distribution $f_t^{\theta,N} = \text{Law}(\bX_t^{\theta,1,N}, \dots, \bX_t^{\theta,N,N})$ is $f_t^\theta$-chaotic, satisfying:\[
    \lim_{N \to \infty}W_2\left(f^{1,\theta,N}_{[0,T]},f^\theta_{[0,T]}\right) =0,
\]
where $f_{[0,T]}^{1,\theta,N}$ is the first marginal of $f_{[0,T]}^{\theta,N}$.
\end{proposition}
The proof of Proposition~\ref{prop:propogation_of_chaos} is provided in Appendix~\ref{app:proof-chaos}.

By establishing the propagation of chaos, Proposition~\ref{prop:propogation_of_chaos} rigorously bridges the gap between the microscopic and macroscopic descriptions. It confirms that the learned particle dynamics systematically converge to the proposed mean-field model in the large-scale limit. The network structure \eqref{eq:mvnn-drift} is closely related to the mean-field neural network architecture studied in
\cite{pham2023mean} for optimal control on Wasserstein space. In particular, a universal approximation theorem
shows that cylindrical neural functionals of the form \eqref{eq:mvnn-drift} can approximate any continuous drift
$\bb^\star:\mathbb{R}^d\times\mathcal{P}_2(\mathbb{R}^d)\to\mathbb{R}^d$ with arbitrarily small error.
\begin{theorem}[Universal Approximation for Measure-Valued Neural Network]\label{thm:uap_mvnn}
    
     Let $\zeta$ be a probability measure on $\mathcal P_2(\mathbb R^d)$, and $\bb^\star$ be a continuous map from $\mathbb R ^d \times \mathcal P_2(\mathbb R^d)\to \mathbb R^d$, such that $\|\bb^\star\|^2_{ L^2(\zeta)}:= \int_{\mathcal P_2 (\mathbb R^d )} \int_{\mathbb R^d} \left\|\bb^\star(\bx,\mu)\right\|^2
      \mu(\intd \bx)\zeta(\intd \mu)< \infty$. Then for all $\epsilon>0$, there exist $k\in \mathbb N$, $\varphi_\mathrm{int}$ a neural network, mapping from $\mathbb R^d \times \mathbb R^k $ into  $\mathbb R^d$ and $\varphi_\mathrm{emb}$ another neural network, mapping from $\mathbb R^d$  into $\mathbb R^k$, such that 
    \[
    \left\|\bb^\star(\cdot, \cdot ) - \varphi_\mathrm{int}(\cdot,\langle\varphi_\mathrm{emb},\cdot\rangle) \right\|_{L^2(\zeta)} \leq \epsilon.
    \]
\end{theorem}

The proof follows directly from Theorem~2.2 in~\cite{pham2023mean}, 
which establishes the universal approximation property for functions 
on measure spaces of the form $\mathbb{R}^d \times \mathcal{P}_2(\mathbb{R}^d)$
using neural network parameterizations of cylindrical form. This theorem is a direct specialization of that result to the drift setting $p=d$, with the notational correspondence $V\leftrightarrow \bb^\star$, $\Psi\leftrightarrow \varphi_{\mathrm{int}}$, and $\phi\leftrightarrow \varphi_{\mathrm{emb}}$.

While Theorem \ref{thm:uap_mvnn} ensures the theoretical capability of MVNN, it does not address the efficiency of the approximation. In fact, without further structural assumptions, the number of parameters required may grow rapidly with dimension.  Following \cite{liu2024neural, weihs2025deep}, we define a ReLU network $q:\mathbb R^{d_1}\to \mathbb R$ as:
\begin{equation}\label{equ:1drelunn}
q(\bx) = W_L\cdot \mathrm{ReLU}(W_{L-1}\cdots\mathrm{ReLU}(W_{1}\bx+b_1)+\cdots +b_{L-1})+b_L, 
\end{equation}
where $W_l$ are weight matrices, $b_l$ are bias vectors, $\mathrm{ReLU}(a) = \max\{a,0\}$. We define the network class $\mathcal F_{NN}:\mathbb R^{d_1}\to \mathbb R^{d_2}$:
\[
\begin{aligned}
\mathcal F_{NN} (d_1,d_2,L,p,K,\kappa,R) = 
   \{ &[q_1,q_2,\cdots,q_{d_2}] ^\top:\mathbb R^{d_1}\to \mathbb R^{d_2}: \text{for each }k = 1,\cdots, d_2,\\
    & q_k:\mathbb R^{d_1}\to \mathbb R \text{ is in the form of \eqref{equ:1drelunn} with $L$ layers, width bounded by $p$} \\
    & \|q_k\|_{L^\infty}\leq R , \|W_l\|_{\infty,\infty}\leq \kappa, \|b_l\|_\infty\leq \kappa, \sum_{l=1}^L \|W_l\|_0 + \|b_l\|_0 \leq K, \forall l\},
\end{aligned}
\]
where $\|q\|_{L^\infty(\Omega)} = \sup_{\bx\in \Omega} |q(\bx)|$, $\|W_l\|_{\infty,\infty} = \max_{ij}|(W_{l})_{i,j}|$, $\|b\|_\infty = \max_{i}|b_i|$ and $\|\cdot\|_0$ denotes the number of nonzero elements of its argument.

\begin{theorem}[Quantitative Approximation via Finite-Dimensional Measure Embeddings]\label{thm:scaleing_law}
    Assume that the empirical measure $\mu\in U \subset \mathcal P_2(\Omega)$ is supported in $\Omega=[-\gamma_1,\gamma_1]^d\subset \mathbb R^d$ for some $\gamma_1>0$. We further assume that the map $f:\Omega \times U \to \mathbb R $ is a Lipschitz continuous map in the sense that for any $\mu,\nu\in U $ and $\bx,\by\in \Omega$, there exists a constant $L_f$ such that 
    \[
    |f(\bx,\mu) -f(\by,\nu) | \leq L_f \left(\|\bx-\by\| +W_1(\mu,\nu)\right).
    \]
    For any $\epsilon>0$, there exists a constant $C$ depending on $\gamma_1 $ and $L_f$, and $\{\bc_m\}_{m=1}^{C_\Omega}\subset\Omega$ so that $\{B_\delta (\bc_m)\}_{m=1}^{C_\Omega}$ is a cover of~$\Omega$ for $\delta = \frac{\epsilon}{2L_f}  $  and some $C_\Omega = \epsilon^{-d}$. Let $H=C(C_\Omega+d)^{\frac{C_\Omega+d}{2}}(C_\Omega d)^{\frac{C_\Omega+d}{2}}\epsilon^{-C_\Omega-d}$ and there are $H$ ReLU neural networks $\{q_k\}_{k=1}^H$, where $q_k\in \mathcal F_{NN}(d+C_\Omega,1,L,p,K,\kappa,R)  $ with \[
\begin{aligned}
L  & = O((d+C_\Omega)^2 \log(\epsilon^{-1}) + (d+C_\Omega)^2 \sqrt{{C_\Omega}d}L_f + (d+C_\Omega)^2 \log(d+C_\Omega)) &
p &= O(1)\\
K &= O((d+C_\Omega)^2 \log(\epsilon^{-1}) + (d+C_\Omega)^2\sqrt{{C_\Omega}d}L_f + (d+C_\Omega) ^2 \log(d+C_\Omega) )\\
\kappa &=  O\left((d+C_\Omega)^{\frac{d+C_\Omega}{2}+1}\epsilon^{-d-C_\Omega-1}L_U^{d+C_\Omega}\right) &
R & = O(1),
\end{aligned}
\] such that 
    \[
    \sup_{\mu\in U ,\bx\in \Omega }\left|f(\bx,\mu)- \sum_{k=1}^H a_k q_k(\bx,\bu) \right|\leq \epsilon,
    \]
    where 
    \(\bu=\left(\left\langle \omega_1,\mu\right\rangle, \cdots , \left\langle \omega_{C_\Omega},\mu\right\rangle\right)
    \), $\{\omega_m(\bx)\}_{m=1}^{C_\Omega}$ is a partition of unity subordinate to the cover \(\{B_\delta (\bc_m)\}_{m=1}^{C_\Omega}\), $a_k$ depends on $f$, and $L_U = 2R\sqrt{C_\Omega d}\,L_f$.
\end{theorem}
The proof of Theorem~\ref{thm:scaleing_law} is provided in Appendix~\ref{app:proof-scaling}.

Theorem \ref{thm:scaleing_law} provides a constructive finite-dimensional approximation result of the same form as MVNN. The partition-of-unity features 
$\bigl(\langle \omega_m, \mu \rangle\bigr)_{m=1}^{C_\Omega}$
provide a fixed embedding of the measure, whereas MVNN replaces them with a learned embedding 
$\bu = \langle \varphi_{\mathrm{emb}}, \mu \rangle$, and the networks $q_k$ correspond to the interaction network $\varphi_{\mathrm{int}}$. For example, when the output dimension is one, the terms are $a_k q_k(\bx,\bu) = \varphi_\mathrm{int}(\bx,\langle\varphi_\mathrm{emb},\mu\rangle)$. Theorem \ref{thm:scaleing_law} highlights a fundamental bottleneck: without structural constraints, the complexity of approximating functionals on the Wasserstein space leads to the curse-of-dimensionality, where the required network size grows quickly with the dimension $d$. This implies that a naive application of MVNNs would be computationally intractable for high-dimensional systems. However, in physical, biological, and social contexts, the effective mean-field interactions are unlikely to be arbitrary functionals of the full probability measure. Instead, collective behaviors, such as flocking, synchronization, or aggregation, are typically governed by a compact set of macroscopic summary statistics, often referred to as order parameters in statistical physics (e.g., local density, mean momentum, or polarity). These systems are embedded in a low-dimensional manifold within the infinite-dimensional space of measures. To formalize and utilize this intrinsic physical structure, we introduce the following assumption, which bridges the gap between the theoretical worst-case complexity and the practical efficiency observed in our experiments.
\begin{assumption}[Finite-Dimensional Measure Dependency]\label{assum:finite-dimensional}
Let $U \subset \mathcal{P}_2(\mathbb{R}^d)$ be the space of measures under consideration. We assume there exists a fixed, finite set of $r$ feature functions $\mathcal{F} = \{f_i: \mathbb{R}^d \to \mathbb{R}\}_{i=1}^r$ that fully characterizes the dependencies for all relevant functionals. Specifically, for any functional $V: U \to \mathbb{R}$, there exists a corresponding function $G: \mathbb{R}^r \to \mathbb{R}$ such that:
\[
V(\mu) = G(\langle f_1, \mu \rangle, \dots, \langle f_r, \mu \rangle) \quad \forall \mu \in U.
\]
\end{assumption}

\begin{remark}          Assumption \ref{assum:finite-dimensional} effectively posits that the measures of interest $U$ lie on a finite-dimensional manifold embedded within $\mathcal{P}_2(\mathbb{R}^d)$. 
    An example is the family of Gaussian distributions $\mathcal{N}(\nu, \Sigma)$, which are uniquely determined by their mean and covariance (moments). 
    Thus, any functional $V$ defined only on this family of measures can be written as $V(\mu) = G(\nu, \Sigma)$, which perfectly matches our assumed form with a finite $r = d + d(d+1)/2$.
\end{remark}
\begin{remark}\label{remark:G}
    A direct consequence of Assumption \ref{assum:finite-dimensional} is that our target drift function $\bb^\star(\bX, \mu)$, which depends on both the state $\bX$ and the measure $\mu$, must also admit a finite-dimensional representation.
    This follows by treating the state $\bX$ as a fixed parameter. For any given $\bX \in \mathbb{R}^d$, the mapping $\bb^\star(\bX, \cdot): U \to \mathbb{R}^d$ is a pure functional of $\mu$. Assumption~\ref{assum:finite-dimensional} then guarantees the existence of a corresponding function $G_\bX: \mathbb{R}^r \to \mathbb{R}^d$ such that:
    \[
    \bb^\star(\bX, \mu) = G_\bX(\langle f_1, \mu \rangle, \dots, \langle f_r, \mu \rangle).
    \]
    Since this holds for all $\bX$, we can construct a single, global function $G: \mathbb{R}^d \times \mathbb{R}^r \to \mathbb{R}^d$ by defining $G(\bX, \bz) := G_\bX(\bz)$, where $\bz \in \mathbb{R}^r$. 
\end{remark}
\begin{theorem}[Approximation Rate of MVNN under Low-Dimensional Assumption]
    \label{thm:mvnn-approximation}
    Let the state $\bX \in [0,1]^d$ and assume that the measure $\mu$ is supported on $[0,1]^d$.
    Suppose  that Assumption~\ref{assum:finite-dimensional} holds on $U$ for some $r \in \mathbb{N}$ and a collection of feature functions $\bg = (g_1, \dots, g_r)$.
    As established in Remark~\ref{remark:G}, this implies that there exists a function $G: \mathbb{R}^d \times \mathbb{R}^r \to \mathbb{R}^d$ such that
    \[
        \bb^\star(\bX, \mu) = G(\bX, \langle \bg, \mu \rangle).
    \]
    We further assume that the  functions $\bg$ and $G$ are Lipschitz continuous with constants $L_{\bg}$ and $L_G$, respectively.
    
    Then for any $\epsilon > 0$, there exist deep ReLU networks $\varphi_\mathrm{emb}$ and $\varphi_\mathrm{int}$,  such that
    \[
        \sup_{\bX \in [0,1]^d,\, \mu \in U} \| \bb^\star(\bX, \mu) - \varphi_\mathrm{int}(\bX,\langle\varphi_\mathrm{emb},\mu\rangle) \|_{\mathbb{R}^d} \le \epsilon,
    \]
    where the interaction network $\varphi_\mathrm{int}$ has depth \(
    L_{\mathrm{i}} = O(\log(\epsilon^{-1})) 
    \) and width \(
    W_{\mathrm{i}} = O(\epsilon^{-(d+r)} )
    \) and the embedding network $\varphi_\mathrm{emb}$ has depth $ L_{\mathrm{e}} = O(\log(\epsilon^{-1}))$ and width \(
    W_{\mathrm{e}} = O\bigl(\epsilon^{-d}\bigr)
    \).
\end{theorem}
The proof of Theorem~\ref{thm:mvnn-approximation} is provided in Appendix~\ref{app:proof-approximation}.

\subsection{Learning Objective and Optimization}

Given the observation dataset:
\[
\mathcal{D}_{\text{obs}}
= \bigl\{ (\bX_{t_l,m}^i, \bV_{t_l,m}^i) \bigr\}_{i=1,l=0,m=1}^{N,L-1,M},
\]
we approximate the empirical measure at time $t_l$ in trajectory $m$ by
\(
\hat{\mu}_{t_l,m} = \frac{1}{N}\sum_{j=1}^N \delta_{\bX_{t_l,m}^j}.
\)
The mean-field drift predicted by the MVNN is then given by:
\[
\hat{\bb}_\theta(\bX_{t_l,m}^i, \hat{\mu}_{t_l,m})
= \varphi_\mathrm{int}\!\left(\bX_{t_l,m}^i,\,
\frac{1}{N}\sum_{j=1}^N \varphi_\mathrm{emb}(\bX_{t_l,m}^j; \theta_\mathrm{emb});
\theta_\mathrm{int} \right).
\]
The model parameters $\theta = (\theta_\mathrm{int}, \theta_\mathrm{emb})$ are learned by minimizing the discrepancy between the observed and predicted trajectories. 

For deterministic trajectory data, we train the MVNN by minimizing
\[
\mathcal J_{\mathrm{det}}(\theta)
=
\frac{1}{MLN}
\sum_{m=1}^{M}\sum_{l=0}^{L-1}\sum_{i=1}^{N}
\left\|
\bV_{t_l,m}^{i}
-
\hat{\bb}_{\theta}
\left(\bX_{t_l,m}^{i},\hat{\mu}_{t_l,m}\right)
\right\|^2.
\]

For stochastic systems with a nondegenerate diffusion coefficient, the same
objective can also be derived from the discretized path likelihood. In
particular, when $\sigma(\bx)\equiv\sigma \mathbf{I}$, the negative log-likelihood
differs from $\mathcal J_{\mathrm{det}}$ only by the positive factor
$\Delta t/(2\sigma^2)$ and an additive constant, and therefore has the same
minimizer.
Let $\mathbb P^\theta$ denote the path measure induced by the solution
$(\bX^{\theta}_s)_{s\in[0,t]}$ of the McKean-Vlasov SDE \eqref{equ:learned_mean_field}.
Under standard regularity conditions ensuring absolute continuity of path measures,
Girsanov's theorem yields the log-likelihood function (see \cite{wen2016maximum})
\begin{equation}\label{eq:girsanov_LL}
\begin{aligned}
\mathcal L_t(\theta)
:= \log \frac{\intd\mathbb P^\theta}{\intd\mathbb P^{\bb}}
= &\int_0^t
\Big\langle
\hat{\bb}_\theta(\bX_s,\mu_s^\theta)-\bb(\bX_s,\mu_s),\,
(\sigma(\bX_s)\sigma(\bX_s)^\top)^{-1}\, \intd \bX_s
\Big\rangle \\
&-\frac12\int_0^t
\left(
\big\|\sigma(\bX_s)^{-1}\hat{\bb}_\theta(\bX_s,\mu_s^\theta)\big\|^2
-\big\|\sigma(\bX_s)^{-1}\bb(\bX_s,\mu_s)\big\|^2
\right)\,\intd s .
\end{aligned}
\end{equation}

Similar to Section 2.1.2 in \cite{sharrock2023online}, we only observe discrete-time samples of
$M$ trajectories of an $N$-particle system at $t_l=l\Delta t$:
$\{\bX^{i}_{t_l,m}\}_{i=1}^N$ for $m=1,\dots,M$ and $l=0,\dots,L$.
Approximating the stochastic integral by the Euler–Maruyama method,
we obtain the discrete-time likelihood (up to $\theta$-independent constants):
\begin{equation}\label{eq:disc_LL}
\begin{aligned}
\hat{\mathcal L}(\theta)
=&\ \frac{1}{MLN}\sum_{m=1}^M\sum_{l=0}^{L-1}\sum_{i=1}^N
\Bigg[
\Big\langle
\hat{\bb}_\theta(\bX^{i}_{t_l,m},\hat\mu_{t_l,m})-\bb(\bX^{i}_{t_l,m},\hat\mu_{t_l,m}),
(\sigma\sigma^\top)^{-1}\Delta \bX^{i}_{t_l,m}
\Big\rangle \\
&
-\frac12\Big(
\|\sigma^{-1}\hat{\bb}_\theta(\bX^{i}_{t_l,m},\hat\mu_{t_l,m})\|^2
-\|\sigma^{-1}\bb(\bX^{i}_{t_l,m},\hat\mu_{t_l,m})\|^2
\Big)\Delta t
\Bigg],
\end{aligned}
\end{equation}
where $\Delta \bX^{i}_{t_l,m}:=\bX^{i}_{t_{l+1},m}-\bX^{i}_{t_l,m}$.
We further assume $\sigma(\bx)\equiv \sigma I$ with a constant scalar $\sigma>0$.
Then, dropping all terms independent of $\theta$, \eqref{eq:disc_LL} reduces to:
\begin{equation}\label{eq:disc_LL_simplify}
\hat{\mathcal L}(\theta)
= -\frac{\Delta t}{2\sigma^2}\,\frac{1}{MLN}\sum_{m=1}^M\sum_{l=0}^{L-1}\sum_{i=1}^N
\Big\|
\hat{\bb}_\theta(\bX^{i}_{t_l,m},\hat\mu_{t_l,m}) - \bV^{i}_{t_l,m}
\Big\|^2 \;+\; C.
\end{equation}
Therefore, maximizing $\hat{\mathcal L}(\theta)$ is equivalent (up to a positive scaling)
to minimizing the mean-squared regression loss:
\begin{equation}\label{eq:mse_final}
\theta^\star \in \arg\min_{\theta}\ 
\frac{1}{MLN}\sum_{m=1}^M\sum_{l=0}^{L-1}\sum_{i=1}^N
\bigl\|
\bV^{i}_{t_l,m} - \hat{\bb}_\theta(\bX^{i}_{t_l,m},\hat{\mu}_{t_l,m})
\bigr\|^2.
\end{equation}
While our objective function \eqref{eq:mse_final} is derived heuristically from the discretization of the path measure, its theoretical validity in the mean-field limit is supported by recent results on parameter estimation for McKean-Vlasov SDEs \citep{sharrock2023online}. They proved that under standard regularity conditions, the estimator obtained by maximizing the particle likelihood is consistent in the limit as $N \to \infty$ (Theorem 1.1 in \cite{sharrock2023online}). Furthermore, the estimator exhibits asymptotic normality with a convergence rate of $\mathcal{O}(N^{-1/2})$ (Theorem 1.2 in  \cite{sharrock2023online}). 
The likelihood formulation is related to the parameter-estimation framework
of~\cite{sharrock2023online}. Their consistency and asymptotic-normality
results concern a specified finite-dimensional parametric family. These
results motivate likelihood-based estimation for McKean--Vlasov systems, but
they do not directly establish statistical consistency without a priori information.
We optimize \eqref{eq:mse_final} using Adam~\citep{kingma2014adam} with mini-batches to improve computational
efficiency and training stability. All gradients are obtained via automatic differentiation in JAX~\citep{jax2018github}.

\subsection{Numerical Result}
We present numerical experiments that validate the accuracy of the learned model and its ability to generalize to previously unseen initial configurations\footnote{The code is available at: \url{https://github.com/Lyuliyao/mean-field-kernel-code}}. Unless otherwise stated, we compare the learned mean-field model against reference particle simulations.
\subsubsection{1D Motsch-Tadmor Dynamics}
We begin by validating our framework on the Motsch-Tadmor model~\citep{motsch2011new,motsch2014heterophilious}. Unlike standard alignment models with additive forces, the Motsch-Tadmor dynamics feature a normalized interaction mechanism, where the influence of neighbors is weighted by their relative distance and normalized by the total interaction strength. The evolution of $N$ agents with scalar states $X^i_t \in \mathbb R$ is governed by:
\[
\dot X^i_t
= 
\frac{1}
{\sum_{k=1}^N \phi\left(|X^k_t-X^i_t|\right)}\sum_{j=1}^N \phi\left(|X^j_t-X^i_t|\right)
\left(X^j_t-X^i_t\right), 
\qquad i=1,\ldots,N.
\]
Here we set the population size $N = 16{,}000$ and the time horizon $T=2$. The interaction kernel is chosen as a Gaussian function $\phi(r)=\exp\left(-(r/\ell)^2\right)$ with a characteristic length $\ell=0.5$. The normalization term in the denominator introduces a non-trivial dependency on the global empirical measure, making the macroscopic drift strictly non-pairwise and highly nonlinear.
 We discretize the system using a forward Euler scheme with timestep $\Delta t = 1\times 10^{-2}$ to generate $M=100$ independent trajectories $\{X^i_{t_l,m}\}$, $l = 0,\cdots,200$, $m=1,\cdots,100$. For each realization $m$, initial states $X^i_{0,m}$ are sampled from a randomly generated multimodal distribution $\mu^m_0$. Specifically, each $\mu_0^m$ is constructed as a mixture of $2$-$8$ Gaussian components whose mean is uniformly distributed in $[0,3]$ and variance is $0.25$. The mixture weights of different Gaussians are sampled from a symmetric Dirichlet distribution, producing a diverse set of initializations. 
All \(M=100\) trajectories form the training set for the full-data MVNN.
The learned dynamics are then evaluated by autonomously simulating the
corresponding \(N\)-particle system:
\[
\dot X_t^{i,N}
=
\hat{b}_\theta
\left(
X_t^{i,N},\hat{\mu}_t^N
\right),
\qquad i=1,\ldots,N,
\]
where
\[
\hat{\mu}_t^N
=
\frac{1}{N}\sum_{j=1}^N\delta_{X_t^{j,N}}.
\]
For qualitative illustration, we independently generate three additional
reference trajectories whose initial distributions are sampled from the same
Gaussian-mixture family as the training data. These trajectories are not used
for training, validation, or model selection. Starting from the corresponding
initial particle configurations, we compare the autonomous MVNN predictions
with the reference dynamics in Figure~\ref{fig:1d_case1}.

Beyond the representative examples above, we quantitatively examine how MVNN
accuracy depends on the amount of training data. We train additional MVNN
models on nested subsets containing
\[
M\in\{5,10,20,40,60,80,100\}
\]
trajectories while keeping the particle number, observation times, network
architecture, optimization procedure, and evaluation protocol fixed. A
separate validation set is used for early stopping. Each trained model is
evaluated by autonomous rollout on the same 30 independently generated
held-out trajectories, none of which is used for training, early stopping, or
model selection.

Figure~\ref{fig:mvnn_learning_curve} reports the relative density \(L^2\)
error and the normalized drift error on the training and test trajectories.
The density error quantifies the discrepancy between the autonomous MVNN
prediction and the reference population dynamics, whereas the drift error is
evaluated directly on reference state--measure pairs. Both errors decrease
substantially as the number of training trajectories increases and level off
for the larger training sets. The large training--test gap observed at
\(M=5\) is substantially reduced at \(M=80\) and \(M=100\). The
Wasserstein-\(1\) error exhibits the same overall trend.

\begin{figure}[t]
    \centering
    \includegraphics[width=0.95\linewidth]
    {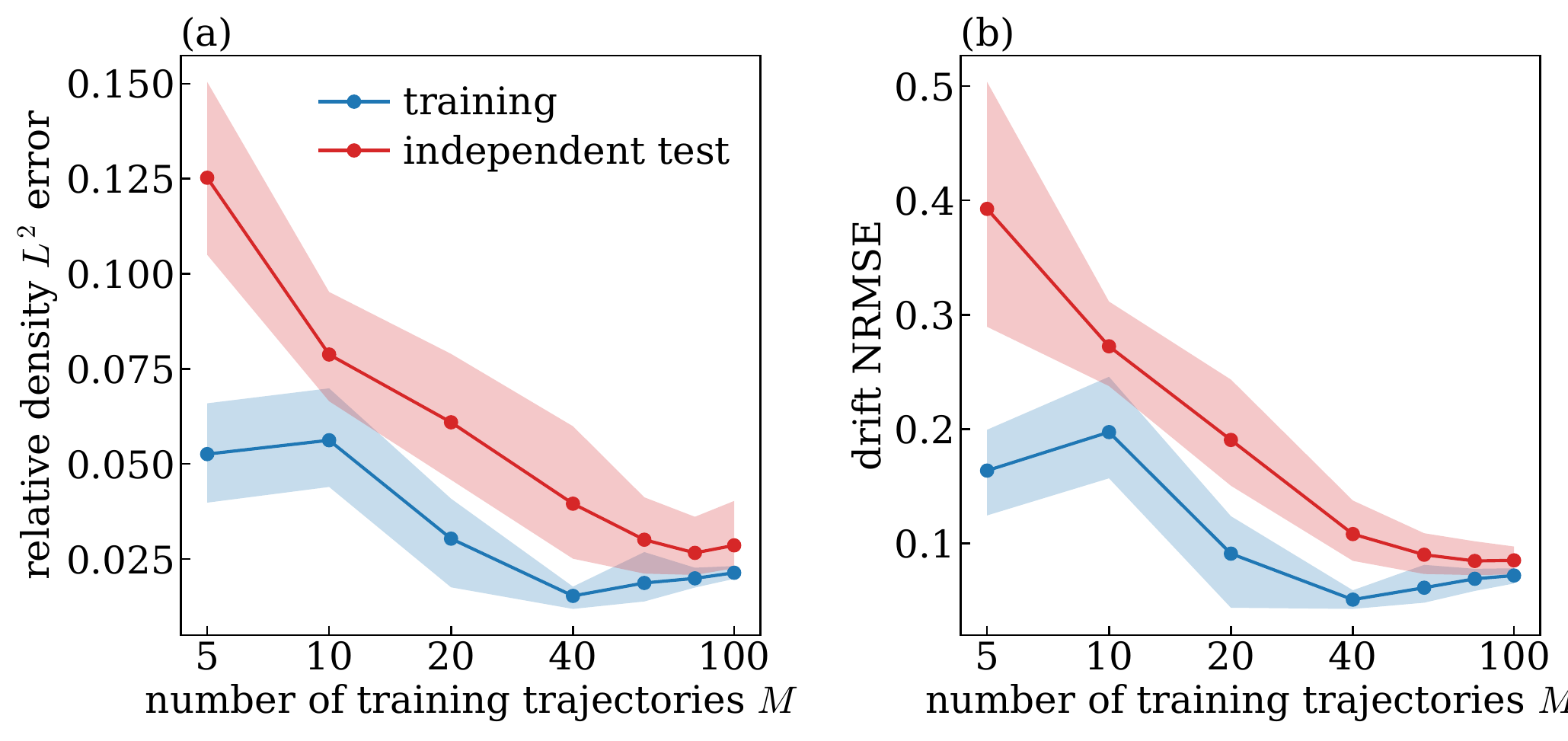}
    \caption{\textbf{Effect of training-data size on MVNN accuracy.}
    (a) Time-averaged relative density \(L^2\) error from autonomous
    rollouts. (b) Normalized error of the learned drift evaluated directly on
    reference state--measure pairs. The test set contains 30 independently
    generated held-out trajectories. Curves show the mean over three
    independently trained models, and the shaded regions show the corresponding
    uncertainty intervals.}
    \label{fig:mvnn_learning_curve}
\end{figure}

Unlike pairwise interaction estimation, which aims to recover the microscopic interaction kernel, our evaluation focuses on the system’s macroscopic behavior. Specifically, we compare the predicted and true population densities over time to assess the quality of the learned mean-field dynamics. Figure~\ref{fig:1d_case1} compares the time evolution of the densities predicted by our model against the ground truth for three unseen initial conditions. The results demonstrate that the MVNN successfully captures the clustering and consensus formation inherent in the dynamics. Despite the complex normalization factor, the learned model accurately reproduces the merging of clusters and the preservation of density peaks without explicit knowledge of the microscopic interaction form.

We compare the learned MVNN dynamics to predictions generated from a Gaussian process model with a Mat\'ern kernel as proposed in~\cite{feng2023learning}. We perform a matched-data comparison between MVNN and the Gaussian process model, with both models trained on $N=16$ agents, $M=9$ independent trajectories, and $L=20$ timesteps of size $\Delta t = 5\times 10^{-2}$ on the interval $[0,1]$. We also examine the effect of increasing the training data scale. For this purpose, we train a second MVNN model using $N= 16{,}000$ agents, $M=100$ trajectories, and $L=200$ timesteps of size $\Delta t = 10^{-2}$ on the interval $[0,2]$. The initial distributions of each trajectory for all three models are sampled from a Gaussian mixture model with $2-8$ components, whose mean is uniformly distributed in $[0,3]$, variance is $0.25$, and mixure weights are sampled from a symmetric Dirichlet distribution. Figure~\ref{fig:mvnn_gp_comparison} compares the time evolution of the densities predicted by the two MVNN models and the Gaussian process model against the ground truth for $N=16{,}000$ agents and three unseen initial conditions. The results show that both MVNN models match the reference dynamics better than the Gaussian process model. The blue and orange curves in Figure~\ref{fig:mvnn_gp_error} report the $L^2$ error between the true density and the predictions using the Gaussian process model and the matched-data MVNN model, respectively.  The matched-data MVNN produces a lower prediction error on the unseen initial conditions. Since the two models are trained on the same data, this difference cannot be attributed to the amount of training data. The green curve in Figure~\ref{fig:mvnn_gp_error} reports the $L^2$ density error of the MVNN model trained on the larger dataset, and it shows that increasing the training scale further reduces the prediction error. Therefore, the comparison between the Gaussian process model and the matched-data MVNN isolates the effect of the model class, whereas the comparison between the two MVNN models demonstrates the additional benefit of increasing the amount of training data. 

The difference between the two methods is related to their respective hypothesis classes. The Gaussian process baseline assumes a fixed additive pairwise interaction law and requires the kernel to be symmetric, whereas the normalization in the Motsch--Tadmor drift depends on the focal agent and the full empirical measure and causes the pairwise weights to become asymmetric. As a result, the Gaussian process model cannot accurately capture the normalization factor. In contrast, MVNN directly models the nonlinear dependence on the empirical measure and is able to accurately learn the normalized interaction kernel.

Training the Gaussian process model on the larger dataset is computationally prohibitive. Letting $n_{\mathrm{obs}}=NML$ denote the total number of particle observations, exact Gaussian process training requires constructing and factorizing a dense $n_{\mathrm{obs}}\times n_{\mathrm{obs}}$ kernel matrix, leading to $\mathcal{O}(n_{\mathrm{obs}}^2)$ memory consumption and $\mathcal{O}(n_{\mathrm{obs}}^3)$ computational cost. Figure~\ref{fig:mvnn_gp_time} reports the average simulation time as a function of the number of agents $N$ for the matched-data MVNN and Gaussian process models over ten trials. Over the tested range, the rollout cost of MVNN increases substantially more slowly than that of the Gaussian process model as $N$ increases. Together with the dense kernel matrix cost, these results show that the Gaussian process baseline becomes computationally prohibitive at large scales in both model training and simulation, limiting it to a smaller number of agents. Conversely, MVNN is able to handle a large number of agents, which may be needed in the mean-field limit.

As an additional comparison, we compare MVNN to mean-field WSINDy~\citep{messenger2022learning} and online parameter estimation~\citep{sharrock2023online}. The mean-field WSINDy formulation approximates the drift as 
\[
\hat{b}(x,\mu)=f_B(x) + \int_\mathbb{R} K_B(y-x)\mu(\diff y),\]
where $f_B$ is the self-drift with no population dependence and $K_B$ is the unknown interaction kernel. Here $f_B$ and $K_B$ are both expanded in a library of basis functions consisting of monomials. The online parameter estimation method consists of estimating parameters of the McKean--Vlasov SDE using continuous-time stochastic gradient ascent with respect to the asymptotic log-likelihood of the interacting particle system. Here we estimate the parameter $\ell$ in the interaction kernel $\phi(r) = \exp(-(r/\ell)^2)$ and pick an initial guess of $\ell_0 = 1.5$. For all the models, we train on $N=16{,}000$ agents, $M=100$ trajectories, and $L=200$ timesteps with $\Delta t = 1\times 10^{-2}$ on the interval $[0,2]$. The initial states of each trajectory are again drawn from a Gaussian mixture model with $2-8$ components, with a mean that is uniformly sampled from $[0,3]$, a variance of $0.25$, and mixture weights that are sampled from a symmetric Dirichlet distribution. Figure~\ref{fig:mvnn_wsindy_param_est} compares the time evolution of the densities predicted by MVNN, mean-field WSINDy, and online parameter estimation against the ground truth for $N=16{,}000$ agents. 
MVNN outperforms both mean-field WSINDy and online parameter estimation, and its prediction best matches the reference dynamics. This is because mean-field WSINDy requires a basis rich enough to fully capture the true dynamics, and online parameter estimation requires both a good initial guess that is sufficiently close to the true parameters and prior knowledge of the mean-field drift function's form. MVNN does not have these constraints; we do not need to know the form of the drift function or pick a good basis or good initial guess for MVNN to accurately learn the mean field drift.

 \begin{figure}
     \centering
     \includegraphics[width=0.99\linewidth]{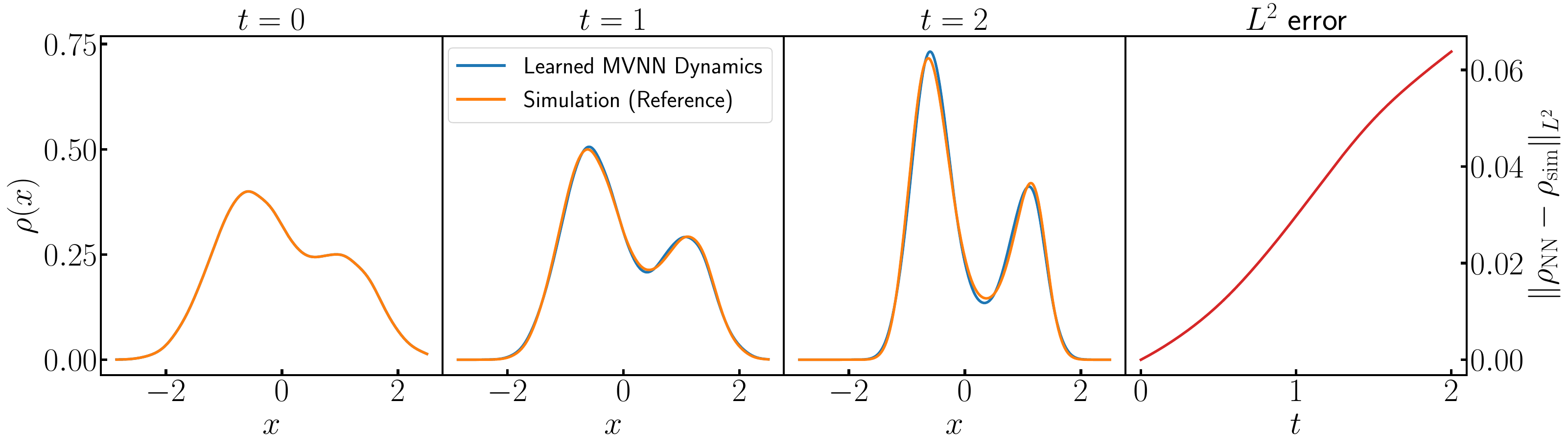}
     \includegraphics[width=0.99\linewidth]{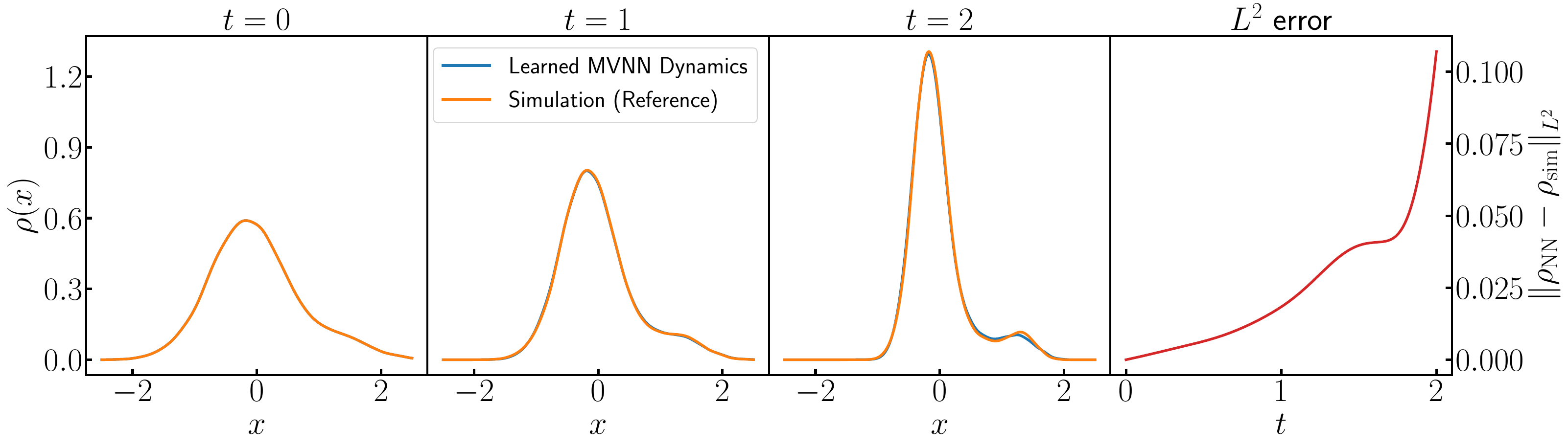}
     \includegraphics[width=0.99\linewidth]{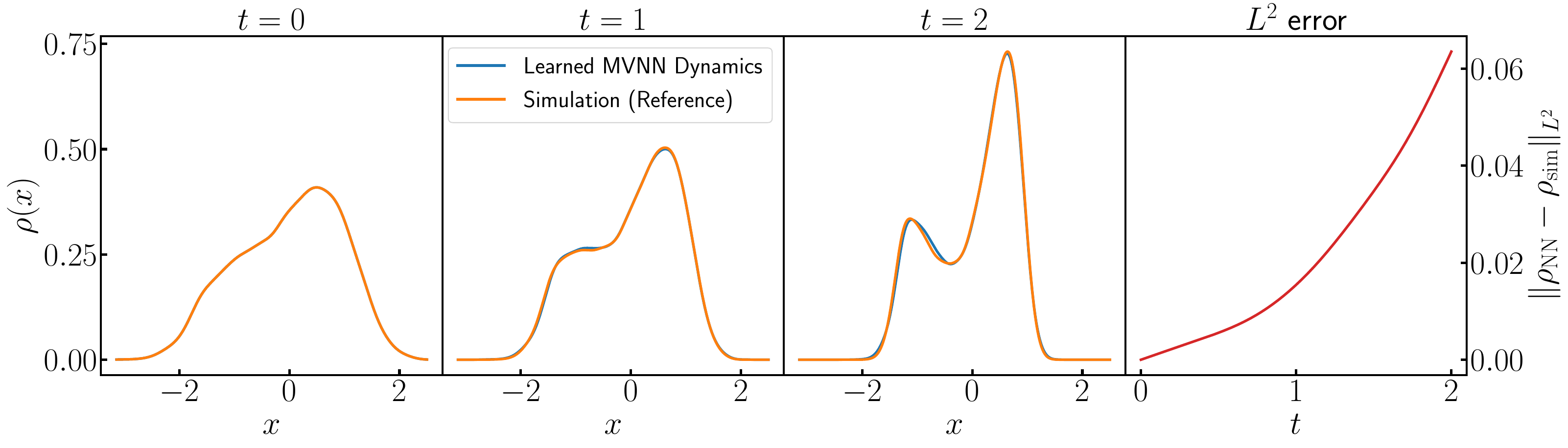}
     \caption{\textbf{1D Motsch-Tadmor dynamics}: Empirical density $\rho(x,t)$ for the 1D Motsch-Tadmor dynamics: comparison between the reference $N$-particle simulation (orange) and the MVNN-learned mean-field model (blue). Columns show $t=0,1,2$; rows correspond to three unseen initial distributions. Densities are estimated using Gaussian kernel density estimation. The $L^2$ error is computed between the KDE-smoothed densities.}
     \label{fig:1d_case1}
 \end{figure}

  \begin{figure}
     \centering
     \includegraphics[width=0.75\linewidth]{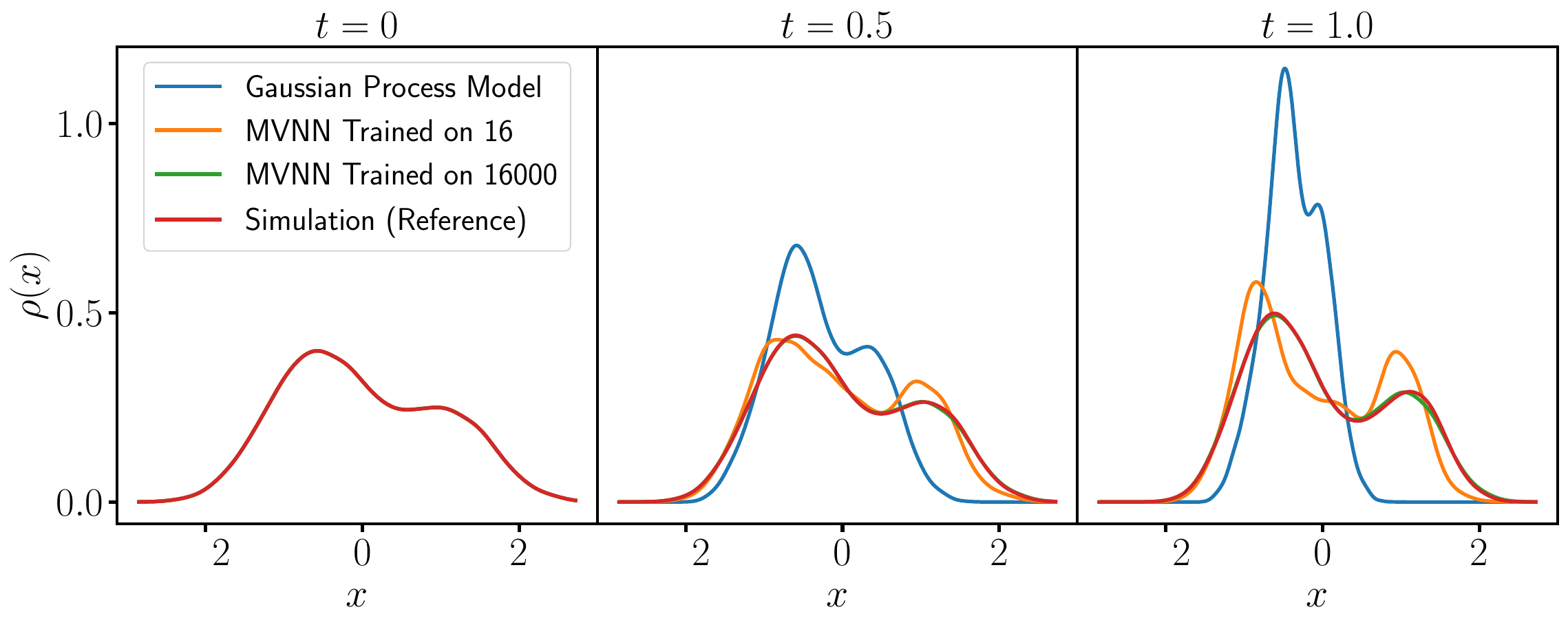}
     \includegraphics[width=0.75\linewidth]{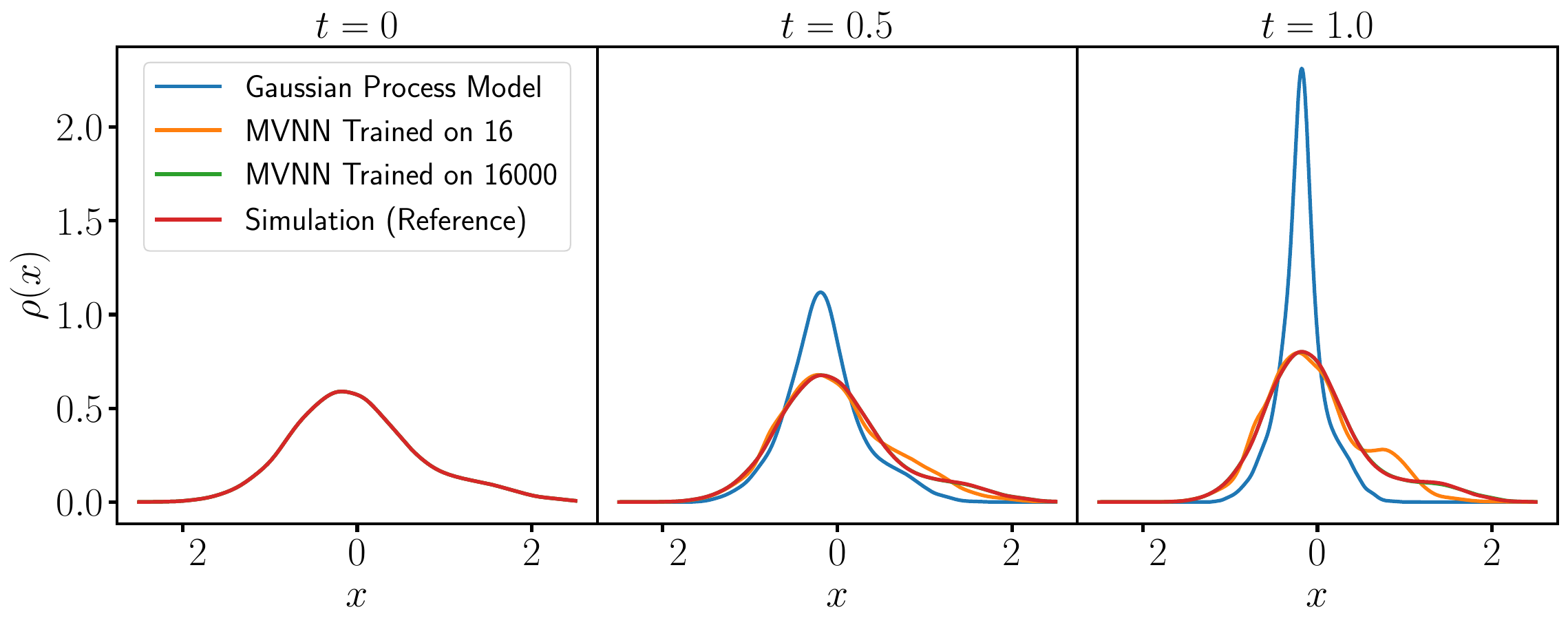}
     \includegraphics[width=0.75\linewidth]{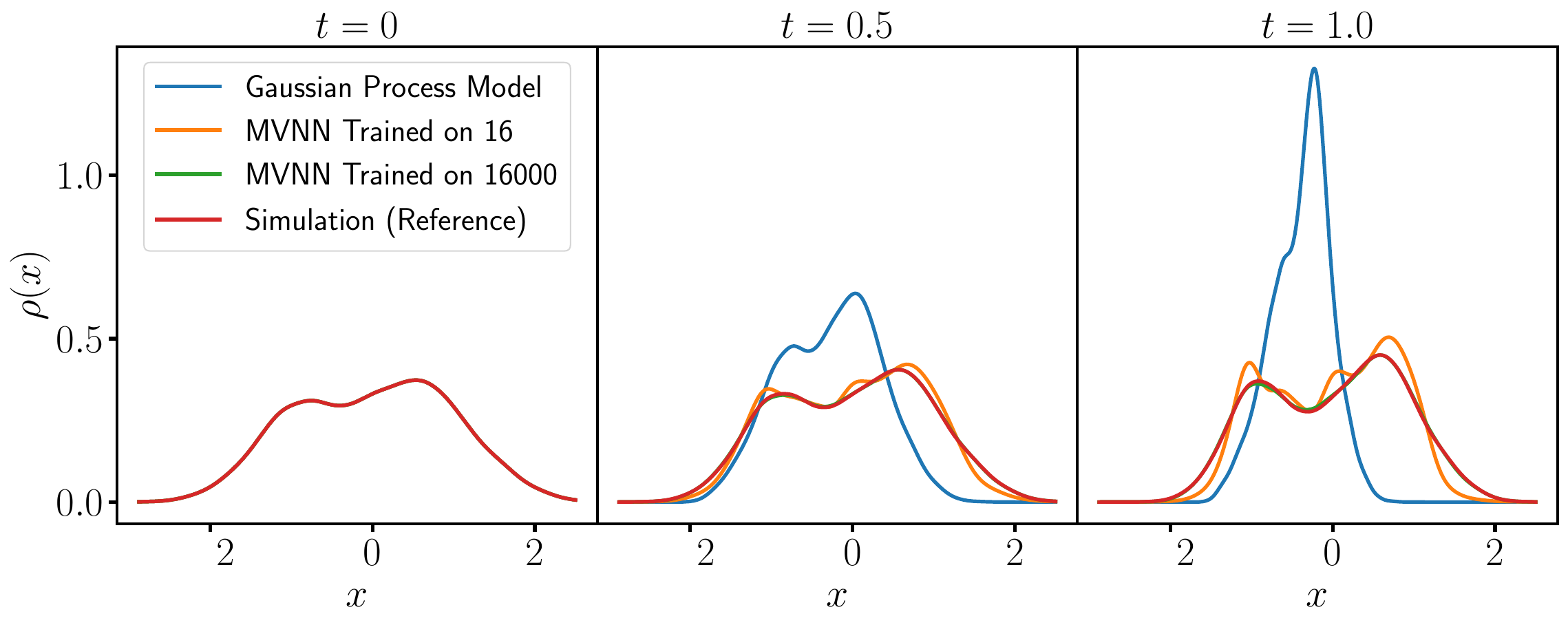}
     \caption{\textbf{Comparisons on 1D Motsch-Tadmor dynamics}: Empirical density $\rho(x,t)$ for the 1D Motsch-Tadmor dynamics: comparison between the reference $N$-particle simulation (red), the MVNN-learned mean-field model trained on 16 agents, 9 trajectories, and 20 timesteps (orange), the MVNN-learned mean-field model trained on $16{,}000$ agents, 100 trajectories, and 200 timesteps (green) and the prediction from the Gaussian process model \citep{feng2023learning} (blue). Columns show $t=0,0.5,1$; rows correspond to three unseen initial distributions. Densities are estimated using Gaussian kernel density estimation.}
     \label{fig:mvnn_gp_comparison}
 \end{figure}

   \begin{figure}
     \centering
     \includegraphics[width=0.75\linewidth]{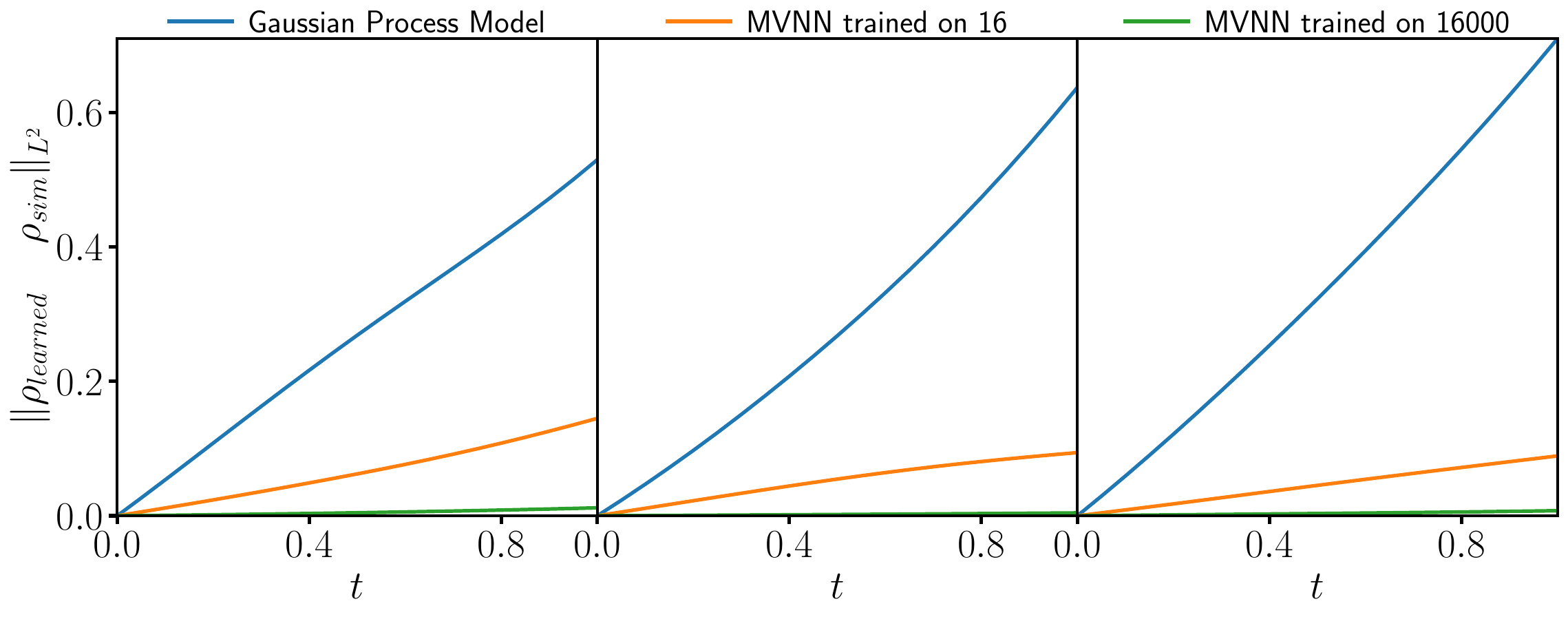}
     \caption{\textbf{Comparisons on 1D Motsch-Tadmor dynamics}: $L^2$ error for the 1D Motsch-Tadmor dynamics: comparison of the $L^2$ error for the Gaussian process model (blue), the MVNN model trained on $16$ agents, $9$ trajectories, and $20$ timesteps (orange), and the MVNN model trained on $16{,}000$ agents, $100$ trajectories, and $200$ timesteps (green). The $L^2$ error is computed between the KDE-smoothed densities. Columns correspond to three unseen initial distributions.}
     \label{fig:mvnn_gp_error}
 \end{figure}

 \begin{figure}
     \centering
     \includegraphics[width=0.4\linewidth]{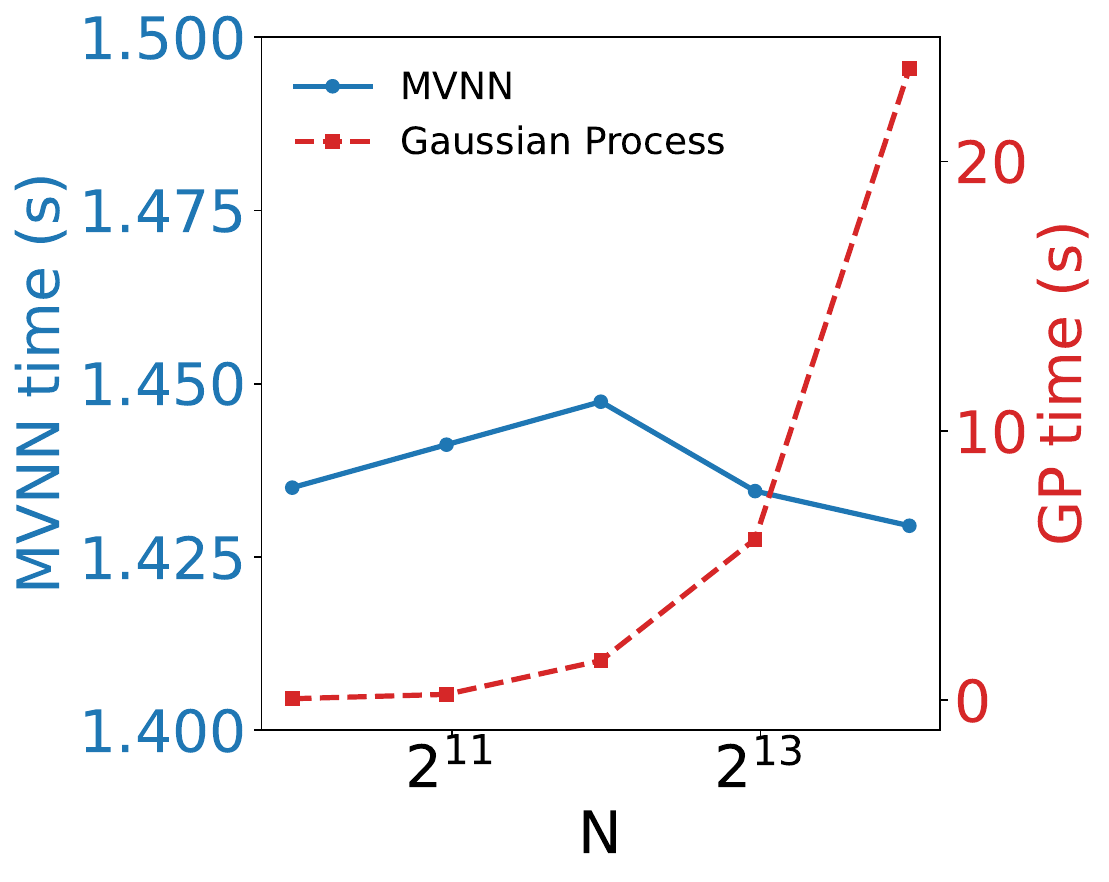}
     \caption{\textbf{Simulation Time Comparison}: Comparison of average simulation times (seconds) against number of agents $N$ for the matched-data MVNN and Gaussian process models over 10 trials.}
     \label{fig:mvnn_gp_time} 
 \end{figure}

\begin{figure}
     \centering
     \includegraphics[width=0.75\linewidth]{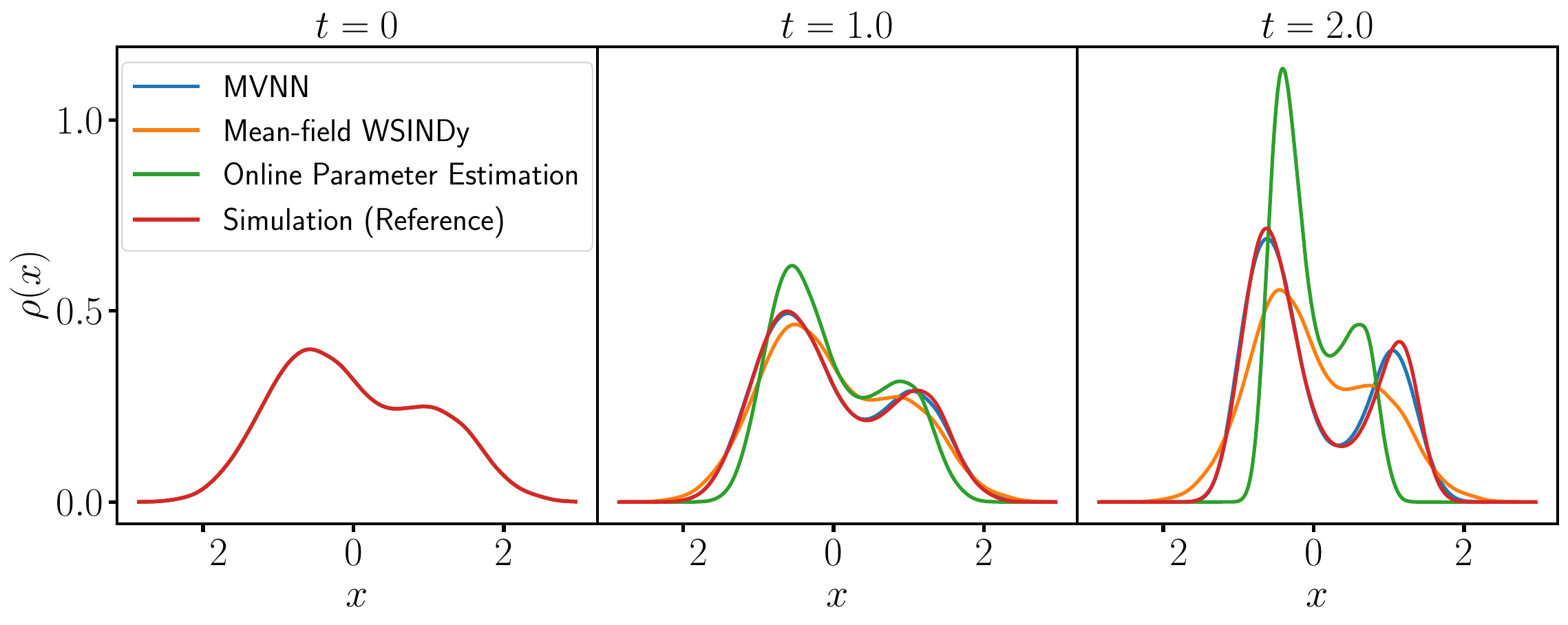}
     \includegraphics[width=0.75\linewidth]{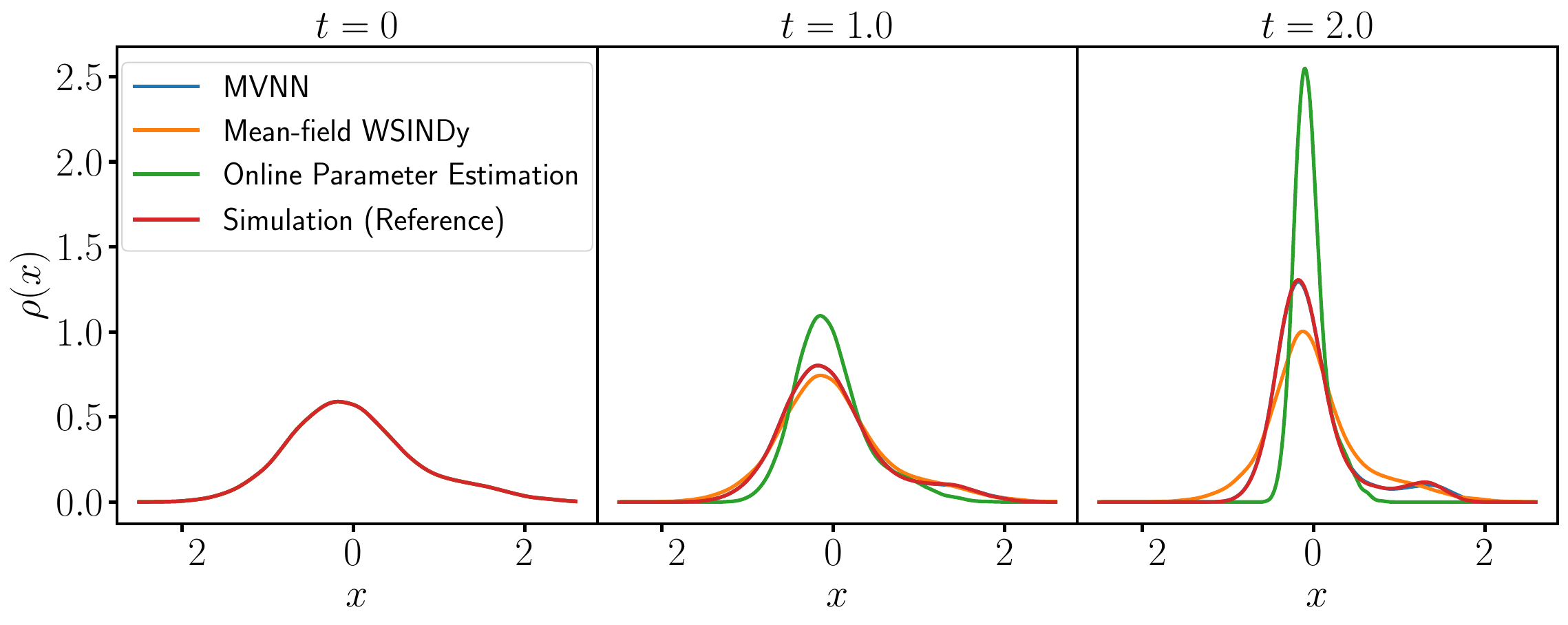}
     \includegraphics[width=0.75\linewidth]{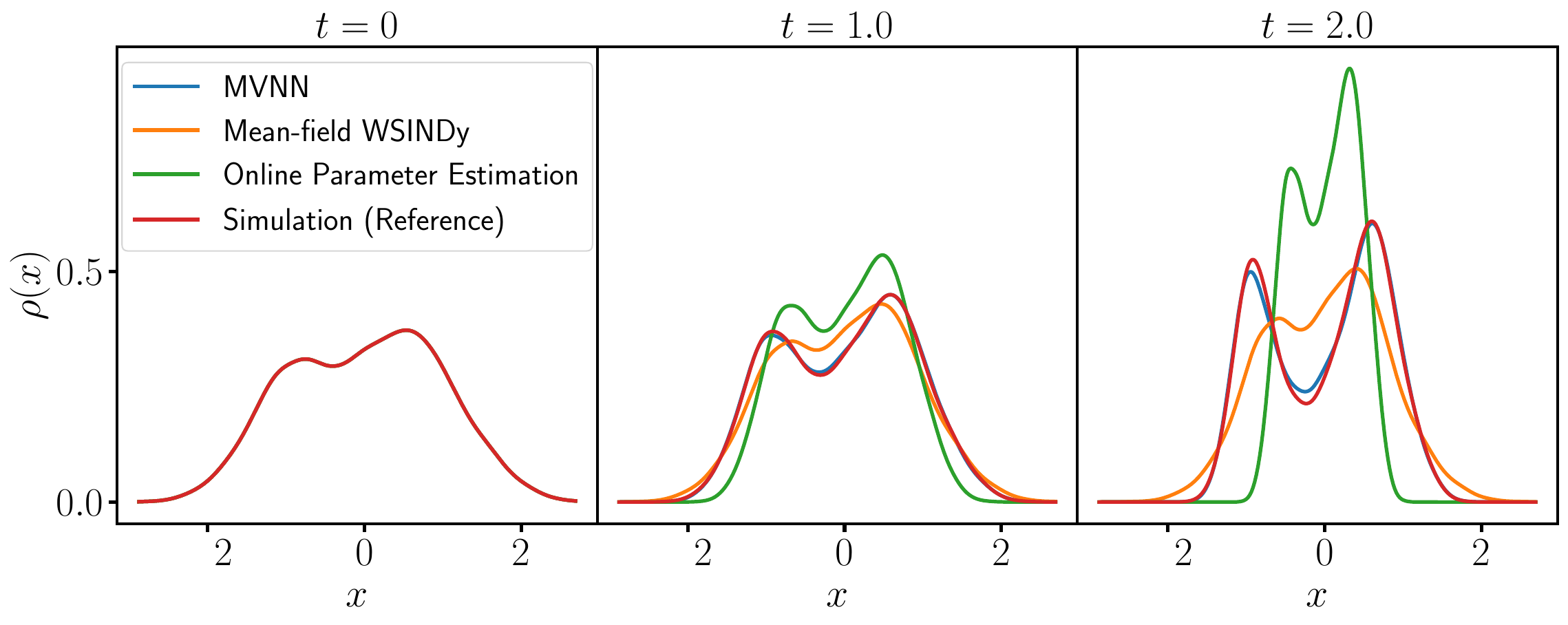}
     \caption{\textbf{Comparisons on 1D Motsch-Tadmor dynamics}: Empirical density $\rho(x,t)$ for the 1D Motsch-Tadmor dynamics: comparison between the reference $N$-particle simulation (red), the MVNN-learned mean-field model trained on $16{,}000$ agents, 100 trajectories, and 200 timesteps (blue), the prediction from mean-field WSINDy~\citep{messenger2022learning}(orange), and the prediction from online parameter estimation~\citep{sharrock2023online} (green). Columns show $t=0,1,2$; rows correspond to three unseen initial distributions. Densities are estimated using Gaussian kernel density estimation.}
     \label{fig:mvnn_wsindy_param_est}
 \end{figure}

\subsubsection{1D Stochastic Motsch-Tadmor Dynamics}
We extend our validation to the stochastic variant of the system, where the dynamics are driven by:
\[
\intd X_t^i
= \frac{1}
{\sum_{k=1}^N \phi\left(|X^k_t-X^i_t|\right)}\sum_{j=1}^N \phi\left(|X^j_t-X^i_t|\right)
\left(X^j_t-X^i_t\right)\,\intd t
+ \sigma\,\intd B_t^i,
\qquad i=1,\ldots,N,
\]
 with noise strength $\sigma = 0.1$. For the stochastic case, we employ the Euler–Maruyama method with time step $\Delta t = 10^{-2}$, whereas the remaining parameters are kept identical to the deterministic configuration. Consistent with our problem setup, the diffusion coefficient $\sigma$ is treated as a known constant, and the learning task focuses solely on identifying the effective drift field. The learned macroscopic dynamics are simulated via the corresponding McKean-Vlasov SDE driven by the trained MVNN:
\[
\intd X_t^{i,N}
= \hat{b}_\theta\left(X_t^{i,N},\,\hat{\mu}_t^N\right)\intd t
+ \sigma\intd B_t^{i,N},
\qquad i=1,\ldots,N,
\]
where $\hat\mu_t^N$ denotes the empirical measure defined above.
Figure~\ref{fig:stochastic_case_1} shows the time dynamics of the agent density under stochastic forcing and a comparison with the true simulation.

We again compare our MVNN model to mean-field WSINDy~\citep{messenger2022learning} and online parameter estimation~\citep{sharrock2023online}. As before, for mean-field WSINDy, we use a finite linear combination of monomial basis functions to approximate the drift, and for online parameter estimation, we estimate the parameter $\ell$ in the interaction kernel and pick an initial guess of $\ell_0 = 1.5$. To generate the training set for all the models, we use the Euler–Maruyama method with  $\Delta t = 10^{-2}$ and keep the remaining parameters the same as the deterministic case. The initial states for each trajectory are the same as in the deterministic configuration as well. Figure~\ref{fig:stochastic_mvnn_wsindy_param_est} shows the time evolution of the densities predicted by MVNN, mean-field WSINDy, and online parameter estimation alongside the true dynamics for $N=16{,}000$ agents. This result shows that even in the stochastic case, the MVNN prediction matches the ground truth better than mean-field WSINDy and online parameter estimation.

\begin{figure}[t]
    \centering
    \includegraphics[width=0.99\linewidth]{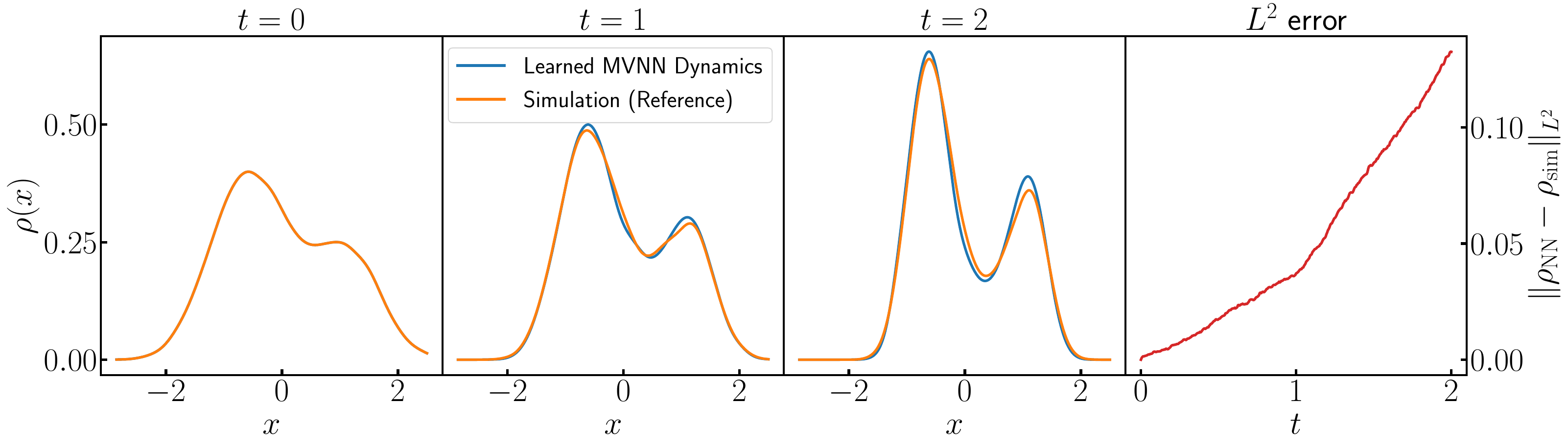}
    \caption{\textbf{Stochastic Motsch-Tadmor dynamics ($\sigma=0.1$)}: density evolution.
Empirical density $\rho(x,t)$ from the reference interacting-particle simulation (orange) and from the MVNN-learned McKean-Vlasov model (blue), shown at $t=0,1,2$ for an unseen initial distribution. Densities are estimated via Gaussian KDE, and the reported $L^2$ error is computed between the kernel-smoothed densities.}
    \label{fig:stochastic_case_1}
\end{figure}

\begin{figure}
     \centering
     \includegraphics[width=0.75\linewidth]{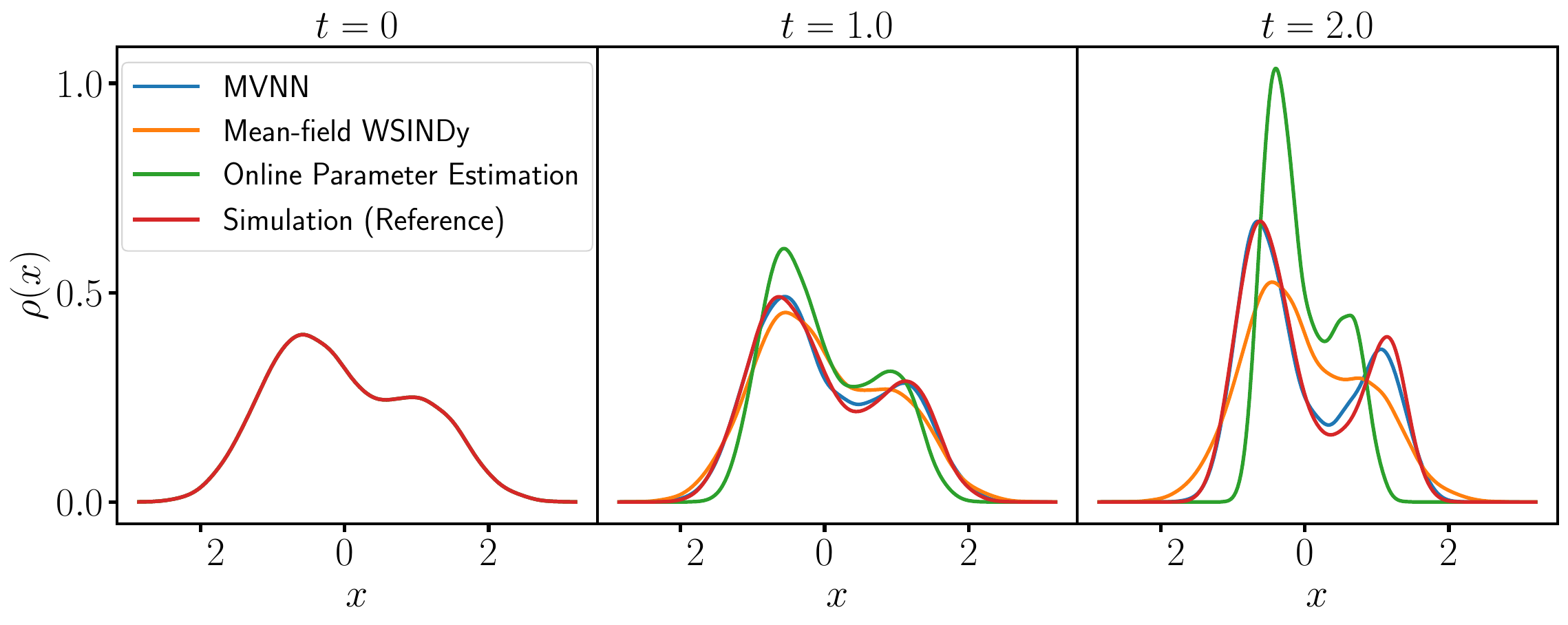}
     \includegraphics[width=0.75\linewidth]{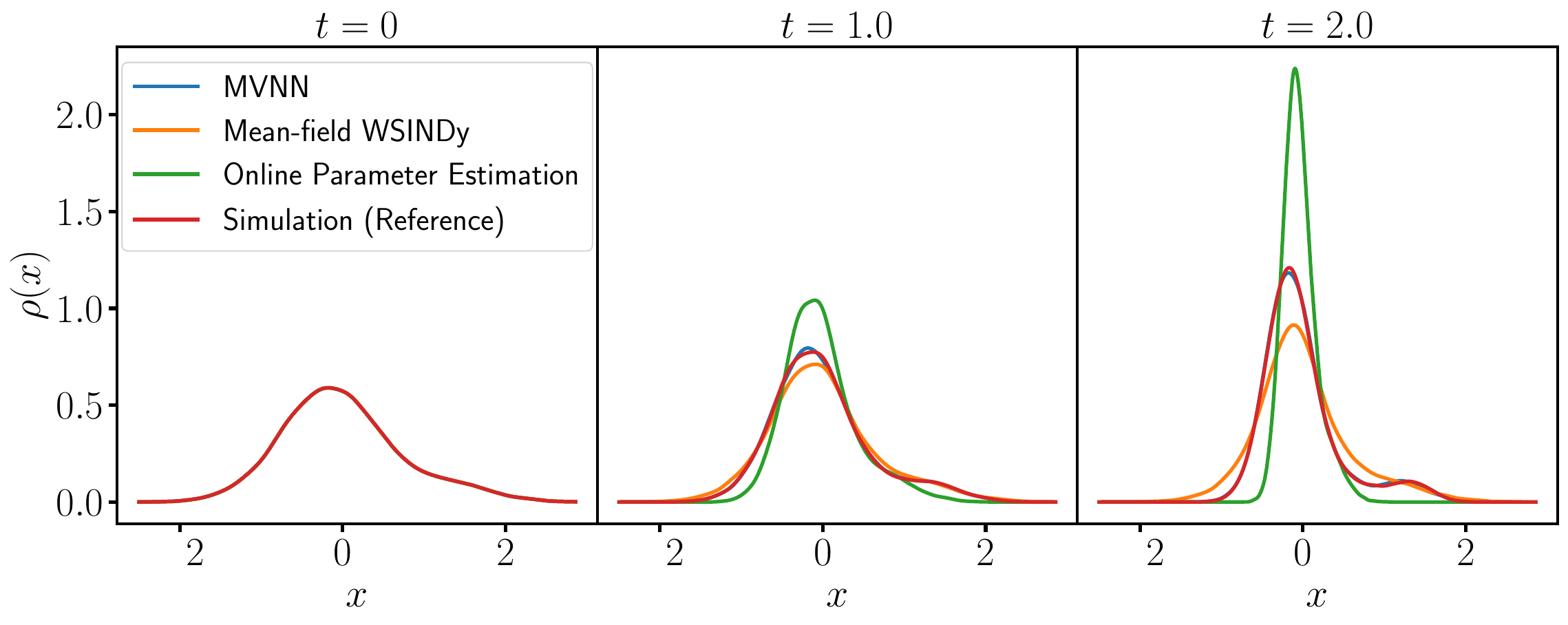}
     \includegraphics[width=0.75\linewidth]{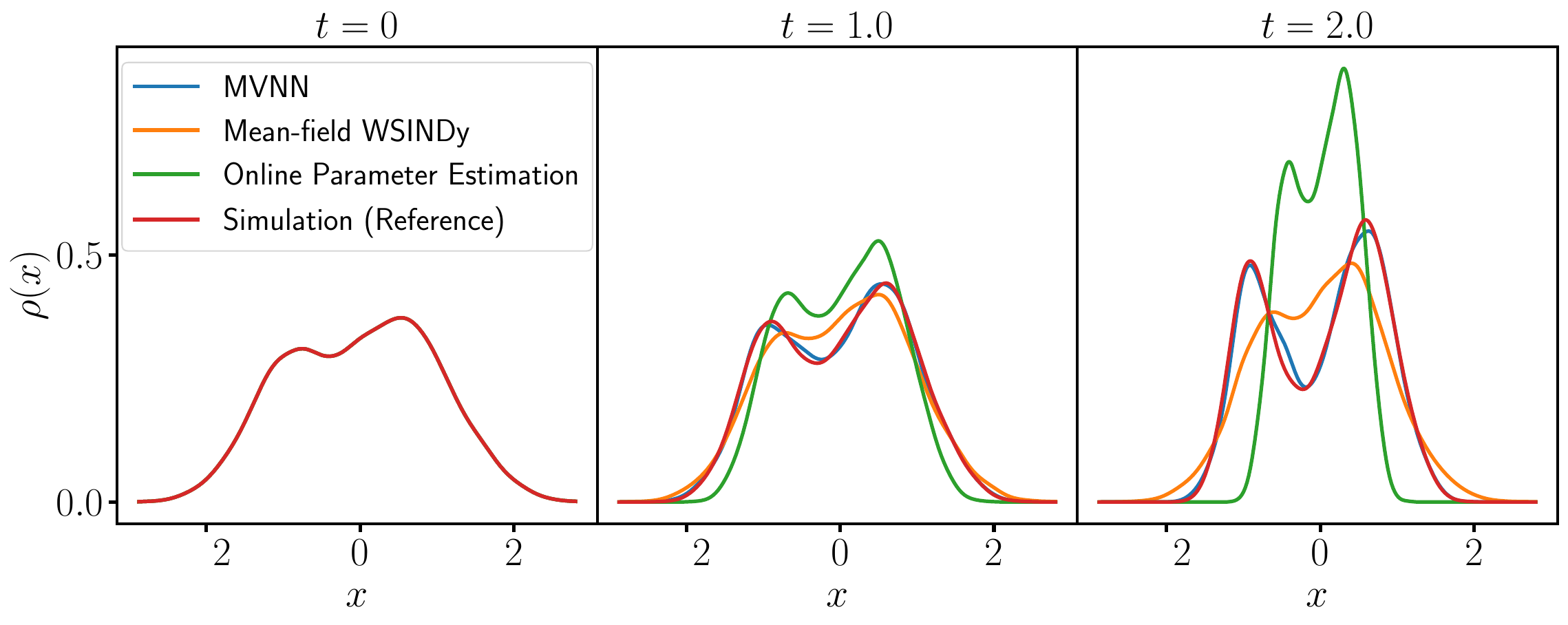}
     \caption{\textbf{Comparisons on 1D Stochastic Motsch-Tadmor dynamics ($\sigma=0.1$)}: Empirical density $\rho(x,t)$ for the 1D Stochastic Motsch-Tadmor dynamics: comparison between the reference $N$-particle simulation (red), the MVNN-learned mean-field model trained on $16{,}000$ agents, 100 trajectories, and 200 timesteps (blue), the prediction from mean-field WSINDy~\citep{messenger2022learning}(orange), and the prediction from online parameter estimation~\citep{sharrock2023online} (green). Columns show $t=0,1,2$; rows correspond to three unseen initial distributions. Densities are estimated using Gaussian kernel density estimation.}
     \label{fig:stochastic_mvnn_wsindy_param_est}
 \end{figure}

\subsubsection{Comparison with Distribution-Based Equation Discovery}
\label{sec:wsindy-comparison}
We further compare MVNN with two WSINDy-based approaches that learn the
evolution of the population distribution and with the online parameter
estimation method of \cite{sharrock2023online}. We consider the stochastic
interacting-particle system
\begin{equation}
\label{eq:global-mean-benchmark}
\mathrm dX_t^i
=
\left[
\gamma(M_t^N-X_t^i)
+
\left(a-\alpha (M_t^N)^2\right)X_t^i
-
\beta (X_t^i)^3
\right]\mathrm dt
+
\sigma\,\mathrm dW_t^i,
\qquad
M_t^N=\frac{1}{N}\sum_{j=1}^N X_t^j ,
\end{equation}
for \(i=1,\ldots,N\). We set
\[
a=\alpha=1,\qquad
\beta=\gamma=0.2,\qquad
\sigma=0.03,\qquad
N=1024.
\]
The corresponding population density satisfies the nonlocal Fokker--Planck
equation
\begin{equation}
\partial_t\rho
=
-\partial_x\left\{
\rho\left[
\gamma(M_\rho-x)
+
(a-\alpha M_\rho^2)x
-
\beta x^3
\right]
\right\}
+
\frac{\sigma^2}{2}\partial_{xx}\rho,
\qquad
M_\rho(t)=\int_{\mathbb R}x\rho(x,t)\,\mathrm dx .
\end{equation}
Particle trajectories are generated using the Euler--Maruyama method with
time step \(\Delta t=0.01\) over \(t\in[0,1]\). The initial particle
positions are sampled from randomly generated two- and three-component
mixtures of compactly supported mollifier distributions,
\[
\psi_{m,w}(x)
=
Z_w^{-1}
\exp\left(
-\frac{1}{1-((x-m)/w)^2}
\right)
\mathbf{1}_{\{|x-m|<w\}},
\]
where the component centers, support widths, and mixture weights vary across
trajectories.

Our first distribution-based baseline is a local PDE-WSINDy model based on
the weak-form PDE discovery framework of
\cite{messenger2021weak}. For a fair comparison, we fix the diffusion term as
\(\sigma^2/2\partial_{xx}\rho\) and identify only the conservative flux $J_\Theta$:
\[
\partial_t \rho  = -\partial_x J_\Theta + \frac{\sigma^2}{2}\partial_{xx}\rho.
\]
The local
PDE-WSINDy library for $J_\Theta$ contains spatially varying polynomial fluxes
\(B_j(x)\rho^q\), where \(B_j\) are cubic B-spline basis functions and
\(q\in\{1,2,3\}\), together with standard nonlinear density and
density-gradient terms $\rho\partial_x\rho,\rho^2\partial_x\rho$.

The second baseline follows the mean-field WSINDy formulation of
\cite{messenger2022learning} and assumes the nonlocal drift
\begin{equation}
\widehat b(x,\mu)
=
f_B(x)+\int_{\mathbb R}K_B(y-x)\,\mu(\mathrm dy),
\end{equation}
where both \(f_B\) and the unknown interaction kernel \(K_B\) are expanded
in generic cubic B-spline bases. Both WSINDy methods are trained
from densities reconstructed from the same particle trajectories, and all
methods use the known diffusion coefficient
\(\sigma^2/2=4.5\times10^{-4}\).

The third baseline applies the online parameter estimation method of
\cite{sharrock2023online} to this benchmark. The four drift coefficients
\((a, \alpha, \beta, \gamma)\) in \eqref{eq:global-mean-benchmark} are
estimated from the same particle trajectories, and we pick an initial guess of $(a_0,\alpha_0,\beta_0,\gamma_0) = (2,2,1.2,1.2)$. The fitted
McKean--Vlasov system is then simulated directly from the held-out initial
conditions.

\begin{figure}
    \centering
    \includegraphics[width=0.99\linewidth]{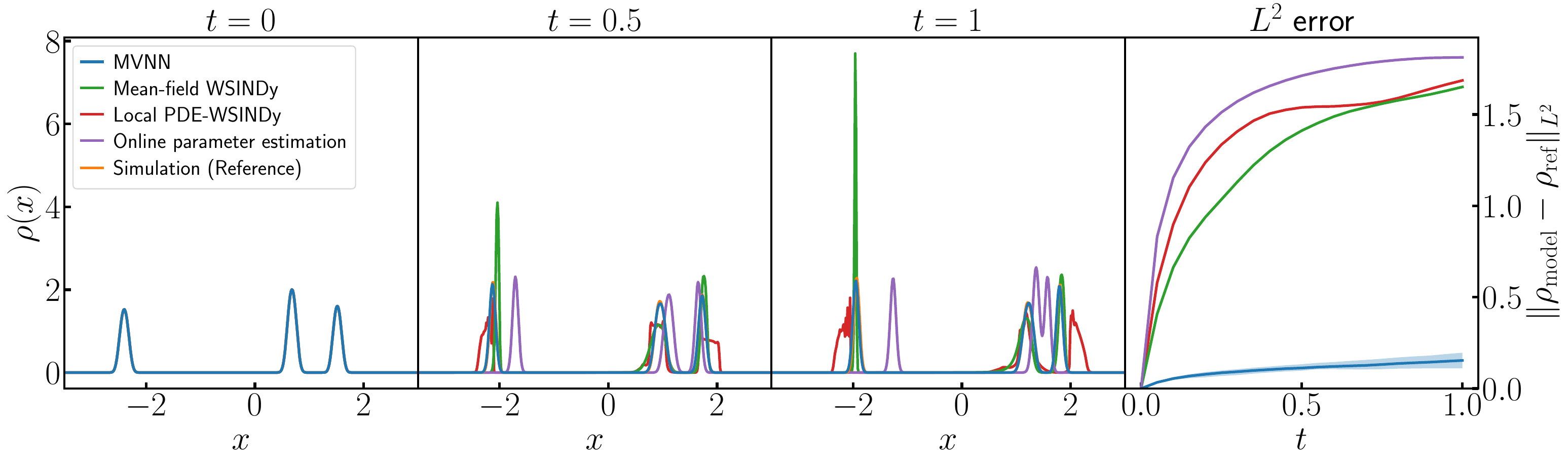}
    \caption{\textbf{Comparison with WSINDy and online parameter estimation}. The first three
    panels show the forecast densities at $t=0$, $0.5$, and $1$ for the learned
    MVNN dynamics, the B-spline mean-field WSINDy model, the local PDE-WSINDy
    model, and online parameter estimation. All methods use the same initial empirical distribution. The PDE models are
    initialized from its KDE reconstruction, and the same KDE observation
    operator is used to evaluate and plot the densities predicted by all methods. The rightmost
    panel shows the density $L^2$ forecast error versus time, averaged over all
    36 held-out test trajectories. The
    shaded band spans the three independently trained MVNN models. The two
    additive-structure WSINDy baselines misplace and over-concentrate the bumps
    and the fitted parametric model displaces them,
    whereas MVNN tracks the reference throughout.}
    \label{fig:wsindy-global-mean}
\end{figure}
Figure~\ref{fig:wsindy-global-mean} shows the density \(L^2\) forecast error.
As an additional distributional metric, the mean time-averaged
Wasserstein-1 errors are \(0.2661\) for local PDE-WSINDy, \(0.09129\) for B-spline kernel
WSINDy, and \(0.1835\) for online parameter estimation, whereas the three independently trained MVNN models obtain errors
between \(0.004677\) and \(0.006271\). The PDE-based models have slightly
lower rollout costs, but their forecast errors are substantially larger than
those of MVNN.

The difference in accuracy is primarily caused by the representational
restrictions imposed by the candidate model classes. The local PDE-WSINDy
model assumes that the population flux can be expressed using only the local
density, its spatial derivatives, and the spatial coordinate. However, the
drift in~\eqref{eq:global-mean-benchmark} depends on the global quantity
\[
M_\rho^2
=
\left(\int_{\mathbb R}x\rho(x,t)\,\mathrm dx\right)^2,
\]
and therefore lies outside the prescribed local PDE library. The nonlocal
mean-field WSINDy model is more flexible, but it assumes an additive pairwise
interaction of the form
\[
f(x)+\int_{\mathbb R}K(y-x)\rho(y,t)\,\mathrm dy.
\]
This expression is linear in the population measure and cannot represent the
quadratic measure dependence \(M_\rho^2x\), which corresponds to an effective
higher-order interaction.
The online parameter estimation baseline operates within the true parametric
family, but the estimated coefficients differ substantially from the true
values on this benchmark, and the method also assumes that the form of the drift function is known. The resulting forecasts displace the density bumps,
and the forecast error is comparable to that of the WSINDy baselines.

In principle, WSINDy could represent this system if its candidate library
were explicitly augmented with the global features \(M_\rho\),
\(M_\rho^2\), and their coupling with \(x\). This, however, requires knowing
the relevant nonlocal structure and interaction order before model
identification. Such information may be unavailable or difficult to specify a priori in
practical applications.
Therefore, selecting a sufficiently expressive but computationally manageable
candidate library is itself a difficult modeling problem. In contrast, MVNN
does not require the user to prescribe the relevant moment \(M_\rho^2\), a
pairwise interaction form, a finite-dimensional parametric family, or the
order of the measure dependence. It provides a flexible and general
representation of continuous measure-dependent drifts.
On the same CPU configuration, the average rollout times per trajectory are
\(0.365\) s for local PDE-WSINDy, \(0.327\) s for mean-field WSINDy, and
\(0.418\)--\(0.432\) s for MVNN.

\subsubsection{2D Aggregation Dynamics with Attraction-Repulsion}
We further evaluate the framework on a two-dimensional first-order swarm model governed by attractive-repulsive interactions. Unlike the consensus models discussed previously, this system exhibits rich spatial pattern formation, such as rings and clumps. The dynamics of the $N$ agents are described by:
\[
\dot{\bX}^i_t 
= \frac{1}{N}\sum_{j=1}^N 
\phi\left(\|\bX^j_t - \bX^i_t\|\right)(\bX^j_t - \bX^i_t),
\qquad i=1,\ldots,N,
\]
where $\bX^i_t \in \mathbb{R}^2$ denotes the position of agent $i$. The interaction kernel $\phi$ combines short-range repulsion and long-range attraction, modeled by a sum of Gaussians:$$\phi(r) = c_\text{rep} \exp\left(-(r/\ell_\text{rep})^2\right) - c_\text{att}\exp\left(-(r/\ell_\text{att})^2\right).$$We use the parameters $c_\text{rep} = 1.0$, $\ell_\text{rep} = 0.5$, $c_\text{att} = 0.7$, and $\ell_\text{att}=2.0$. The simulation is performed with $N=16{,}000$ agents for $200$ steps using a time step $\Delta t=10^{-2}$.
To test the model's ability to learn and represent complex geometric structures, we construct the training set using noisy annulus initial configurations. Specifically, for each trajectory, 
the initial position of each agent $i$ is then sampled via polar coordinates with additive Gaussian noise:
\[\begin{aligned}
\Theta_i &\sim \mathcal U(0,2\pi), \quad
\rho_i \sim \mathcal U\bigl(R_0-\tfrac{W}{2},\,R_0+\tfrac{W}{2}\bigr), \\
\bX^i_0 &= \rho_i\bigl(\cos\Theta_i,\sin\Theta_i\bigr) + \varepsilon_i, \quad \text{with} \quad
\varepsilon_i \sim \mathcal N(\mathbf 0,\,\sigma_0^2 I_2),
\end{aligned}
\]
where $I_2$ denotes the identity matrix. We generate 100 distinct initial distributions to form the dataset.
Figure~\ref{fig:ring} visualizes the evolution of the system initialized from a ring distribution. The learned mean-field dynamics successfully reproduce the stability of the ring structure and the correct contraction rate, demonstrating that the MVNN can capture effective potentials that support metastable geometric patterns.
\begin{figure}[h!]
    \centering
    \includegraphics[width=0.6\linewidth]{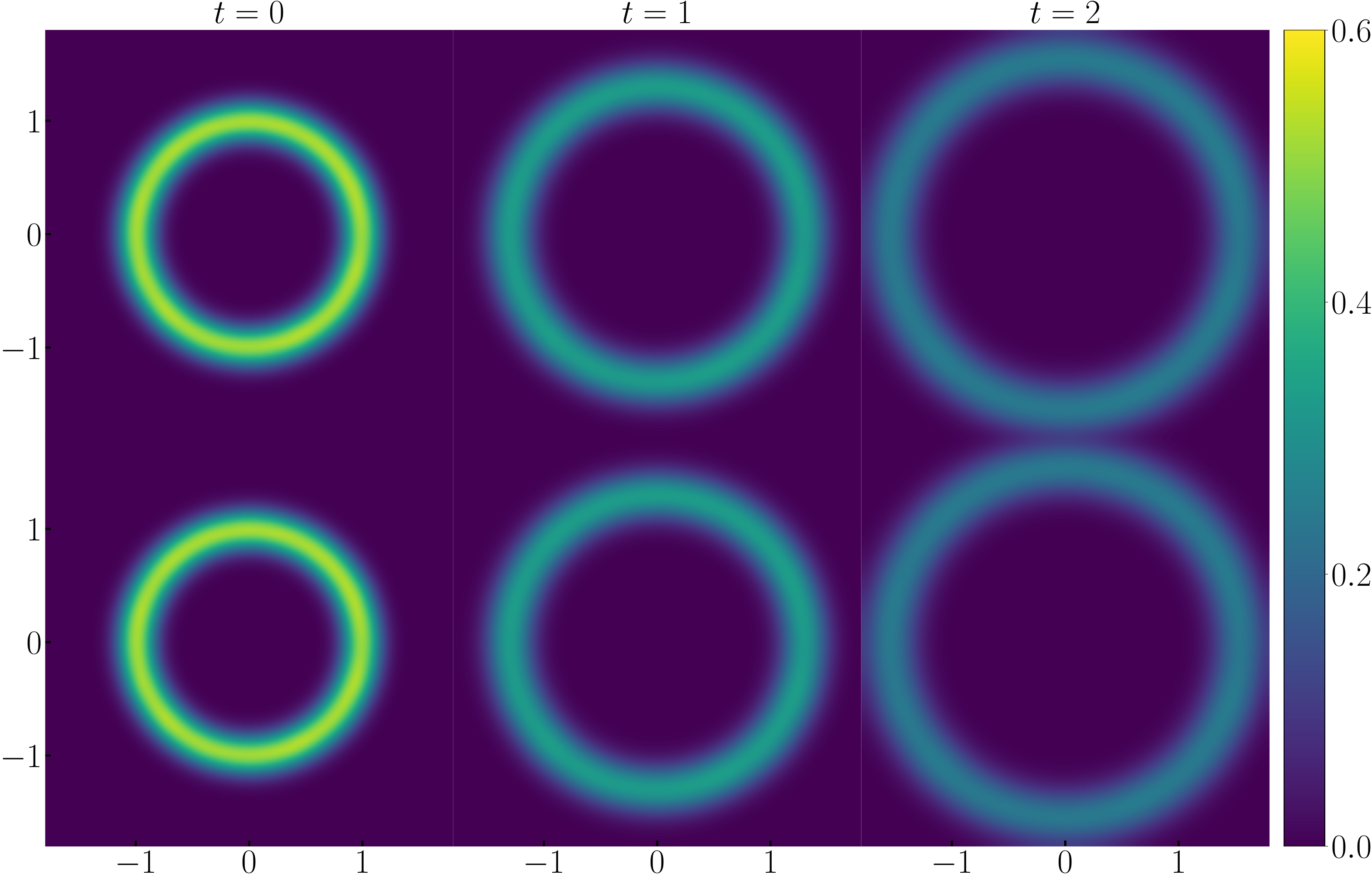}
    \caption{\textbf{2D aggregation model with ring-shaped initialization.} The upper row displays the ground truth particle trajectories, while the lower row shows the evolution predicted by the learned mean-field dynamics. The model accurately preserves the topological structure of the ring over time.}
    \label{fig:ring}
\end{figure}

Here, generalization to an unseen initial distribution refers specifically to
a change in the geometry and density profile of the initial population
measure. The model is trained exclusively on noisy single-annulus
distributions and tested on three geometrically distinct families: a
double-ring distribution, a uniform disk, and a binary asymmetric
distribution. The governing equation, interaction parameters, particle
number, and prediction horizon are unchanged. Thus, these experiments test
whether the same learned measure-dependent drift can be transferred across
different initial-distribution geometries, rather than whether the model can
extrapolate to a different governing system.

The results are presented in Figures~\ref{fig:case1_2ring},
\ref{fig:case1_disk}, and~\ref{fig:case1_hl}. In all three cases, the learned
MVNN dynamics agree well with the reference particle simulations. These
results provide empirical evidence that MVNN can transfer the learned drift
beyond the single-ring training family. We do not, however, claim guaranteed
extrapolation to arbitrary population measures, different governing dynamics,
or physical parameters outside the training regime.

\begin{figure}[h!]
    \centering
    \includegraphics[width=0.6\linewidth]{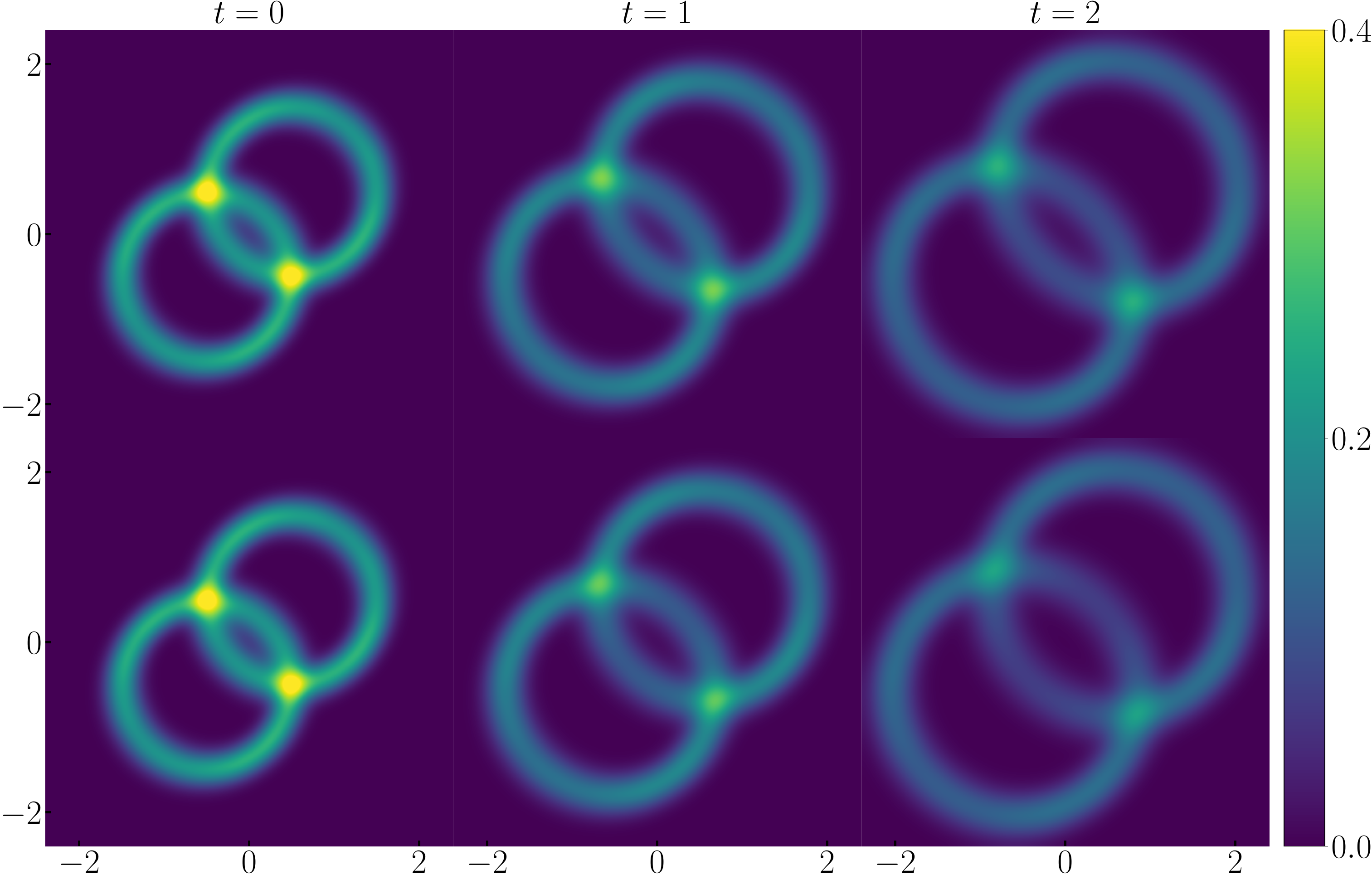}
    \caption{\textbf{2D aggregation model  with double-ring initialization.} Comparison between the ground truth particle system (upper row) and the learned mean-field dynamics (lower row). The model correctly reproduces the contraction of both rings despite never seeing this topology during training.}
    \label{fig:case1_2ring}
\end{figure}

\begin{figure}[h!]
    \centering
    \includegraphics[width=0.6\linewidth]{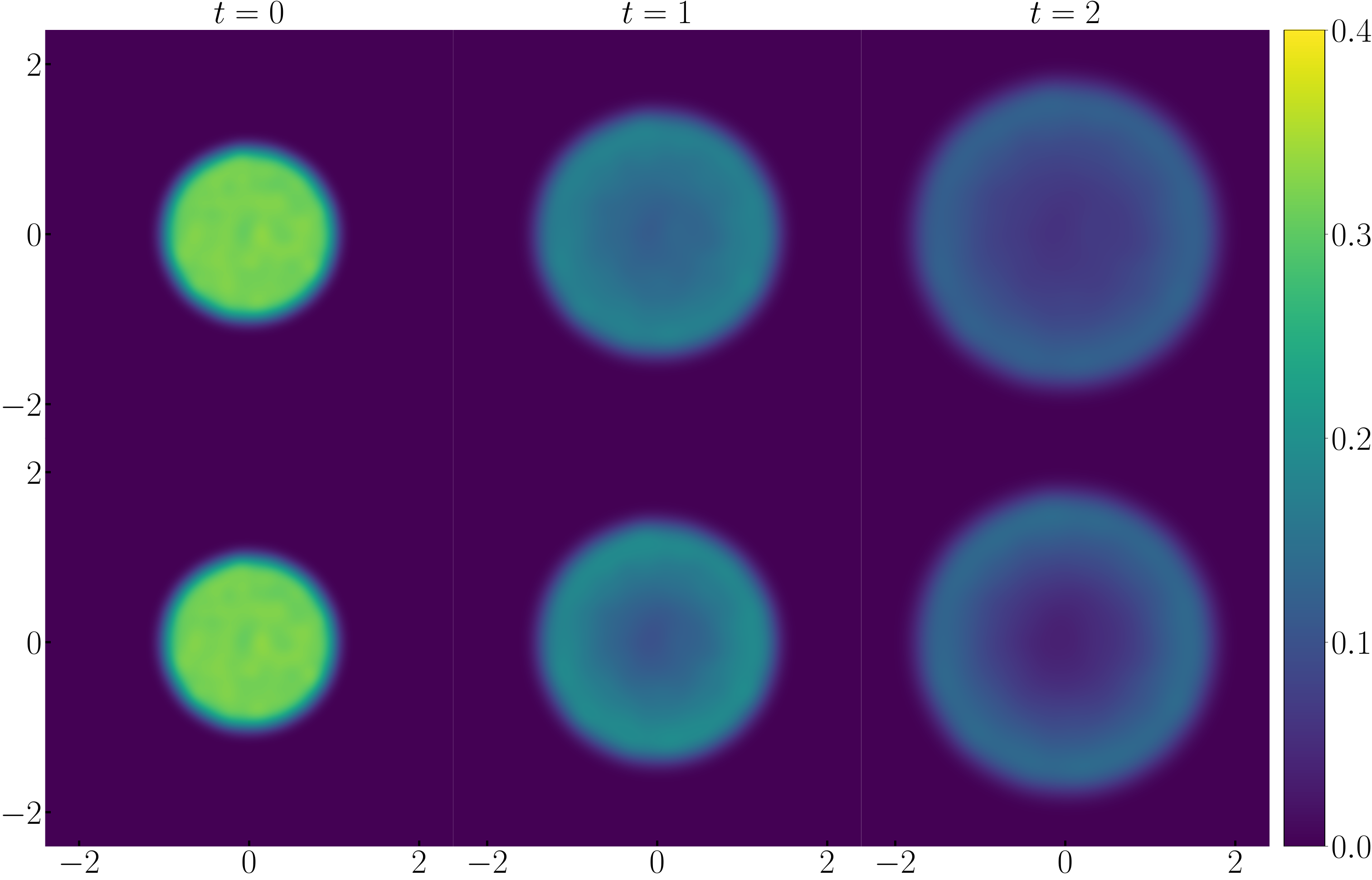}
    \caption{\textbf{2D aggregation model with disk-shaped initialization.} The learned dynamics (lower row) accurately capture the collapse of the uniform disk, matching the ground truth (upper row).}
    \label{fig:case1_disk}
\end{figure}

\begin{figure}[h!]
    \centering
    \includegraphics[width=0.6\linewidth]{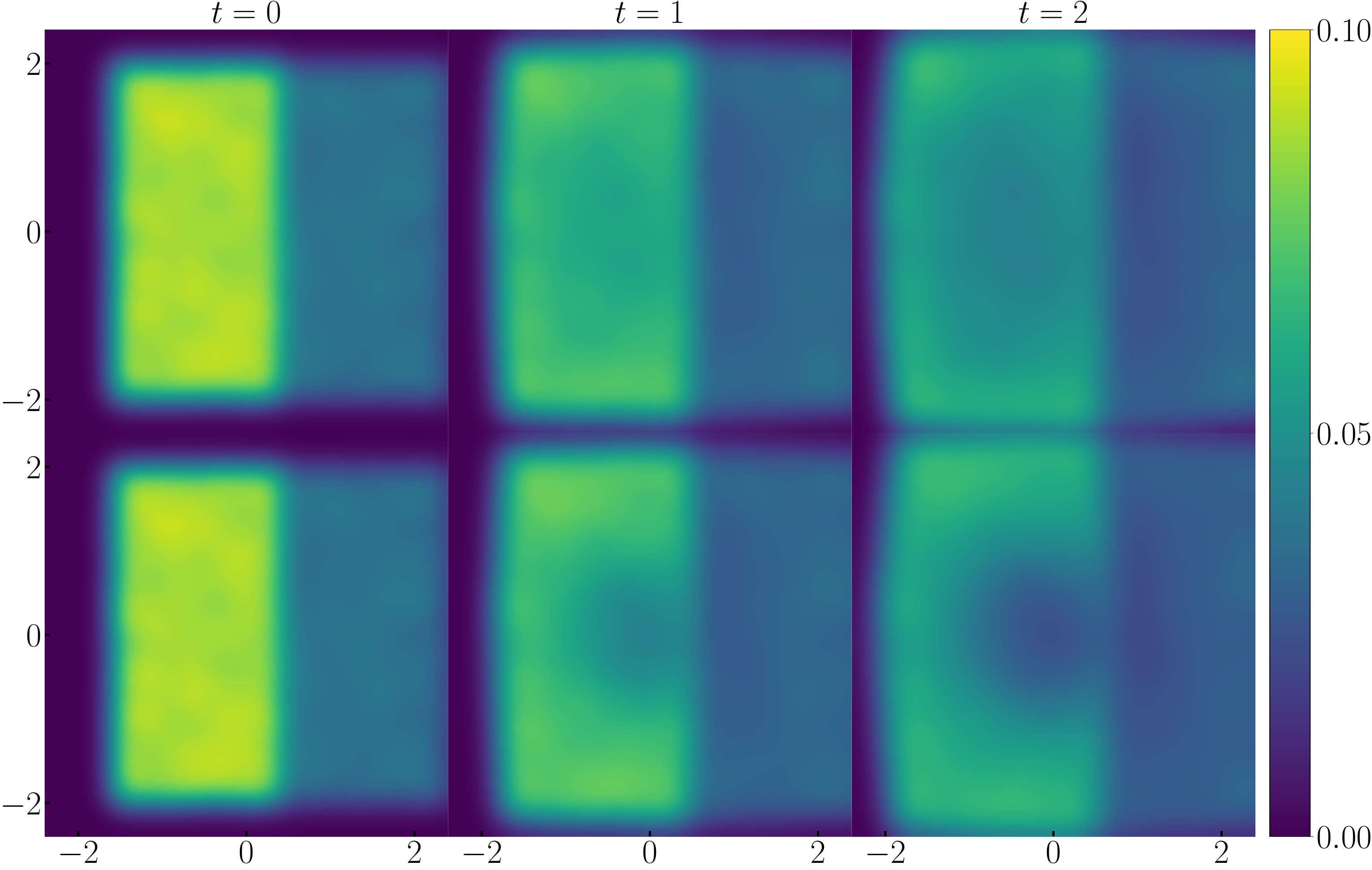}
    \caption{\textbf{2D aggregation model with binary asymmetric initialization.} Evolution of a system initialized with heterogeneous densities (high density left, low density right). The learned model (lower row) preserves the density gradient and correctly predicts the asymmetric aggregation process.}
    \label{fig:case1_hl}
\end{figure}

\section{Extensions to Systems with Heterogeneous Agent Groups}
Our method extends naturally to a wide class of interacting agent systems that arise in different applications, including systems involving multiple groups of agents. Let the agents be divided into $K$ distinct groups, with the $k$-th group consisting of $N_k$ agents whose states are denoted by $\bX^{i,k}$, for $i=1,\ldots,N_k$ and $k=1,\ldots,K$. 
The dynamics of the system can then be described by:
\[
\dot{\bX}_t^{i,k}
= \sum_{l=1}^K \frac{1}{N_l - \delta_{k,l}}
\sum_{\substack{j=1 \\ (i,k)\neq (j,l)}}^{N_l}
\phi_{k,l}\!\left(\|\bX_t^{i,k} - \bX_t^{j,l}\|\right)
\bigl(\bX_t^{i,k} - \bX_t^{j,l}\bigr),
\qquad i=1,\ldots,N_k,\;\; k=1,\ldots,K,
\]
where $\phi_{k,l}$ represents the interaction kernel between agents in group $k$ and group $l$. The corresponding mean-field dynamics can be written in the form:
\[
\dot{\bX}_t^{i,k} = \bb_k\!\left(\bX_t^{i,k},\,\mu_1,\ldots,\mu_K\right),
\qquad i=1,\ldots,N_k,\;\; k=1,\ldots,K,
\]
where each $\mu_k \in \mathcal{P}(\mathbb{R}^d)$ denotes the probability distribution of agents of group $k$.
The observed data consist of the positions of agents from different groups along multiple trajectories:
\[
\bX_{\mathrm{tr}} := 
\bigl\{\bX^{i,k}_{t_\ell,m}\bigr\}_{i=1,\ell=1,m=1,k=1}^{N_k,L,M,K},
\]
where $0 = t_1 < \cdots < t_L = T$ denote the observation times, 
$m$ indexes the $M$ independent trajectories, and $k=1,\ldots,K$ indexes the agent groups. 
For each group $k$, the corresponding velocities 
$\bV^{i,k}_{t_\ell,m}$ are approximated by finite differences.
Each trajectory $m$ is initialized from a collection of independent probability measures 
\(
\mu_{0,1}^m,\ldots,\mu_{0,K}^m,
\)
representing the initial distributions of the $K$ groups, 
yielding different realizations of the interacting multi-group system.
Our objective is to infer the mean-field interaction drifts 
\(
\bb_k(\bX,\mu_1,\ldots,\mu_K), \; k=1,\ldots,K,
\)
from these observations.
\subsection{Multi-Group Measure-Valued Neural Network}
We extend the proposed framework to heterogeneous systems composed of multiple interacting agent groups. We introduce the \emph{Multi-Group Measure-Valued Neural Network} (MG-MVNN) to capture the complex inter-group coupling. Consider a system with $K$ distinct groups, where the $k$-th group is characterized by its population distribution $\mu_k \in \mathcal{P}(\mathbb{R}^d)$. The effective mean-field dynamics for an agent $i$ in group $k$ are governed by a drift function $\bb_k$ dependent on the state and the distributions of all groups:
\[
\dot{\bX}_t^{i,k} = \bb_k\!\left(\bX_t^{i,k},\,\mu_1,\ldots,\mu_K\right).
\]
To allow for efficient learning of these high-dimensional dependencies, we parametrize each drift $\bb_k$ using a composite neural operator. This architecture maps the local agent state $\bX$ and the aggregated population features from all groups to the target drift. Specifically, for each group $l \in \{1, \dots, K\}$, we employ a group-specific embedding network $\varphi_{\mathrm{emb},l}:\mathbb{R}^d\to\mathbb{R}^{r_{l}}$ to extract latent representations. The collective state of group $l$ is then summarized by the moment vector:$$\mathbf{z}_l(\mu_l) := \big\langle \varphi_{\mathrm{emb},l}(\cdot;\theta_{\mathrm{emb},l}),\,\mu_l\big\rangle = \int \varphi_{\mathrm{emb},l}(\mathbf{x})\,\mu_l(\mathrm{d}\mathbf{x}).$$
The drift for group $k$ is approximated by an interaction network $\varphi_{\mathrm{int},k}$ that consumes the local state and the concatenated global features of all groups:
$$
\bb_k(\bX,\mu_1,\ldots,\mu_K)
\;\approx\;
\varphi_{\mathrm{int},k}\Bigl(
\bX,\; \mathbf{z}_1(\mu_1), \ldots, \mathbf{z}_K(\mu_K);\;
\theta_{\mathrm{int},k}
\Bigr).
$$
This design enforces permutation invariance within each group while allowing for complex, asymmetric interactions between different groups.
Given a multi-group dataset $\mathcal{D}_{\mathrm{obs}} = \{ (\bX_{t_\ell,m}^{i,k}, \bV_{t_\ell,m}^{i,k}) \}$, we replace the theoretical moments $\mathbf{z}_l(\mu_l)$ with their Monte Carlo estimates based on the empirical measures $\hat{\mu}_{t_\ell,m}^l = \frac{1}{N_l}\sum_{j=1}^{N_l}\delta_{\bX_{t_\ell,m}^{j,l}}$. The empirical population feature for group $l$ becomes:
$$\hat{\mathbf{z}}_{t_\ell,m}^l = \frac{1}{N_l}\sum_{j=1}^{N_l}\varphi_{\mathrm{emb},l}(\bX_{t_\ell,m}^{j,l};\theta_{\mathrm{emb},l}).$$
Consequently, the predicted mean-field drift is given by:$$\hat{\bb}_{\theta,k} = \varphi_{\mathrm{int},k}\!\left(
\bX_{t_\ell,m}^{i,k},\; \hat{\mathbf{z}}_{t_\ell,m}^1, \ldots, \hat{\mathbf{z}}_{t_\ell,m}^K;\; \theta_{\mathrm{int},k}
\right).$$The network parameters $\Theta = \{ \theta_{\mathrm{int},k}, \theta_{\mathrm{emb},k} \}_{k=1}^K$ are jointly learned by minimizing the global trajectory matching error:$$\mathcal{L}(\Theta)
= \frac{1}{M L}\sum_{m=1}^M \sum_{\ell=1}^L
\sum_{k=1}^K \frac{1}{N_k}
\sum_{i=1}^{N_k}
\bigl\|
\bV_{t_\ell,m}^{i,k}
- \hat{\bb}_{\theta,k}
\bigr\|^2.$$The optimization is performed using Adam~\citep{kingma2014adam} within the \texttt{JAX} framework~\citep{jax2018github}, utilizing automatic differentiation to compute gradients through the empirical averages. To ensure scalability and training stability, we employ randomized mini-batching over the agent-time indices.

\subsection{Numerical Result for MG-MVNN}
We evaluate the multi-group framework on a hierarchical system with asymmetric interactions, designed to mimic a stratified communication network. The system comprises $K=3$ groups with population sizes $(N_1, N_2, N_3) = (16{,}000,\, 4{,}000,\, 200)$, labeled as Group 1, Group 2, and Group 3.
Each pair of groups $(k,l)$ interacts through a group-dependent influence kernel:
\[
\phi_{k,l}(r)
= D_{k,l}\,
\exp\!\left(1 - \frac{1}{1 - \left|\nicefrac{r}{R_l}\right|^{10}}\right)
\mathbf{1}_{(-1,1)}\!\left(\frac{r}{R_l}\right),
\]
where $\mathbf{1}_{(-1,1)}$ denotes the indicator function on $(-1,1)$.
The coefficients $D_{k,l}$ quantify the influence strength of group $l$ on group $k$,
while $R_l$ specifies the interaction radius of group $l$.
In this hierarchical setting, the influence matrix $(D_{k,l})$ is given by:
\[
D =
\begin{pmatrix}
5  & 10 & 0 \\
0  & 2  & 5 \\
0  & 0  & 1
\end{pmatrix},
\qquad
(R_1, R_2, R_3) = (1.0,\, 2.5,\, 5.0).
\]
The upper-triangular structure of $D$ enforces a directional flow of influence from higher-ranking groups to lower ones, while the diagonal terms represent intra-group cohesion. Notably, the interaction radii increase with the hierarchy level ($R_3 > R_2 > R_1$), so Group 3 acts over the broadest spatial range, whereas Group 1 interacts most locally.

The training dataset consists of $M=100$ trajectories initialized from random
Gaussian mixtures, as described in the single-group experiments. We evaluate
the model on two held-out scenarios in which the groups are initially
spatially separated. Figures~\ref{fig:multi_species_case3}
and~\ref{fig:multi_species_case4} display the evolution of the group
densities.
The MG-MVNN accurately predicts the hierarchical entrainment process: Group 3 moves independently to form a consensus; Group 2 is pulled towards Group 3; and Group 1 subsequently clusters around Group 2. This sequential locking of dynamics confirms that our multi-group architecture correctly learns the asymmetric causal structure encoded in the influence matrix.
\begin{figure}
    \centering
    \includegraphics[width=0.8\linewidth]{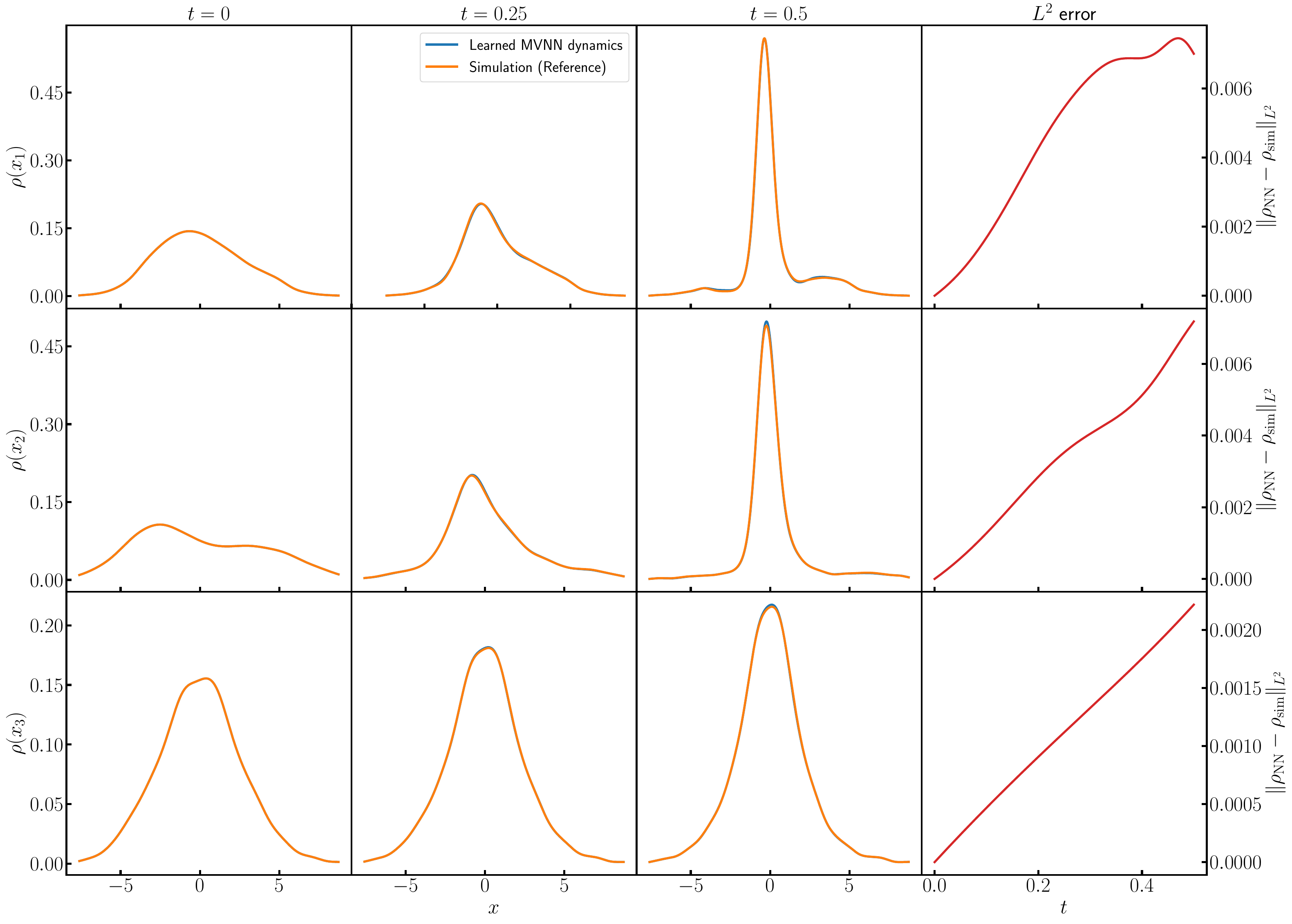}
    \caption{\textbf{Hierarchical dynamics:} Initial condition 1. Evolution of the multi-group system initialized with spatially separated populations. The rows correspond to Group 1, Group 2, and Group 3. The learned model (blue) faithfully reproduces the reference particle dynamics (orange), capturing the directional information flow from Group 3 down to Group 1.}
    \label{fig:multi_species_case3}
\end{figure}

\begin{figure}
    \centering
    \includegraphics[width=0.8\linewidth]{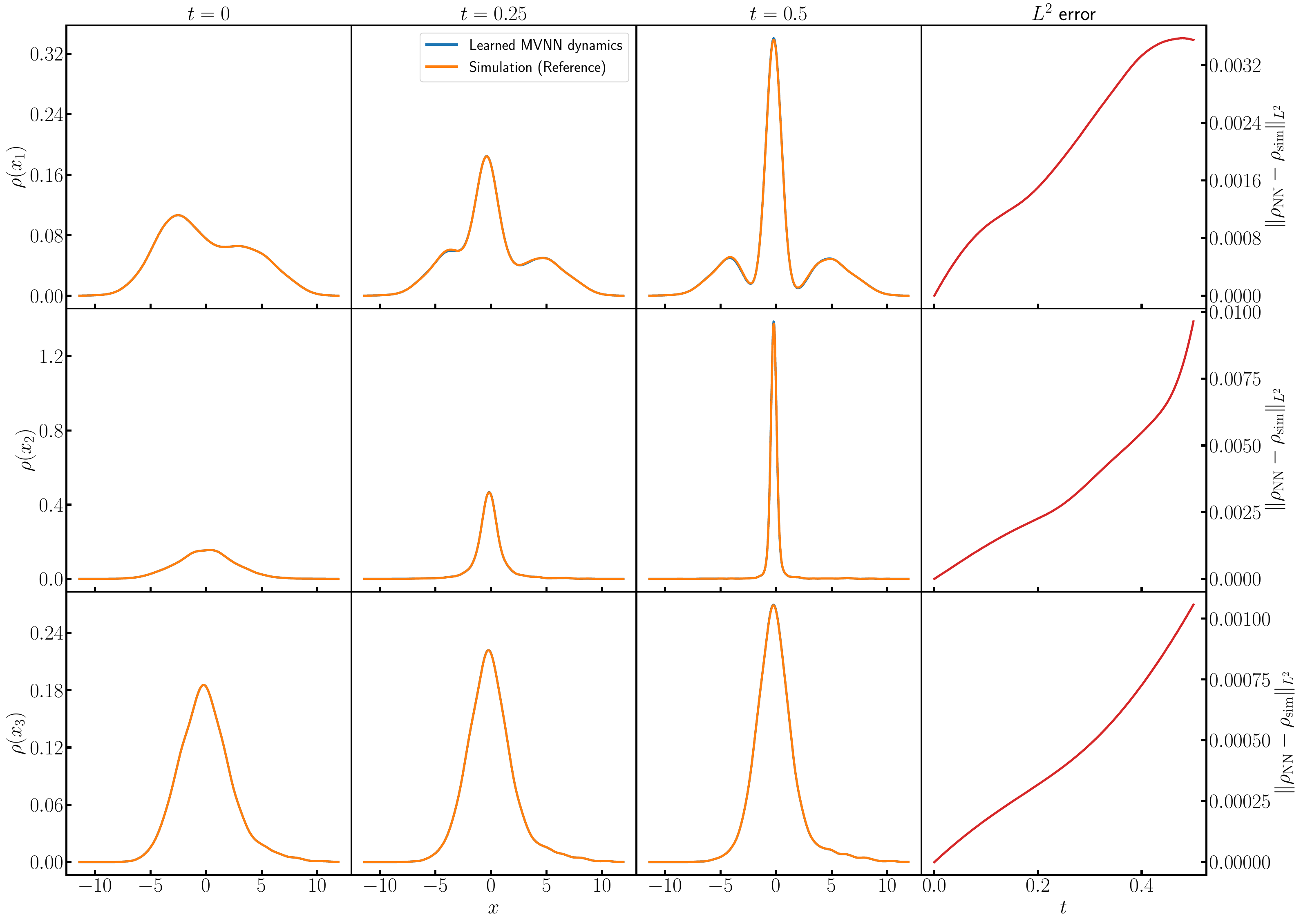}
    \caption{\textbf{Hierarchical dynamics:} Initial condition 2. Evolution of the multi-group system initialized with spatially separated populations. The rows correspond to Group 1, Group 2, and Group 3. The learned model (blue) accurately reproduces the reference particle dynamics (orange), capturing the directional information flow from Group 3 down to Group 1.}
    \label{fig:multi_species_case4}
\end{figure}

\section{Second-Order Dynamics}
The proposed framework can be applied to second-order interacting particle/agent systems. 
Consider the second-order McKean–Vlasov stochastic differential equation for particles involving position $(\bX_t)$ and velocity $(\bV_t)$:
\begin{align*}
    \intd \bX_t &= \bV_t \intd t \\
    \intd \bV_t &= \bb(\bX_t, \bV_t,f_t) \intd t + \sigma(\bX_t, \bV_t, f_t)\intd \mathbf{B}_t,
\end{align*} 
where $f_t = \mathcal L\left(\bX_t, \bV_t\right)$ describes the law of the pair $\left(\bX_t, \bV_t\right)$, and $\mathbf{B}_t$ is a $d\mhyphen$dimensional Wiener process. We can approximate the second-order McKean-Vlasov dynamics by a system of $N$ stochastic differential equations involving the position-velocity pairs $\left(\bX_t^{i,N},\bV_t^{i,N}\right)_{1\leq i\leq N}$, where $N$ is sufficiently large. The initial conditions \(
\left(\bX_0^{i,N}, \bV_0^{i,N}\right)_{1\leq i \leq N} 
\) are independent and identically distributed with law $\mu_0$ and each particle evolves according to: 
\begin{align*}
    \intd \bX^{i,N}_t &= \bV^{i,N}_t \intd t\\
    \intd \bV^{i,N}_t &= \bb(\bX^{i,N}_t,\bV_t^{i,N},\mu^N_t)\intd t + \sigma(\bX^{i,N}_t,\bV_t^{i,N},\mu^N_t)\intd \mathbf{B}^{i,N}_t,
\end{align*}
where $\mathbf{B}_t^{i,N}$ are  $d\mhyphen$dimensional Wiener processes, \(\mu_t^N = \frac{1}{N} \sum_{j=1}^N \delta_{\left(\bX^{j,N}_t,\bV^{j,N}_t\right)}\) denotes the empirical measure of the particle system, and $\delta$ is the Dirac measure. We again focus on the drift term defined here as $\bb:\mathbb R ^d \times \mathbb R ^d \times \mathcal P(\mathbb R^d \times \mathbb R ^d ) \to \mathbb R^d$ and treat the diffusion $\sigma$ as constant or zero. Note that $\mathcal P (\mathbb R^d \times \mathbb R ^d)$ is the space of probability measures on $\mathbb R^d \times \mathbb R ^d$. 

Our observations are positions and velocities along $M$ trajectories at the times $0=t_0<t_1<\cdots <t_L=T$: 
\begin{equation*}
    \left(\bX_{\text{tr}}, \bV_{\text{tr}}\right) :=\left(\{\bX^i_{t_l,m}\}_{i=1,l=0,m=1}^{N,L,M}, \{\bV^i_{t_l,m}\}_{i=1,l=0,m=1}^{N,L,M} \right).
\end{equation*}
The accelerations $\bA^i_{t_l,m}$ are approximated using first-order finite differences:
\begin{equation*}
    \bA^i_{t_l,m} := \dfrac{\bV^i_{t_{l+1},m} -\bV^i_{t_{l},m}}{t_{l+1}-t_l}, \qquad l = 0,\cdots, L-1.
\end{equation*}
For each trajectory $m$, the initial conditions $\left(\bX^i_{0,m},\bV^i_{0,m}\right)$ are independently sampled from the probability measure \(\mu_{0,m}\) for $i=1,\cdots,N$.
Our goal is to learn the mean-field interaction kernel $\bb(\bx,\bv,\mu)$ from the observations.

\subsection{Measure-Valued Neural Network}
Similar to the first-order dynamics, we consider approximating $\bb$ by a composition of two neural networks, $\varphi_\text{emb}$ and $\varphi_\text{int}$:
\begin{equation*}
        \bb_\theta(\bx, \bv, \mu) \approx
\varphi_\text{int}\left(\bx, \bv, \big\langle \varphi_\text{emb}(\cdot,\cdot; \theta_\text{emb}), \mu \big\rangle; \theta_\text{int}\right),
\end{equation*}
where $\mu \in \mathcal{P}_2(\mathbb{R}^d\times \mathbb{R}^d)$, $\varphi_\text{emb}(\cdot,\cdot;\theta_\text{emb}):\mathbb{R}^d\times \mathbb{R}^d\to\mathbb{R}^k$ maps particle positions and velocities to their feature representations, and $\varphi_\text{int}(\cdot,\cdot,\cdot;\theta_\text{int}): \mathbb{R}^d\times\mathbb{R}^d\times\mathbb{R}^k\to \mathbb{R}^d$ learns the drift as a function of the position, velocity, and feature embedding $\langle \varphi_\text{emb},\mu\rangle$. Here $\mathcal{P}_2(\mathbb{R}^d\times\mathbb{R}^d)$ is the space of probability measures on $\mathbb{R}^d\times\mathbb{R}^d$ with finite second moments:
\begin{equation*}
    \mathcal{P}_2(\mathbb{R}^d\times\mathbb{R}^d)=\left\{\mu \text{ probability measure on } \mathbb{R}^d\times\mathbb{R}^d \ \Big| \int_{\mathbb{R}^d\times\mathbb{R}^d} \|\mathbf{x}\|^2+\|\mathbf{y}\|^2d\mu(\mathbf{x},\mathbf{y})<\infty\right\}.
\end{equation*}

We have that for the empirical measure:
\begin{equation*}
    \langle \varphi_\text{emb},\mu_t^N\rangle = \dfrac{1}{N} \sum_{j=1}^N \varphi_\text{emb}\left(\bX_t^{j,N},\bV_t^{j,N};\theta_\text{emb}\right),
\end{equation*}
hence,
\begin{equation*}
    \bb_\theta(\bX,\bV,\mu_t^N)\approx
\varphi_\text{int}\left(\bX, \bV, \frac{1}{N}\sum_{j=1}^N \varphi_\text{emb}(\bX^{j,N}_t, \bV^{j,N}_t; \theta_\text{emb}); \theta_\text{int}\right).
\end{equation*}

The dynamics of the corresponding $N$-particle system are described by:
\begin{align*}
    \intd \bX_t^{\theta,i,N} &= \bV_t^{\theta,i,N} \intd t\\
    \intd \bV_t^{\theta,i,N} &= \bb_\theta\left(\bX_t^{\theta,i,N},\bV_t^{\theta,i,N},\mu_t^{\theta,N}\right) \intd t + \sigma \intd \bB_t^{i,N},
\end{align*}
where $\mu_t^{\theta,N}=\frac{1}{N}\sum_{i=1}^N \delta_{\left(\bX_t^{\theta,i,N},\bV_t^{\theta,i,N}\right)}$ is the empirical measure of the particle system. The limiting McKean-Vlasov stochastic differential equation for a single representative particle with position $\bX^\theta_t$ and velocity $\bV^\theta_t$ is:
\begin{equation}
\begin{split}
\label{eq:mckean-vlasov-second-order}
    \intd \bX_t^\theta &= \bV_t^\theta \intd t\\
    \intd \bV_t^\theta &= \bb_\theta\left(\bX_t^\theta,\bV_t^\theta,f^\theta_t\right)\intd t + \sigma \intd \bB_t\\
    &=\varphi_\text{int}\left(\bX_t^\theta,\bV_t^\theta,\int \varphi_\text{emb}(\bx,\bv)f_t^\theta(\bx,\bv)\intd \bx \intd \bv \right)\intd t + \sigma \intd \bB_t,
\end{split}
\end{equation}
where $f_t^\theta = \mathcal{L}\left(\bX_t^\theta,\bV_t^\theta\right)$ is the law of the position-velocity pair $\left(\bX_t^\theta,\bV_t^\theta\right)$ at time $t$. The Fokker-Planck equation for $(f_t^\theta)_{t\geq0}$ in weak form is given by:
\begin{align*}
    \dfrac{\intd}{\intd t} \langle f_t^\theta,\psi \rangle &= \left\langle f_t^\theta,\bv\cdot\nabla_\bx\psi+\bb_\theta(\bx,\bv,f_t^\theta)\cdot\nabla_\bv\psi+\dfrac{1}{2}\sigma^2\Delta_\bv\psi\right\rangle\\
    &= \left\langle f_t^\theta,\bv\cdot\nabla_\bx \psi+\varphi_\text{int}\left(\bx,\bv,\int\varphi_\text{emb}(\bx,\bv)f_t^\theta(\bx,\bv)\intd \bx \intd \bv\right)\cdot\nabla_\bv\psi+\dfrac{1}{2}\sigma^2\Delta_\bv \psi\right\rangle,
\end{align*}
for any smooth test function $\psi:\mathbb{R}^d\times\mathbb{R}^d\to \mathbb{R}$ with compact support.

\subsection{Learning Objective and Optimization}
Given the observations
\begin{equation*}
    \mathcal{D}_\text{obs} = \left\{\left(\bX^i_{t_l,m},\bV^i_{t_l,m},\bA^i_{t_l,m}\right)\right\}^{N,L,M}_{i=1,l=0,m=1},
\end{equation*}
the predicted mean-field drift from the MVNN is:
\begin{equation*}
    \hat{\bb}_\theta\left(\bX^i_{t_l,m},\bV^i_{t_l,m},\hat{\mu}_{t_l,m}\right) = \varphi_\text{int}\left(\bX^i_{t_l,m}, \bV^i_{t_l,m}, \frac{1}{N}\sum_{j=1}^N \varphi_\text{emb}(\bX^j_{t_l,m}, \bV^j_{t_l,m}; \theta_\text{emb}); \theta_\text{int}\right).
\end{equation*}

We learn the model parameters $\theta=(\theta_\text{int},\theta_\text{emb})$ by minimizing the differences between the predicted and observed trajectories. For deterministic trajectory data, the MVNN is trained by minimizing the mean-squared error:
\[
\mathcal J_{\mathrm{det}}(\theta) = \frac{1}{MLN}\sum_{m=1}^M\sum_{l=0}^{L-1}\sum_{i=1}^N \left\|\bA_{t_l,m}^{i}-\hat{\bb}_{\theta}\left(\bX_{t_l,m}^i,\bV_{t_l,m}^i,\hat{\mu}_{t_l,m}\right)\right\|^2.
\]
Similar to the first-order case, the same objective function can be derived from the discretized path likelihood for stochastic systems with a nondegenerate diffusion coefficient. In the case $\sigma(\bx)\equiv \sigma \mathbf{I}$, the negative log-likelihood has the same minimizer as $\mathcal J_{\mathrm{det}}$. Let $\mathbb{P}^\theta$ be the measure induced by the solution $\left(\bX^\theta_s,\bV^\theta_s\right)_{s\in[0,t]}$ of the McKean-Vlasov SDE \eqref{eq:mckean-vlasov-second-order}. By Girsanov's theorem, assuming the usual regularity conditions, we have the log-likelihood function
\begin{align*}
    \mathcal{L}_t(\theta):=\log\dfrac{\intd \mathbb{P}^\theta}{\intd \mathbb{P}^\bb} = &\int_0^t \left\langle\hat{\bb}_\theta\left(\bX_s,\bV_s,\mu^\theta_s\right)-\bb\left(\bX_s,\bV_s,\mu_s\right),\left(\sigma(\bX_s,\bV_s)\sigma(\bX_s,\bV_s)^\top\right)^{-1}\intd \bV_s\right\rangle \\
    &-\dfrac{1}{2} \int_0^t \left(\|\sigma(\bX_s,\bV_s)^{-1}\hat{\bb}_\theta(\bX_s,\bV_s,\mu^\theta_s)\|^2-\|\sigma(\bX_s,\bV_s)^{-1}\bb(\bX_s,\bV_s,\mu_s)\|^2\right) \intd s.
\end{align*}
We assume $\sigma(\bx,\bv)\equiv \sigma \mathbf{I}$, where $\sigma > 0$ is a constant, and approximate the stochastic integral using the Euler-Maruyama method to get the discrete log-likelihood:
\begin{align*}
    \hat{\mathcal{L}}(\theta) = &\dfrac{1}{MLN}\sum\limits_{m=1}^M\sum\limits_{l=0}^{L-1}\sum\limits_{i=1}^N\bigg[\left\langle\hat{\bb}_\theta(\bX^i_{t_l,m},\bV^i_{t_l,m},\hat{\mu}_{t_l,m})-\bb(\bX^i_{t_l,m},\bV^i_{t_l,m},\hat{\mu}_{t_l,m}),(\sigma\sigma^\top)^{-1} \Delta\bV^i_{t_l,m}\right\rangle\\
    &-\dfrac{1}{2}\left(\|\sigma^{-1}\hat{\bb}_\theta(\bX^i_{t_l,m},\bV^i_{t_l,m},\hat{\mu}_{t_l,m})\|^2-\|\sigma^{-1}\bb(\bX^i_{t_l,m},\bV^i_{t_l,m},\hat{\mu}_{t_l,m})\|^2\right)\Delta t\bigg],
\end{align*}
where $\hat{\mu}_{t_l,m} = \frac{1}{N}\sum_{j=1}^N \delta_{\left(\bX^j_{t_l,m},\bV^j_{t_l,m}\right)}$ is the empirical measure for trajectory $m$ at time $t_l$, and $\Delta \bV^i_{t_l,m}:=\bV^i_{t_{l+1},m}-\bV^i_{t_l,m}$. Dropping terms independent of $\theta$ yields:
\begin{equation*}
    \hat{\mathcal{L}}(\theta) = -\dfrac{\Delta t}{2\sigma^2} \dfrac{1}{MLN} \sum\limits_{m=1}^M\sum\limits_{l=0}^{L-1}\sum\limits_{i=1}^N \left\|\hat{\bb}_\theta(\bX^i_{t_l,m},\bV^i_{t_l,m},\hat{\mu}_{t_l,m})-\bA^i_{t_l,m}\right\|^2+C.
\end{equation*}
Hence, maximizing the log-likelihood $\hat{\mathcal{L}}(\theta)$ is equivalent to minimizing the mean-squared error loss:
\begin{equation}
\label{eq:second-order-MSE}
    \theta^* \in \arg\min_\theta \dfrac{1}{MLN}\sum\limits_{m=1}^M\sum\limits_{l=0}^{L-1}\sum\limits_{i=1}^N\left\|\bA^i_{t_l,m}-\hat{\bb}_\theta(\bX^i_{t_l,m},\bV^i_{t_l,m},\hat{\mu}_{t_l,m})\right\|^2.
\end{equation}
The numerical optimization of \eqref{eq:second-order-MSE} is performed using Adam~\citep{kingma2014adam} with mini-batches. 

\subsection{Numerical Result}
We present numerical experiments that validate the accuracy of our MVNN model for second-order systems and show its ability to generalize to unseen initial position configurations. We compare the dynamics from the learned mean-field model to the true dynamics from reference particle simulations.
\subsubsection{Second-Order Attraction-Repulsion Swarming Model}
We consider the second-order attraction-repulsion swarming model in two dimensions:
\[
\ddot{\bX}^i_t 
= \frac{1}{N}\sum_{j=1}^N 
\phi\left(\|\bX^j_t - \bX^i_t\|\right)(\bX^j_t - \bX^i_t),
\qquad i=1,\ldots,N,
\]
where $\bX^i_t \in \mathbb{R}^2$ denotes the position of agent $i$, and \[
\phi(r) = c_\text{rep} \exp(-(r/\ell_\text{rep})^2) - c_\text{att}\exp(-(r/\ell_\text{att})^2) 
\] 
with $c_\text{rep} = 1.0 $, $\ell_\text{rep} = 0.5$, $c_\text{att} = 0.7$, and $\ell_\text{att}=2.0$. 
We use the forward Euler method to simulate $100$ trajectories with $N=16{,}000$ agents and $200$ steps with time step $\Delta t=10^{-2}$. We construct our training set with $80\%$ of our initial positions sampled from a noisy annulus with random radius and width and the remaining $20\%$ of the initial positions generated from a noisy double annulus with identical random radius and width. In particular, for each trajectory $m\in\{1,\cdots,80\}$, the initial position of agent $i$ is sampled using polar coordinates with additive Gaussian noise:
\[\begin{aligned}
\Theta_i &\sim \mathcal U(0,2\pi), \quad
\rho_i \sim \mathcal U\bigl(R_0-\tfrac{W}{2},\,R_0+\tfrac{W}{2}\bigr), \\
\bX^i_0 &= \rho_i\bigl(\cos\Theta_i,\sin\Theta_i\bigr) + \varepsilon_i, \quad \text{with} \quad
\varepsilon_i \sim \mathcal N(\mathbf 0,\,\sigma_0^2 \mathbf{I}_2),
\end{aligned}
\]
where $\mathbf{I}_2$ is the $2\times 2$ identity matrix. For each trajectory $m\in\{81,\cdots,100\}$, the initial position of half the agents, i.e. agents $i\in\{1,\cdots,8000\}$, is generated from:
\[\begin{aligned}
\Theta_i &\sim \mathcal U(0,2\pi), \quad
\rho_i \sim \mathcal U\bigl(R_0-\tfrac{W}{2},\,R_0+\tfrac{W}{2}\bigr), \\
\bX^i_0 &= (0.5,0.5) + \rho_i\bigl(\cos\Theta_i,\sin\Theta_i\bigr) + \varepsilon_i, \quad \text{with} \quad
\varepsilon_i \sim \mathcal N(\mathbf 0,\,\sigma_0^2 \mathbf{I}_2),
\end{aligned}
\]
and the initial position of agents $i\in\{8001,\cdots,16000\}$ takes the value:
\[\begin{aligned}
\bX^i_0 &= (-0.5,-0.5) + \rho_{i-8000}\bigl(\cos\Theta_{i-8000},\sin\Theta_{i-8000}\bigr) + \varepsilon_{i-8000}.
\end{aligned}
\]
All initial velocities for the training samples are drawn from $\mathcal{N}(\mathbf{0},0.25\mathbf{I}_2)$. Figure \ref{fig:ring_x_gaussian_v} displays the evolution of the particle system with initial position generated from a ring distribution and initial velocity drawn from a Gaussian distribution, with a comparison between the true dynamics and the learned mean-field dynamics. Figure \ref{fig:2ring_x_gaussian_v} shows the dynamics of the system with initial position drawn from a double ring distribution and Gaussian initial velocity, again comparing the true versus learned dynamics. 

\begin{figure}[h!]
    \centering \subfigure[Upper row: true positions, lower row: learned positions]{\includegraphics[width=0.6\linewidth]{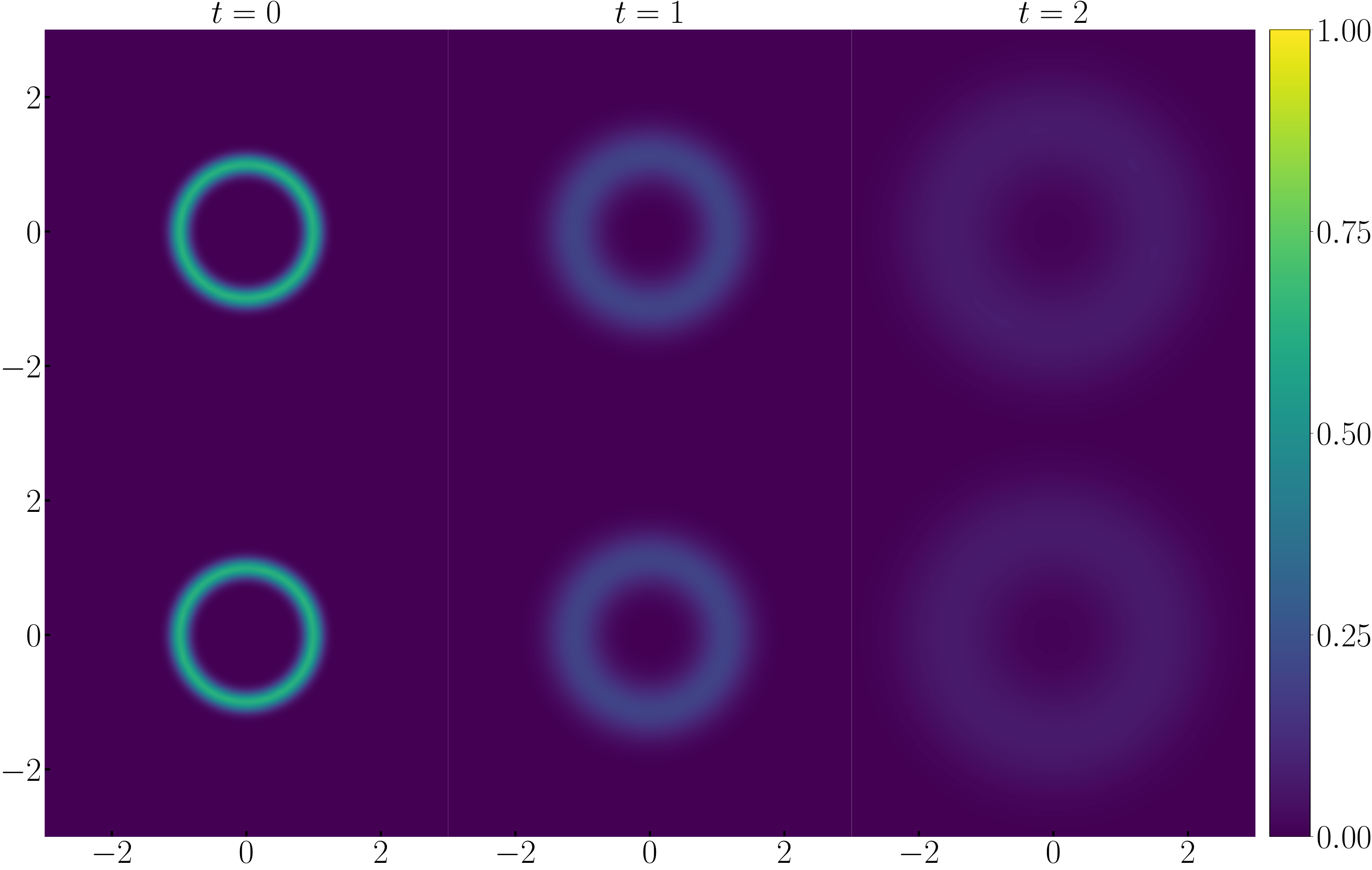}}\\
    \subfigure[Upper row: true velocities, lower row: learned velocities]{\includegraphics[width=0.6\linewidth]{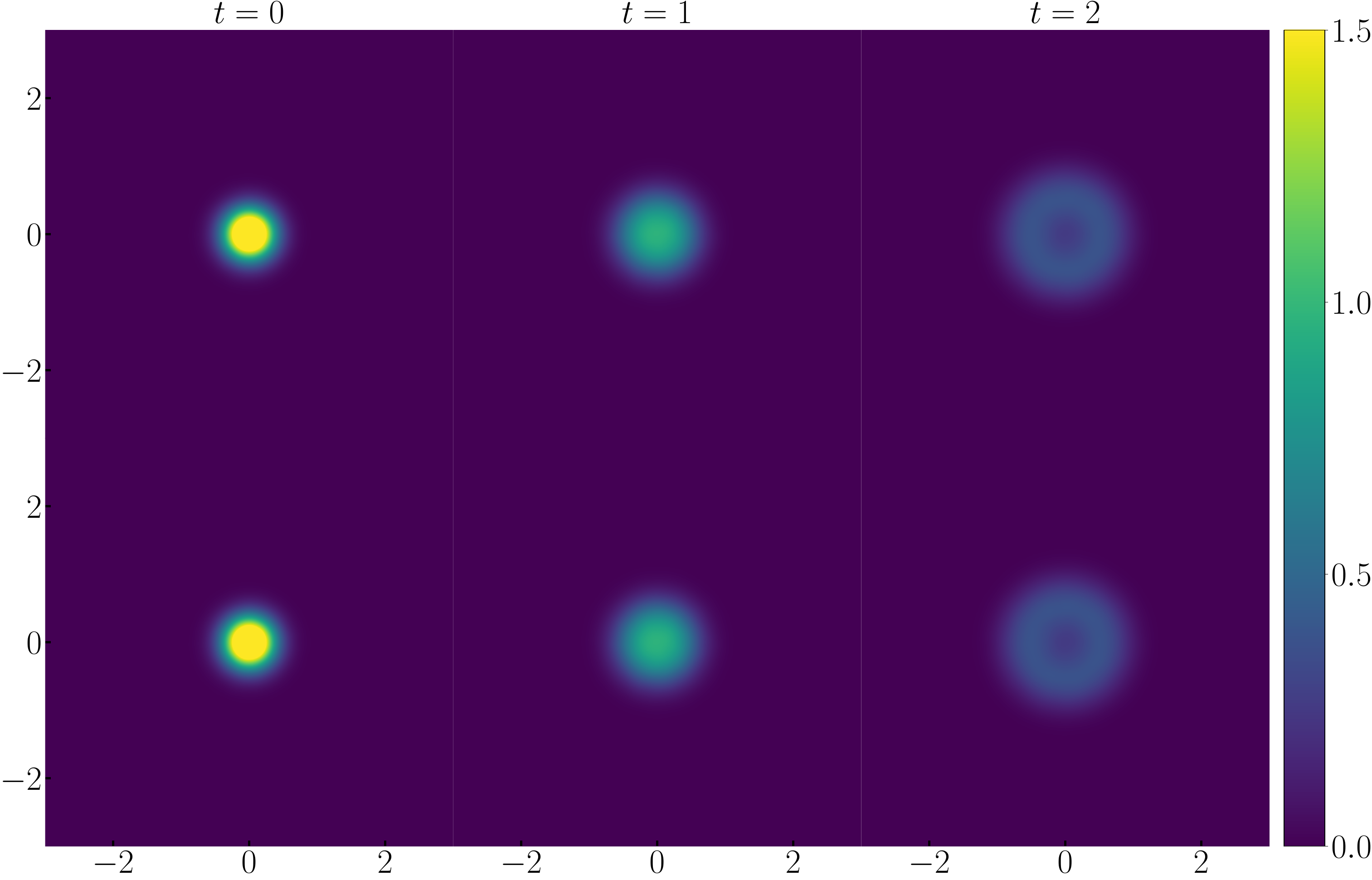}}
    \caption{\textbf{Second-order attraction-repulsion}:  Evolution of the second-order attraction-repulsion model initialized with ring-shaped initial position and Gaussian initial velocity. The upper rows display ground truth particle positions and velocities, while the lower rows show the positions and velocities predicted by the learned MVNN.}    \label{fig:ring_x_gaussian_v}
\end{figure}

\begin{figure}[h!]
    \centering \subfigure[Upper row: true positions, lower row: learned positions]{\includegraphics[width=0.6\linewidth]{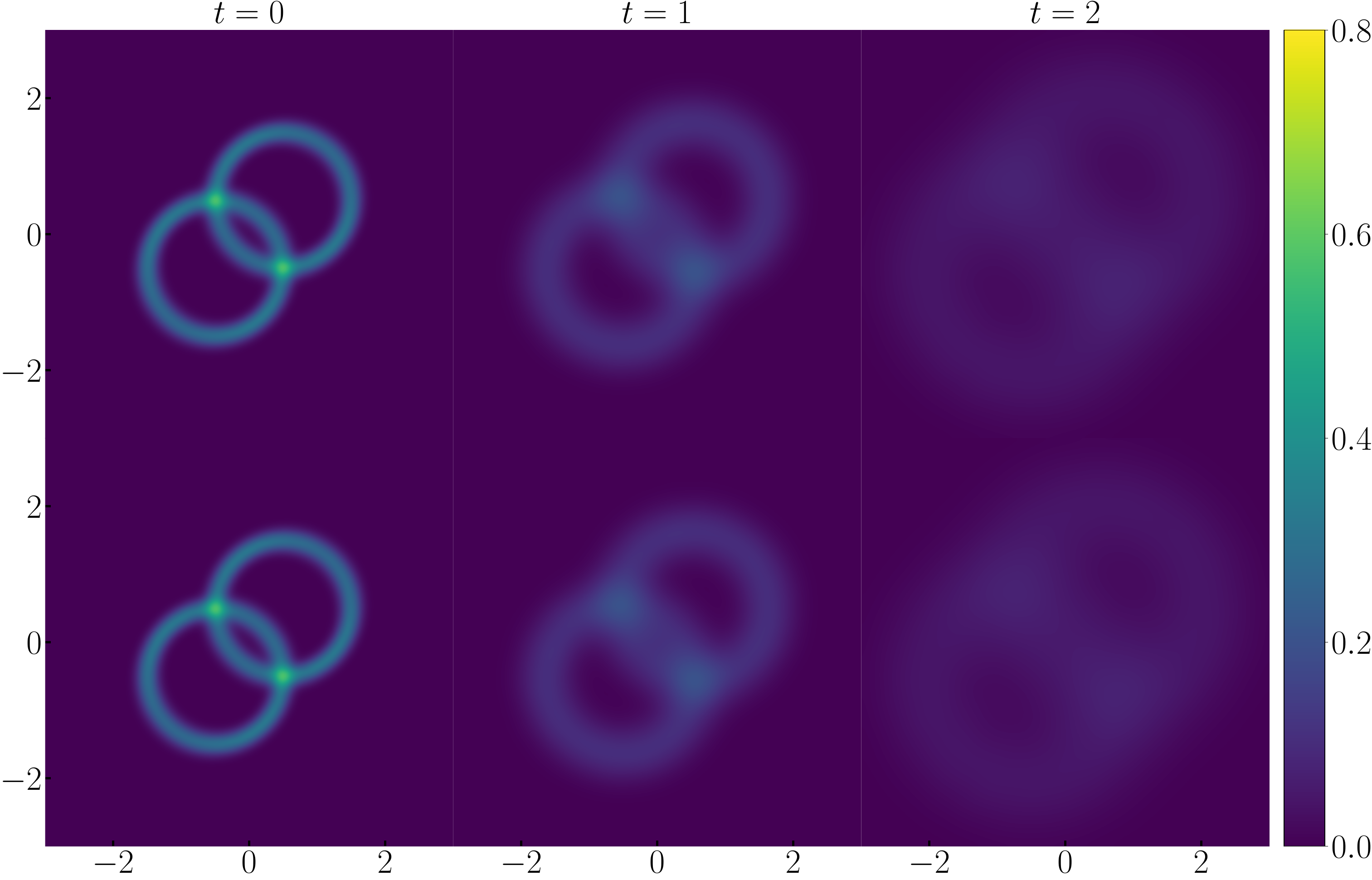}}\\
    \subfigure[Upper row: true velocities, lower row: learned velocities]{\includegraphics[width=0.6\linewidth]{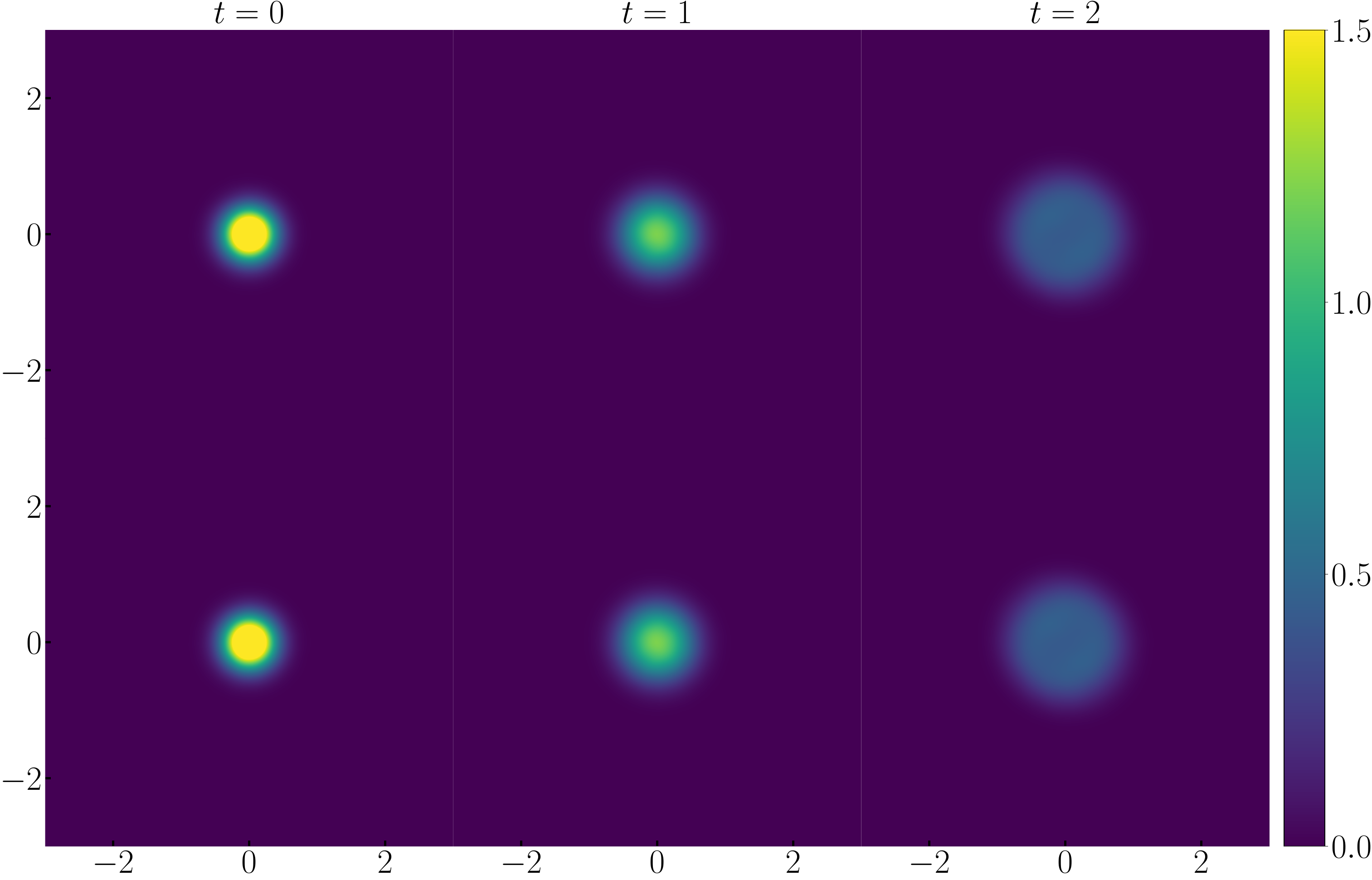}}
    \caption{\textbf{Second-order attraction-repulsion}: Evolution of the second-order attraction-repulsion model initialized with double ring-shaped initial position and Gaussian initial velocity. The upper rows display ground truth particle positions and velocities, while the lower rows show the positions and velocities predicted by the learned MVNN.}
    \label{fig:2ring_x_gaussian_v}
\end{figure}

We also evaluate the generalization capability of the learned MVNN by testing the model on initial position distributions that differ from those in the training set. We consider test cases with initial positions generated from a uniform disk and a binary distribution with heterogeneous density. We still sample the initial velocities from a Gaussian distribution. The true evolution of the particle system compared to the evolution predicted by the learned MVNN is shown in Figures \ref{fig:disk_x_gaussian_v} and \ref{fig:binary_x_gaussian_v}. 

\begin{figure}[h!]
    \centering \subfigure[Upper row: true positions, lower row: learned positions]{\includegraphics[width=0.6\linewidth]{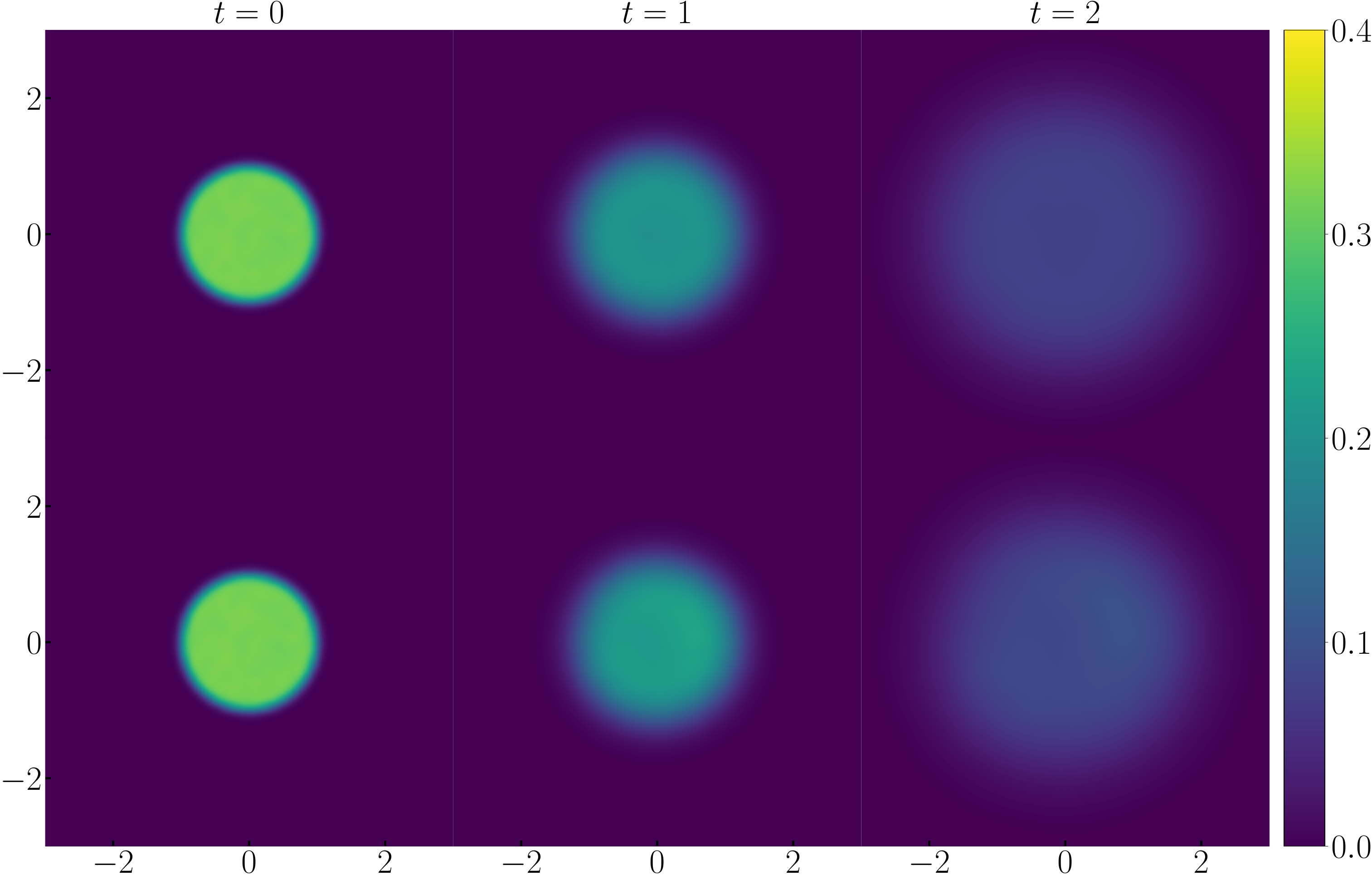}}\\
    \subfigure[Upper row: true velocities, lower row: learned velocities]{\includegraphics[width=0.6\linewidth]{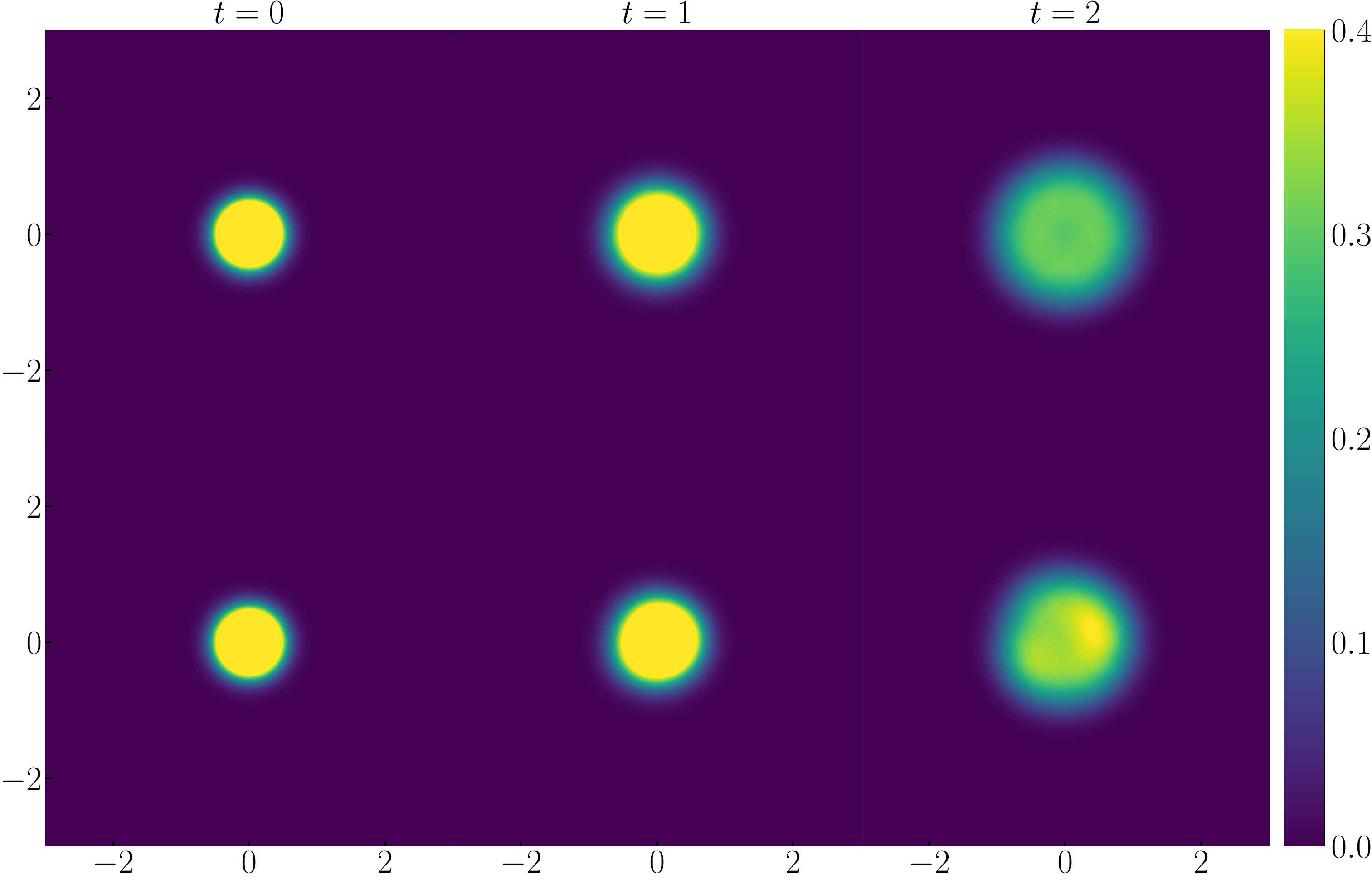}}
    \caption{\textbf{Second-order attraction-repulsion}: Evolution of the second-order attraction-repulsion model initialized with disk-shaped initial position and Gaussian initial velocity. The upper rows display ground truth particle positions and velocities, while the lower rows show the positions and velocities predicted by the learned MVNN.}
    \label{fig:disk_x_gaussian_v}
\end{figure}

\begin{figure}[h!]
    \centering \subfigure[Upper row: true positions, lower row: learned positions]{\includegraphics[width=0.66\linewidth]{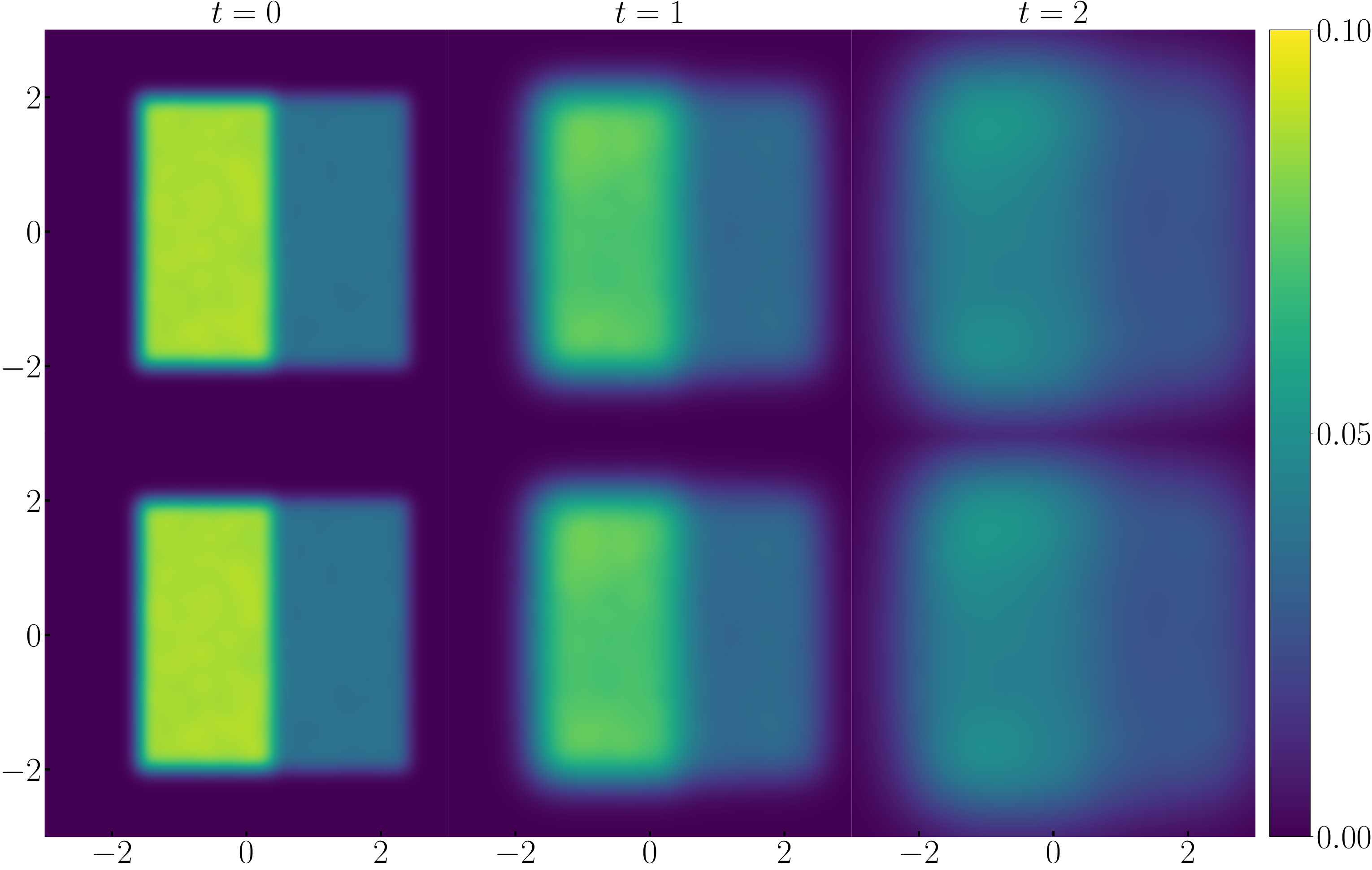}}\\
    \subfigure[Upper row: true velocities, lower row: learned velocities]{\includegraphics[width=0.66\linewidth]{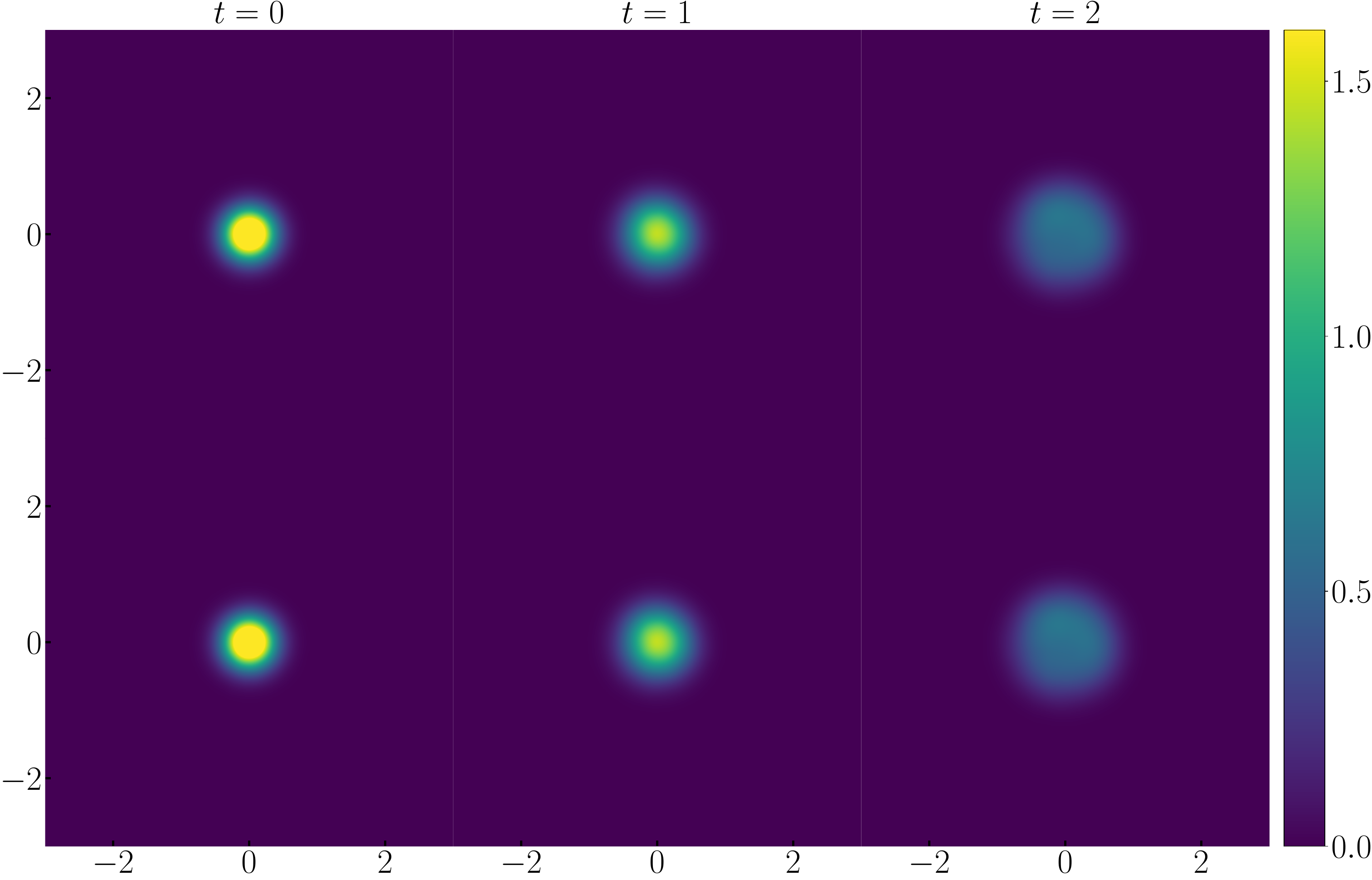}}
    \caption{\textbf{Second-order attraction-repulsion}: Evolution of the second-order attraction-repulsion model initialized with binary asymmetric initial position (high density left, low density right) and Gaussian initial velocity. The upper rows display ground truth particle positions and velocities, while the lower rows show the positions and velocities predicted by the learned MVNN.}
    \label{fig:binary_x_gaussian_v}
\end{figure}

\subsubsection{Second-Order Cucker-Smale Model}
We also consider the two-dimensional second-order Cucker-Smale model, which involves velocity alignment rather than distance-based interactions: \[
\ddot{\bX}^i_t 
= \frac{1}{N}\sum_{j=1}^N 
\phi\left(\|\bX^j_t - \bX^i_t\|\right)(\dot{\bX}^j_t - \dot{\bX}^i_t),
\qquad i=1,\ldots,N,
\]
where $\bX^i_t \in \mathbb{R}^2$ denotes the position of agent $i$, and \[
\phi(r) = c_\text{rep} \exp(-(r/\ell_\text{rep})^2) - c_\text{att}\exp(-(r/\ell_\text{att})^2) 
\] 
with $c_\text{rep} = 1.0 $, $\ell_\text{rep} = 0.5$, $c_\text{att} = 0.7$, and $\ell_\text{att}=2.0$. For our simulations, we use forward Euler to generate $100$ trajectories for $N=16{,}000$ particles with $200$ time steps and $\Delta t = 10^{-2}$. Our training set is formed by sampling $80\%$ of the initial positions from a  Gaussian distribution with random scaling and $20\%$ of the initial positions from a two-component Gaussian mixture that is also randomly scaled. Specifically, for trajectories $m\in\{1,\cdots,80\}$, the initial positions are generated from $\mathcal{N}(\mu,(s_m\sigma)^2\mathbf{I}_2)$, where $s_m \sim \mathcal{U}(s_\text{min},s_\text{max})$. 
For trajectories $m\in\{81,\cdots,100\}$, the initial position for half the agents is drawn from $\mathcal{N}(\mu_1,(s_{m,1}\sigma_1)^2\mathbf{I}_2)$ with $s_{m,1} \sim \mathcal{U}(s_\text{min},s_\text{max})$ and the initial position for the other half is sampled from $\mathcal{N}(\mu_2,(s_{m,2}\sigma_2)^2\mathbf{I}_2)$, where $s_{m,2} \sim \mathcal{U}(s_\text{min},s_\text{max})$. All initial training velocities are drawn from $\mathcal{N}(\mathbf{0},0.25\mathbf{I}_2)$. Figure \ref{fig:gaussian_x_gaussian_v} visualizes the dynamics of the particle system with Gaussian initial position and velocity, and Figure \ref{fig:double_gaussian_x_gaussian_v} shows the evolution of the system with initial position generated from a two-component Gaussian mixture and initial velocity drawn from a Gaussian distribution. 

\begin{figure}[h!]
    \centering \subfigure[Upper row: true positions, lower row: learned positions]{\includegraphics[width=0.55\linewidth]{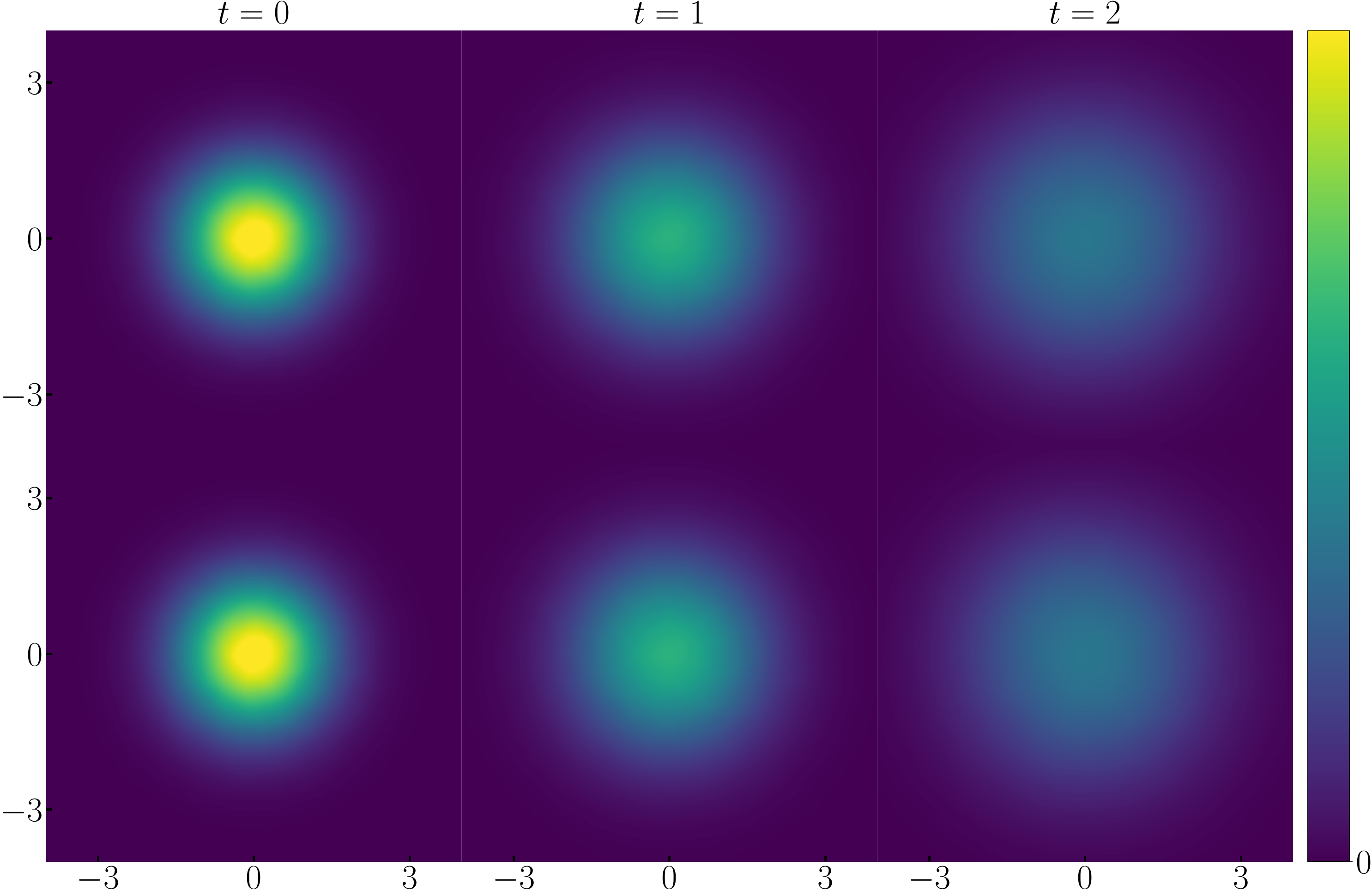}}\\
    \subfigure[Upper row: true velocities, lower row: learned velocities]{\includegraphics[width=0.55\linewidth]{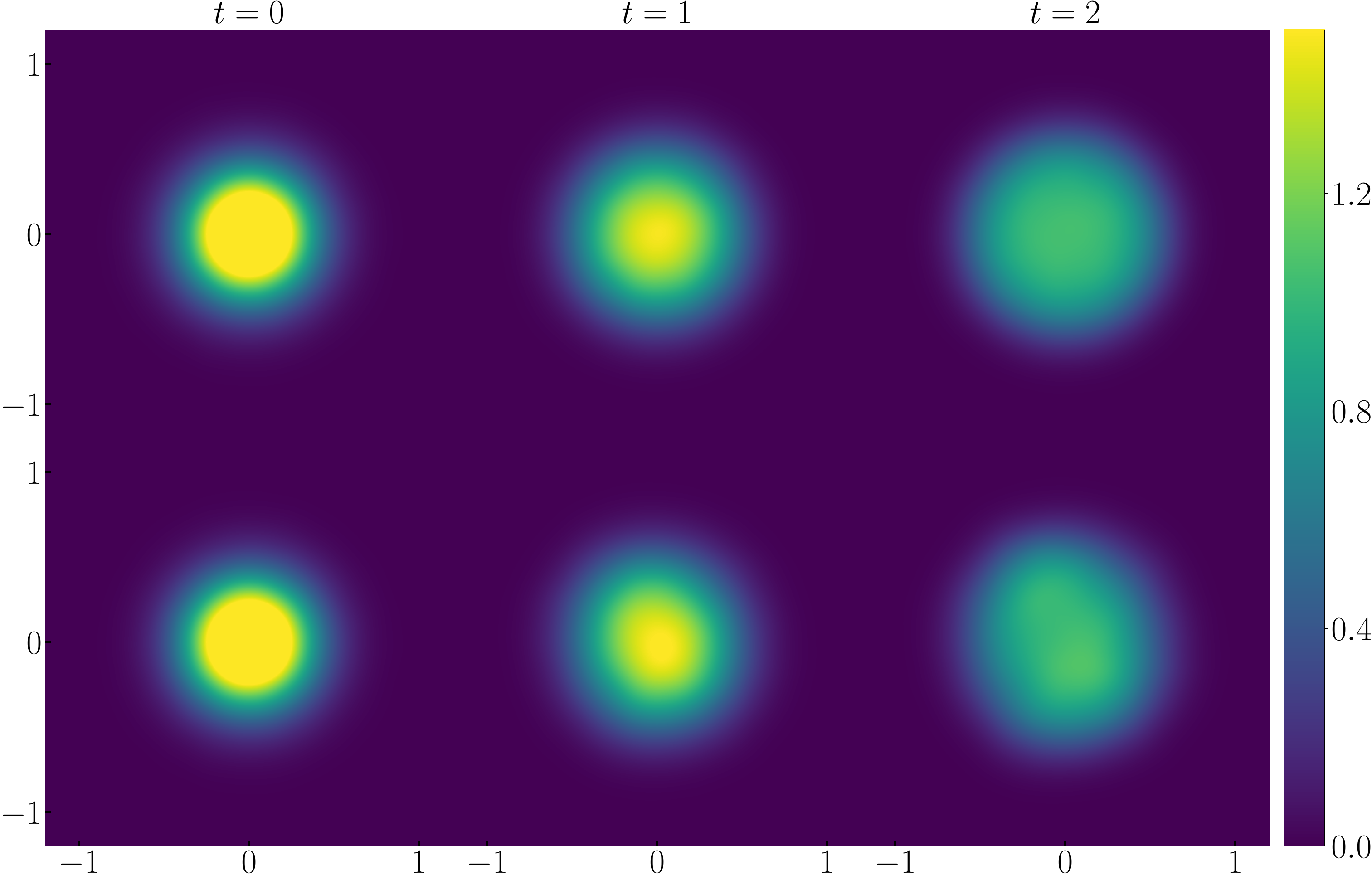}}
    \caption{\textbf{Second-order Cucker-Smale}: Evolution of the second-order Cucker-Smale model initialized with Gaussian initial position and Gaussian initial velocity. The upper rows display ground truth particle positions and velocities, while the lower rows show the positions and velocities predicted by the learned MVNN.}
    \label{fig:gaussian_x_gaussian_v}
\end{figure}

\begin{figure}[h!]
    \centering \subfigure[Upper row: true positions, lower row: learned positions]{\includegraphics[width=0.55\linewidth]{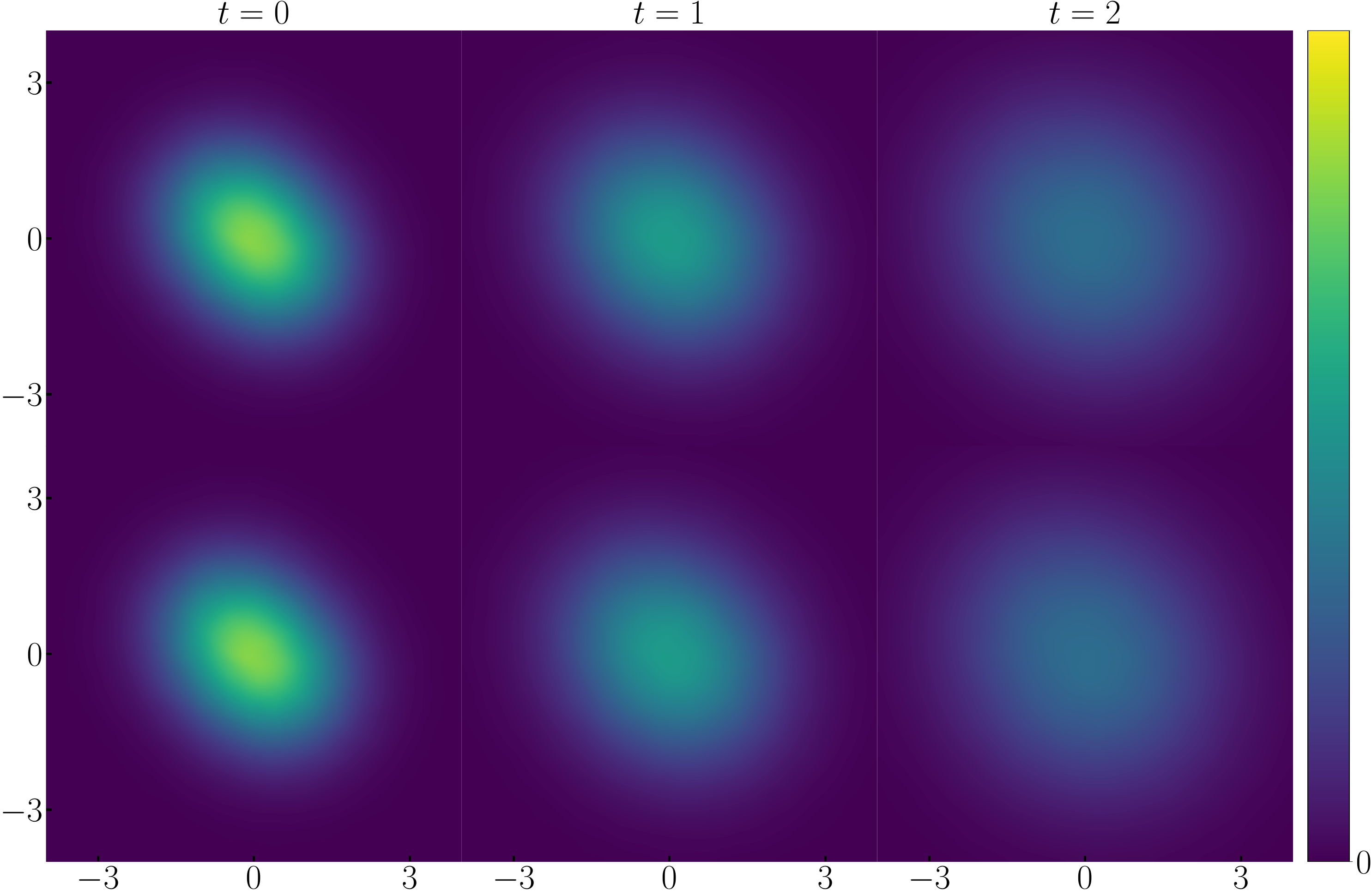}}\\
    \subfigure[Upper row: true velocities, lower row: learned velocities]{\includegraphics[width=0.55\linewidth]{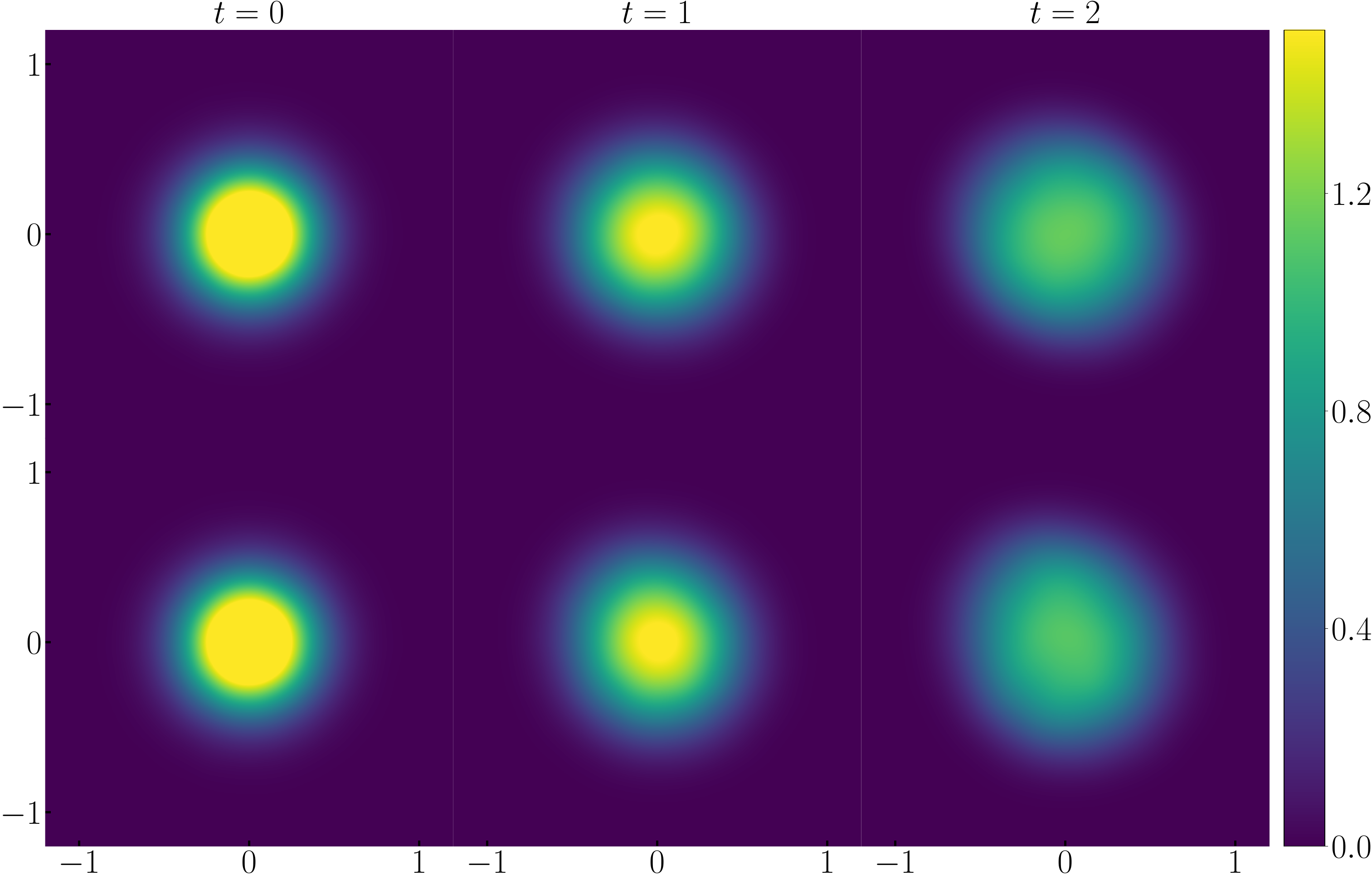}}
    \caption{\textbf{Second-order Cucker-Smale}: Evolution of  the second-order Cucker-Smale model initialized with a two-component Gaussian mixture initial position and Gaussian initial velocity. The upper rows display ground truth particle positions and velocities, while the lower rows show the positions and velocities predicted by the learned MVNN.}
    \label{fig:double_gaussian_x_gaussian_v}
\end{figure}

\section{Discussion}
We proposed a measure-valued neural network (MVNN) framework for estimating mean-field dynamics directly from particle-level observations of interacting agent systems. The approach generalizes neural network architectures to operate on probability measures, enabling the learning of measure-dependent drift terms that arise in the mean-field limit of large-scale interacting particle systems. 
Unlike mean-field approximations, our method provides a data-driven alternative that can efficiently infer complex mean-field interactions from empirical trajectory data without requiring explicit functional forms. We demonstrated through several representative examples, including Motsch-Tadmor dynamics, Cucker–Smale flocking, and hierarchical multi-group systems, that the proposed framework accurately captures the emergent collective behavior and provides empirical evidence of transfer across initial-distribution geometries not represented in the training data, while the governing dynamics
and physical parameters are held fixed. In particular, the multi-group extension of MVNN successfully recovers asymmetric and hierarchical communication structures among heterogeneous groups of agents. We also show that our framework can be extended to second-order systems. Comparisons with Gaussian process models demonstrate that our approach achieves greater efficiency and improved accuracy in predicting complex dynamics.

Our model can be viewed as a weak-form variant of operator learning \citep{schaeffer2017sparse, messenger2021weak}, conceptually related to but distinct from existing architectures such as DeepONet~\citep{lu2021learning} and the Fourier Neural Operator~\citep{kovachki2023neuraloperator}. DeepONet learns operators by mapping discrete pointwise evaluations of input functions to outputs, similar to finite difference schemes that rely on pointwise discretizations of differential operators. FNO, in contrast, represents input functions through their Fourier spectra, resembling spectral methods that approximate operators in a global frequency basis. Our proposed approach can be viewed as using the weak formulation to learn operators through integral constraints with test functions rather than pointwise values, akin to finite element methods in numerical analysis.

In future work, we note that the mean-field approximation is not always sufficient to capture the full range of interactions in complex systems. For instance, in plasma physics and dense particle systems, higher-order correlations play a significant role, and models such as the BBGKY hierarchy or kinetic closures are required to describe the coupled evolution of multi-particle distributions. Extending the proposed framework to learn higher-order interaction terms or reduced representations of two-particle distributions could therefore provide a promising pathway toward the data-driven discovery of beyond–mean-field dynamics. Such extensions would help bridge the gap between microscopic simulations and macroscopic kinetic models, potentially enabling efficient modeling of systems in which mean-field assumptions break down. This direction may also contribute to the development of foundation models for partial differential equations, capable of capturing multiscale structure and transferring across a broad class of dynamical systems.

\section*{Acknowledgments}
We would like to thank Adrien Weihs for the insightful discussions. This work used Anvil at Purdue through allocation MTH260007 from the Advanced Cyberinfrastructure Coordination Ecosystem: Services \& Support (ACCESS) program\cite{boerner2023access}, which is supported by U.S. National Science Foundation grants \#2138259, \#2138286, \#2138307, \#2137603, and \#2138296.
\bibliographystyle{unsrt}  
\bibliography{references}  

\appendix
\section{Proof of Proposition \ref{prop:well-possness}}\label{app:proof-wellposedness}
\begin{proof}
    The proof is based on Theorem 1.7 in \cite{carmona2016lectures}. For simplicity,
    we denote the embedding and interaction networks by $\varphi_{\mathrm{e}}$ and $\varphi_{\mathrm{int}}$, respectively.
    We only need to show that for all $\bx,\bx'\in \mathbb R^d$ and $\mu,\mu'\in \mathcal P$, we have:
    \[
    |\bb_\theta(\bx,\mu) - \bb_\theta (\bx',\mu')|\leq C(|\bx-\bx'|+W_2(\mu,\mu')),
    \]
    where $W_2$ denotes the Wasserstein-2 distance. Notice that, by the triangle inequality, we have: 
    \[
    |\bb_\theta(\bx,\mu) - \bb_\theta (\bx',\mu')| \leq \left|\bb_\theta(\bx,\mu) - \bb_\theta (\bx,\mu')\right| +|\bb_\theta(\bx,\mu') - \bb_\theta (\bx',\mu')|.
    \]
The first term is bounded by:  
    \[
    \begin{aligned}
       \left |\bb_\theta(\bx,\mu) - \bb_\theta (\bx,\mu')\right| = &\left|\varphi_{\mathrm{int}}\left(\bx,\int  \varphi_{\mathrm{e}}(\by ) \mu (\intd\by)\right) - \varphi_{\mathrm{int}}\left(\bx,\int  \varphi_{\mathrm{e}}(\by ) \mu' (\intd\by)\right)\right| \\
       \leq&  C_i\left| 
       \int  \varphi_{\mathrm{e}}(\by ) \mu (\intd\by) - \int  \varphi_{\mathrm{e}}(\by ) \mu' (\intd\by)
       \right|.
    \end{aligned}
    \]
    
For the second term, keeping the measure argument fixed at $\mu'$ and using the Lipschitz continuity of $\varphi_{\mathrm{int}}$ in its first variable, we obtain:
    \[
    \begin{aligned}
    \left|\bb_\theta(\bx,\mu') - \bb_\theta (\bx',\mu')\right| =& \left|\varphi_{\mathrm{int}}\left(\bx,\int  \varphi_{\mathrm{e}}(\by )  \mu' (\intd\by)\right) - \varphi_{\mathrm{int}}\left(\bx',\int  \varphi_{\mathrm{e}}(\by ) \mu' (\intd\by)\right)\right| \\
    \leq & C_i\left|\bx-\bx'\right|.
    \end{aligned}
    \]
 
For the first term, consider any coupling $\pi\in\Pi(\mu,\mu')$, whose marginals are $\mu$ and $\mu'$, i.e.
\[
\begin{aligned}
\int_{\mathbb{R}^d\times\mathbb{R}^d} \phi(x)\,\pi(\mathrm{d}x,\mathrm{d}y)
=\int_{\mathbb{R}^d} \phi(x)\,\mu(\mathrm{d}x),\\
\int_{\mathbb{R}^d\times\mathbb{R}^d} \psi(y)\,\pi(\mathrm{d}x,\mathrm{d}y)
=\int_{\mathbb{R}^d} \psi(y)\,\mu'(\mathrm{d}y).
\end{aligned}
\]
For all $\pi$, we have
\[
\begin{aligned}
    \left| 
       \int  \varphi_{\mathrm{e}}(\by ) \mu (\intd\by) - \int  \varphi_{\mathrm{e}}(\by ) \mu' (\intd\by)
       \right| =&
       \left|
        \int
        \left(\varphi_{\mathrm{e}}(\by )-
        \varphi_{\mathrm{e}}(\by' )\right)\pi(\intd \by,\intd \by')
       \right|\\
       \leq & 
       \left( \int\left(
        \varphi_{\mathrm{e}}(\by )-
        \varphi_{\mathrm{e}}(\by' )\right)^2\pi(\intd \by,\intd \by')\right)^{\frac{1}{2}}
       \\
       \leq & C_e \left(\int\left(
        \by -
        \by' \right)^2\pi(\intd \by,\intd \by')\right)^{\frac{1}{2}}.
\end{aligned}
\]
Taking the $\pi$ that minimizes the $\int\left|
        \by -
        \by' \right|^2\pi(\intd \by,\intd \by')$ yields:
\[
\begin{aligned}
    \left |\bb_\theta(\bx,\mu) - \bb_\theta (\bx,\mu')\right| \leq C_i C_e W_2(\mu,\mu').
\end{aligned}
\]
Therefore, we obtain a Lipschitz bound on the drift:
  \[
    |\bb_\theta(\bx,\mu) - \bb_\theta (\bx',\mu')|\leq C(|\bx-\bx'|+W_2(\mu,\mu')).
    \]
\end{proof}
\section{Proof of Proposition \ref{prop:propogation_of_chaos}}\label{app:proof-chaos}
\begin{proof}
    The proof is based on a synchronous coupling argument \cite{sznitman2006topics}. We construct the $N$-particle system and a set of $N$ independent ``ideal'' mean-field processes driven by the same Brownian motions and show that their $L^2$ distance vanishes as $N \to \infty$. Let us define $N$ independent processes $\bar{\mathcal{X}}^{\theta,N}_t = \left(\bar \bX^{\theta,1,N}_t,\cdots, \bar \bX^{\theta,N,N}_t\right)$ as the solution of $N$ SDEs:
    \[
    \begin{aligned}
     \intd \bar{\bX}^{\theta,i,N}_t  =& \bb_\theta(\bar \bX^{\theta,i,N}_t, f^\theta_t) \intd t + \sigma \intd \mathbf{B}^i_t\\
     = &\varphi_{\mathrm{int}}\left(\bar \bX^{\theta,i,N}_t,\int  \varphi_{\mathrm{e}}(\by ) f^\theta_t(\intd\by)\right)\intd t + \sigma \intd \mathbf{B}^i_t, \quad i\in \{1,\cdots,N\},
    \end{aligned}
    \]
    where $(\bB^i_t)$ are the same Brownian motions and where we recall that $f^\theta_t= \mathrm{Law}(\bar{\bX}_t^{\theta,i,N})$. Since the Brownian motions are independent, the law $f^\theta_t$ is independent of index $i$. We will show that:
    \begin{equation}\label{equ:bound}
       \frac{1}{N}\sum_{i=1}^N \mathbb E\left[\sup_{t\leq T}\left|\bX^{\theta,i,N}_t-\bar{\bX}^{\theta,i,N}_t\right|^2\right] \leq \epsilon (N,T). 
    \end{equation}   
Fix $i=1$ and define the path-space coupling
\[
\pi_N:=\mathrm{Law}\!\left( (\bX_t^{\theta,1,N})_{t\in[0,T]},\,(\bar \bX_t^{\theta,1,N})_{t\in[0,T]} \right).
\]
Its marginals are $f_{[0,T]}^{1,\theta,N}$ and $f_{[0,T]}^\theta$, respectively. Therefore, by the definition of the Wasserstein distance on path space,
\[
\begin{aligned}
W_2^2\!\left(f_{[0,T]}^{1,\theta,N},f_{[0,T]}^\theta\right)
&\le
\int \sup_{0\le t\le T}\left|\bx_t-\by_t\right|^2\,\pi_N(\intd \bx,\intd \by) \\
&=
\mathbb{E}\!\left[\sup_{0\le t\le T}\left|\bX_t^{\theta,1,N}-\bar \bX_t^{\theta,1,N}\right|^2\right].
\end{aligned}
\]
Since the right-hand side of Equation \eqref{equ:bound} tends to $0$ as $N\to\infty$, we conclude that
\[
W_2\!\left(f_{[0,T]}^{1,\theta,N},f_{[0,T]}^\theta\right)\to 0.
\]
Moreover, by exchangeability, the same estimate holds for any fixed finite collection of particles, which yields $f_{[0,T]}^\theta$-chaoticity of the learned particle system.
Using the It\^o formula, and since the stochastic terms cancel due to the synchronous coupling, we have:
\[
\left|\bX^{\theta,i,N}_t - \bar{\bX}^{\theta,i,N}_t\right|^2 = 
2 \int_0^t \left\langle \bX^{\theta,i,N}_s - \bar{\bX}^{\theta,i,N}_s, 
\bb_\theta(\bX^{\theta,i,N}_s, \mu_s^{\theta,N}) - \bb_\theta(\bar{\bX}^{\theta,i,N}_s, f^\theta_s) \right\rangle \intd s.
\]
Taking the supremum and then expectation:
\[
\begin{aligned}
& \mathbb E\left[\sup_{t\leq T}\left|\bX^{\theta,i,N}_t-\bar{\bX}^{\theta,i,N}_t\right|^2\right] \\
& = \mathbb E\left[\sup_{t\leq T} \left| 2 \int_0^t \left\langle \bX^{\theta,i,N}_s-\bar{\bX}^{\theta,i,N}_s, \bb_\theta(\bX^{\theta,i,N}_s, \mu_s^{\theta,N}) - \bb_\theta(\bar{\bX}^{\theta,i,N}_s, f^\theta_s) \right\rangle \intd s \right| \right] \\
& \leq 2 \int_0^T \mathbb E\left[ \left| \left\langle \bX^{\theta,i,N}_s-\bar{\bX}^{\theta,i,N}_s, \bb_\theta(\bX^{\theta,i,N}_s, \mu_s^{\theta,N}) - \bb_\theta(\bar{\bX}^{\theta,i,N}_s, f^\theta_s) \right\rangle \right| \right] \intd s \\
& \leq \int_0^T \mathbb E\left[ \left|\bX^{\theta,i,N}_s-\bar{\bX}^{\theta,i,N}_s\right|^2 \right] \intd s + \int_0^T \mathbb E\left[ \left| \bb_\theta(\bX^{\theta,i,N}_s, \mu_s^{\theta,N}) - \bb_\theta(\bar{\bX}^{\theta,i,N}_s, f^\theta_s) \right|^2 \right] \intd s \\
& \leq \int_0^T \mathbb E\left[\sup_{r\leq s}\left|\bX^{\theta,i,N}_r-\bar{\bX}^{\theta,i,N}_r\right|^2\right] \intd s + \int_0^T \mathbb E\left[ \left| \bb_\theta(\bX^{\theta,i,N}_s, \mu_s^{\theta,N}) - \bb_\theta(\bar{\bX}^{\theta,i,N}_s, f^\theta_s) \right|^2 \right] \intd s.
\end{aligned}
\]
The drift term can be split into two terms as follows:
\[
\begin{aligned}
     &\mathbb E \left[ \left|\bb_\theta(\bX^{\theta,i,N}_s,\mu^{\theta,N}_s)
     -\bb_\theta(\bar{\bX}^{\theta,i,N}_s,f^\theta_s)  \right|^2\right]\\
     &\quad \leq 2\mathbb E \left[ \left|\bb_\theta(\bX^{\theta,i,N}_s,\mu^{\theta,N}_s)-\bb_\theta(\bar{\bX}^{\theta,i,N}_s,\bar{\mu}^{\theta,N}_s)  \right|^2\right]+ 2\mathbb E \left[ \left|\bb_\theta(\bar{\bX}^{\theta,i,N}_s,\bar{\mu}^{\theta,N}_s)-\bb_\theta(\bar{\bX}^{\theta,i,N}_s,f^\theta_s)  \right|^2\right],
\end{aligned}
\]
where $\bar{\mu}^{\theta,N}_t = \frac{1}{N}\sum_{i=1}^N \delta_{\bar{\bX}^{\theta,i,N}_t}$.
For the first term, we have:
\[
\begin{aligned}
   \mathbb E \left[ \left|\bb_\theta(\bX^{\theta,i,N}_s,\mu^{\theta,N}_s)-\bb_\theta(\bar{\bX}^{\theta,i,N}_s,\bar{\mu}^{\theta,N}_s)  \right|^2\right] 
   \leq &2 C_i^2 \mathbb E \left[\left(\bX^{\theta,i,N}_s - \bar{\bX}^{\theta,i,N}_s \right)^2 \right]  + 2 C_i^2 C_e^2 \mathbb E\left[\left(W_2\left( \mu^{\theta,N}_s,\bar{\mu}^{\theta,N}_s\right) \right)^2\right]\\
   \leq & 2C_i^2(1+C_e^2)\mathbb E \left[|\bX^{\theta,i,N}_s - \bar{\bX}^{\theta,i,N}_s|^2 \right],
\end{aligned}
\]
and for the second term:
$$\begin{aligned}
    & \mathbb E \left[ \left|\bb_\theta(\bar{\bX}^{\theta,i,N}_s,\bar{\mu}^{\theta,N}_s)-\bb_\theta(\bar{\bX}^{\theta,i,N}_s,f^\theta_s)  \right|^2\right]\\
    \leq & (C_i C_e)^2 \mathbb E \left[ \left| \int \varphi_{\mathrm{e}}(\bx) \bar{ \mu}^{\theta,N}_s(\intd\bx)  - \int \varphi_{\mathrm{e}}(\bx)f^\theta_s(\bx)\intd \bx \right|^2 \right] \\
    \leq & (C_i C_e)^2 \mathbb E \left[\left|  \frac{1}{N} \sum_{j=1}^N \varphi_{\mathrm{e}} (\bar{\bX}^{\theta,j,N}_s ) - \mathbb E[\varphi_{\mathrm{e}}(\bar{\bX}^{\theta,j,N}_s )] \right|^2  \right].
\end{aligned}$$
Since the $\bar{\bX}^{\theta,j,N}_s$ are i.i.d. random variables with law $f^\theta_s$, the term inside the expectation is the squared error of a sample mean estimate. This is equal to the variance of the sample mean:
$$ \mathbb E \left[ \left|\bb_\theta(\bar{\bX}^{\theta,i,N}_s,\bar{\mu}^{\theta,N}_s)-\bb_\theta(\bar{\bX}^{\theta,i,N}_s,f^\theta_s)  \right|^2\right] \leq\frac{(C_i C_e)^2}{N} \text{Var}\left( \varphi_{\mathrm{e}}(\bar{\bX}^{\theta,j,N}_s) \right).$$
Since $\varphi_{\mathrm{e}}$ 
is Lipschitz and $f^\theta_s$ has finite second moments (guaranteed by linear growth and Lipschitz continuity of $\bb$), the variance is finite and bounded by some constant $C_V$. Therefore, the difference is bounded by:
$$ \mathbb E \left[ \left|\bb_\theta(\bar{\bX}^{\theta,i,N}_s,\bar{\mu}^{\theta,N}_s)-\bb_\theta(\bar{\bX}^{\theta,i,N}_s,f^\theta_s)  \right|^2\right] \leq \frac{(C_i C_e)^2 C_V}{N}.$$
Let $Y(t) = \mathbb E\left[\sup_{r\leq t}\left|\bX^{\theta,i,N}_r-\bar{\bX}^{\theta,i,N}_r\right|^2\right]$. Using the It\^o inequality, we have:
$$ \begin{aligned} Y(T) &\leq \int_0^T Y(s) \intd s + \int_0^T \mathbb E\left[ \left| \bb_\theta(\bX^{\theta,i,N}_s, \mu_s^{\theta,N}) - \bb_\theta(\bar{\bX}^{\theta,i,N}_s, f^\theta_s) \right|^2 \right] \intd s \\
&\leq \int_0^T Y(s) \intd s + \int_0^T \left( 2(2 C_i^2 + 2 (C_i C_e)^2) \mathbb E \left[|\bX^{\theta,i,N}_s - \bar{\bX}^{\theta,i,N}_s|^2 \right] + 2 \frac{(C_i C_e)^2 C_V}{N} \right) \intd s \\
&\leq \frac{2 T (C_i C_e)^2 C_V}{N} + \left(1 + 4 C_i^2 + 4 (C_i C_e)^2 \right) \int_0^T Y(s) \intd s.
\end{aligned}$$
The conclusion follows by the Gronwall lemma.
\end{proof} 
\section{Proof of Theorem \ref{thm:scaleing_law}}\label{app:proof-scaling}
First, we provide an approximation result for Lipschitz functions by deep neural networks. It follows the same arguments as in \cite{liu2024neural,weihs2025deep}, but we need to include the dependence on the Lipschitz constant explicitly here. 
We keep the dependence on the Lipschitz constant $L_h$ explicit because, in the proof of Theorem~\ref{thm:scaleing_law}, the theorem is applied to an auxiliary finite-dimensional map whose Lipschitz constant depends on $L_f$ and the covering complexity; this dependence must therefore be tracked in the final approximation-rate estimate.
\begin{theorem}\label{theorem:Lip_depend_scaleing}
    Let $d_1\in\mathbb N$, $\gamma_1>0$, and $\Omega_h=[-\gamma_1,\gamma_1]^{d_1}$.
Let $h:\Omega_h\to\mathbb R$ be $L_h$-Lipschitz, i.e.
\[
|h(x)-h(y)|\le L_h\|x-y\|_2,\qquad \forall x,y\in\Omega_h,
\]
and bounded, i.e., $\|h\|_{L^\infty(\Omega_h)}\le \beta_h$. We have that there exists a constant $C$ dependent on $\gamma_1$ such that the following holds: for any $\epsilon>0$, set $N = C L_h \sqrt{d_1}\epsilon^{-1}$. Let $\{\bc_k\}_{k=1}^{N^{d_1}}$ be a uniform grid on $\Omega_h$ with spacing $\frac{2\gamma_1}{N-1}$ along each dimension. There exist a neural network architecture $\mathcal F_{NN}(d_1,1,L,p,K,\kappa,R)$  and networks $\{\tilde q_k\}_{k=1}^{N^{d_1}}$ with $\tilde q_k\in \mathcal F_{NN}(d_1,1,L,p,K,\kappa,R)$, for $k=1,\cdots,N^{d_1}$, such that for any $h$ satisfying the above assumption, we have 
    \begin{equation}
        \left\|h-\sum_{k=1}^{N^{d_1}} h(\bc_k) \tilde q_k\right\|_{L^\infty(\Omega_h)} \leq \epsilon,
    \end{equation}
where
\[
\begin{aligned}
L  & = O(d_1^2 \log(\epsilon^{-1}) + d_1^2 L_h + d_1 ^2 \log(d_1) )\\
p &= O(1)\\
K &= O(d_1^2 \log(\epsilon^{-1}) + d_1^2 L_h + d_1 ^2 \log(d_1) )\\
\kappa &=  O\left(d_1^{\frac{d_1}{2}+1}\epsilon^{-d_1-1}L_h^{d_1}\right)\\
R & = O(1).
\end{aligned}
\]
\end{theorem}
\begin{proof}
    We partition $\Omega_h$ into $N^{d_1}$ subcubes for some $N$ to be specified later. We are going to approximate $h$ on each cube by a constant function and then assemble them together to get an approximation of $h$ on $\Omega_h$. Denote the centers of the subcubes by $\{\bc_k\}_{k=1}^{N^{d_1}}$ with $\bc_k = [c_{k,1},\cdots,c_{k,d_1}]^\top$. Let $\{\bc_k\}_{k=1}^{N^{d_1}}$ be a uniform grid on $\Omega_h$ so that each $\bc_k \in \left\{ -\gamma_1,-\gamma_1+ \frac{2\gamma_1}{N-1},\cdots, \gamma_1 \right\}^{d_1}$ for each $k$. Define 
    \begin{equation}
        \psi(a) = \left\{
    \begin{aligned}
    &1,&|a|\leq 1,\\
    &0, &|a|>2,\\
    &2-|a|,& 1\leq |a| \leq 2, 
    \end{aligned}
    \right.
    \end{equation}
with $a\in \mathbb R$, and 
\[
\phi_{\bc_k} (\bx) = \prod_{j=1}^{d_1} \psi\left(\frac{3(N-1)}{2\gamma_1}(x_j-c_{k,j})\right).
\]
For any $h$, we construct a piecewise constant approximation of $h$ as 
\[
\bar h (\bx)  = \sum_{k=1}^{N^{d_1}} h(\bc_k)\phi_{\bc_k} (\bx).
\]
By utilizing the partition of unity property given by $ \sum_{k=1}^{N^{d_1}}\phi_{\bc_k} (\bx)=1$, it follows that for any $\bx\in \Omega_h$, we have 
\begin{equation}
\begin{aligned}
    |h(\bx)-\bar h (\bx)| =& \left|\sum_{k=1}^{N^{d_1}} \phi_{\bc_k}(\bx) \left(h(\bx)-h(\bc_k)\right) \right|\\
    \leq & \sum_{k=1}^{N^{d_1}}  \phi_{\bc_k}(\bx) \left| h(\bx)-h(\bc_k)\right| \\
    \leq & \sum_{k:\|\bc_k-\bx\|_\infty \leq \frac{2\gamma_1}{N-1}}  \phi_{\bc_k}(\bx) \left| h(\bx)-h(\bc_k)\right| \\
    \leq & \max_{k:\|\bc_k-\bx\|_\infty \leq \frac{2\gamma_1}{N-1}}  \left| h(\bx)-h(\bc_k)\right|\left(\sum_{k:\|\bc_k-\bx\|_\infty \leq \frac{2\gamma_1}{N-1}}  \phi_{\bc_k}(\bx)\right)\\
     \leq & \max_{k:\|\bc_k-\bx\|_\infty \leq \frac{2\gamma_1}{N-1}}  \left| h(\bx)-h(\bc_k)\right|
     \leq \frac{2\sqrt{d_1}\gamma_1L_h}{N-1}.
\end{aligned}
\end{equation}
Setting $N= \lceil \frac{4 \sqrt{d_1} \gamma_1 L_h}{\epsilon} \rceil+1$ yields:
\[
|h(\bx) - \bar h(\bx)|\leq \frac{\epsilon}{2} ,\quad \forall \bx \in \Omega_h.
\]
Then we show that $\phi_{\bc_k}$ can be approximated by a network with arbitrary accuracy. Note $\phi_{\bc_k}$ is a product of $d_1$ functions, and each of them can be a piecewise linear function and can be realized by $4\mhyphen$layer ReLU networks. 
\begin{lemma}\label{lemma:relunn}
    Given $M>0$ and $\epsilon>0$, there is a ReLU network $\tilde f :\mathbb R^2 \to \mathbb R $ in $\mathcal F_{NN}(2,1,L,p,K,\kappa, R)$ such that for any $|x|\leq M $ and $|y|\leq M$, we have  
    \[
    \left| \tilde f (x,y) - xy\right| <\epsilon, 
    \]
    where $L=O(\log \epsilon^{-1})$, $p = 6$, $K=O(\log \epsilon ^{-1})$, $\kappa = O(\epsilon^{-1})$, $R=M^2$. The constant hidden in $O$ depends on $M$.
\end{lemma}
Let $\tilde f$ be the network defined in Lemma \ref{lemma:relunn} with accuracy $\delta$. We approximate $\phi_{\bc_k}$ by $\tilde q_k$ defined as: 
\[
\tilde q_k(\bx) = \tilde f\left(\psi \left(\frac{3(N-1)}{2\gamma_1}(x_1-c_{k,1})\right), \tilde f\left(\psi \left(\frac{3(N-1)}{2\gamma_1}(x_2-c_{k,2})\right),\cdots\right)\right).
\]
For each $k,\tilde q_k \in \mathcal F_{NN} (d_1,1,L,p,K,\kappa,R)$ with 
\[
L = O(d_1\log \delta^{-1}), p = O(1), K = O(d_1\log\delta^{-1}),\kappa=O(\delta^{-1}+N) , R= 1.
\]
For any $\bx\in \Omega_h$, we have that 
\[
\begin{aligned}
    &|\tilde q _k(\bx)-\phi_{\bc_k} (\bx)| \\
\leq & \left|\tilde f\left(\psi \left(\frac{3(N-1)}{2\gamma_1}(x_1-c_{k,1})\right), \tilde f\left(\psi \left(\frac{3(N-1)}{2\gamma_1}(x_2-c_{k,2})\right),\cdots\right)\right)    -\phi_{\bc_k} (\bx)\right|\\  
\leq & \left|\tilde f\left(\psi \left(\frac{3(N-1)}{2\gamma_1}(x_1-c_{k,1})\right), \tilde f\left(\psi \left(\frac{3(N-1)}{2\gamma_1}(x_2-c_{k,2})\right),\cdots\right)\right)\right.  \\
& - \left. \psi \left(\frac{3(N-1)}{2\gamma_1}(x_1-c_{k,1})\right) \tilde f\left(\psi \left(\frac{3(N-1)}{2\gamma_1}(x_2-c_{k,2})\right),\cdots\right) \right|\\
& + \left|\psi \left(\frac{3(N-1)}{2\gamma_1}(x_1-c_{k,1})\right)\tilde f\left(\psi \left(\frac{3(N-1)}{2\gamma_1}(x_2-c_{k,2})\right),\cdots\right)
-\phi_{\bc_k} (\bx)\right|\\
\leq & \delta + \mathcal E_2,
\end{aligned}
\]
where
\[
\begin{aligned}
\mathcal E_2 = & \left|\psi \left(\frac{3(N-1)}{2\gamma_1}(x_1-c_{k,1})\right)\tilde f\left(\psi \left(\frac{3(N-1)}{2\gamma_1}(x_2-c_{k,2})\right),\cdots\right)
-\phi_{\bc_k} (\bx)\right|\\
= &  \left|\psi \left(\frac{3(N-1)}{2\gamma_1}(x_1-c_{k,1})\right)\right|\left|\tilde f\left(\psi \left(\frac{3(N-1)}{2\gamma_1}(x_2-c_{k,2})\right),\cdots\right)
-\prod_{j=2}^{d_1} \psi \left(\frac{3(N-1)}{2\gamma_1}(x_j-c_{k,j})\right)\right|.
\end{aligned}
\]
Repeat this process to estimate $\mathcal E_2,\mathcal E_3,\cdots, \mathcal E_{d_1+1}$, where $\mathcal E_{d_1+1}  = 0$. This implies that $\|\phi_{\bc_k}-\tilde q_k\|_{L^{\infty}(\Omega_h)}\leq d_1\delta$. It follows that
\[
\begin{aligned}
    \left\|\sum_{k=1}^{N^{d_1}} h(\bc_k)\tilde q_k -\bar h\right\|_{L^\infty(\Omega_h)}
    = 
    \left\|\sum_{k=1}^{N^{d_1}} h(\bc_k)\tilde q_k -\sum_{k=1}^{N^{d_1}} h(\bc_k)\phi_{\bc_k}\right\|_{L^\infty(\Omega_h)}
    \leq \sum_{k=1}^{N^{d_1}} |h(\bc_k)|\left\|\tilde q_k  - \phi_{\bc_k}\right\|_{L^\infty(\Omega_h)}
    \leq d_1 N^{d_1}\beta_h \delta
\end{aligned}
\]
and setting $\delta=  \frac{\epsilon}{2d_1N^{d_1}\beta_h}$ yields:
\[
\begin{aligned}
    \left\| h - \sum_{k=1}^{N^{d_1}} h(\bc_k)\tilde q_k  \right\|_{L^\infty(\Omega_h)} \leq \epsilon.
\end{aligned}
\]
Therefore 
\[
\begin{aligned}
L  &= O\left(d_1\log\left(\frac{2d_1N^{d_1}\beta_h}{\epsilon}\right)\right) =  O\left(d_1\log(\epsilon^{-1}) + d_1^2 \log(N) \right) = O(d_1^2 \log(\epsilon^{-1}) + d_1^2 L_h + d_1 ^2 \log(d_1) )\\
p &= O(1)\\
K &= O(d_1^2 \log(\epsilon^{-1}) + d_1^2 L_h + d_1 ^2 \log(d_1) )\\
\kappa &= O\left(\frac{d_1\left( \frac{\sqrt{d_1}  L_h}{\epsilon}\right)^{d_1}}{\epsilon}  \right) =  O\left(d_1^{\frac{d_1}{2}+1}\epsilon^{-d_1-1}L_h^{d_1}\right)\\
R & = O(1).
\end{aligned}
\]
\end{proof}

\begin{proof}[Proof of Theorem \ref{thm:scaleing_law}.]
    By Lemma 2 in \cite{liu2024neural}, there exists a cover \(\{\mathcal B_\delta (\mathbf c_m)\}_{m=1}^{C_\Omega}\) of \(\Omega\) by $C_\Omega$ Euclidean balls with $C_\Omega\leq C\delta^{-d}$.  There exists a partition of unity $\{\omega_m(\bx)\}_{m=1}^{C_\Omega}$ subordinate to the cover \(\{\mathcal B_\delta (\mathbf c_m)\}_{m=1}^{C_\Omega}\) such that $0\leq \omega_m\in C^\infty (\Omega)$ and $\sum_{m=1}^{C_\Omega}\omega_m(\bx)=1$ for all $\bx\in \Omega$.
    For any $\mu\in U$, we define 
   \(\bu=\left(\left\langle \omega_1,\mu\right\rangle, \cdots , \left\langle \omega_{C_\Omega},\mu\right\rangle\right)
    \). We note that: 
    \[\sum_{m=1}^{C_\Omega} \bu_m = \int_\Omega \sum_{m=1}^{C_\Omega} \omega_m(\bx)\mu(\intd \bx)=1,\]
    and $\bu_m\geq0$ for all $m$.
    Therefore, we construct the approximate measure: 
    \[
    \mu_\omega = \sum_{m=1}^{C_\Omega} \bu_m \delta_{\bc_m}(\intd \bx).
    \]
    The $W_2$ error estimation of the approximation is given by: 
    \[
    \begin{aligned}
    W^2_2(\mu,\mu_\omega ) =& \inf_{\pi\in \Pi(\mu,\mu_\omega)}\int_{\Omega\times\Omega} |\bx-\by|^2 \intd \pi(\bx,\by)\\
        \leq & \int_{\Omega\times\Omega} |\bx-\by|^2 \intd \pi^*(\bx,\by) \quad \left( \intd \pi^*(\bx,\by) = \left(\sum_{m=1}^{C_\Omega}\omega_m(\bx)\delta_{\bc_m}(\intd \by )\right)\intd \mu(\bx)  \right)\\
        = & \sum_{m=1}^{C_\Omega}\int_{\Omega\times\Omega} |\bx-\bc_m|^2 \omega_m(\bx)\intd \mu(\bx)\\
        = & \sum_{m:\|\bx-\bc_m\|_2\leq \delta}\int_{\Omega\times\Omega} |\bx-\bc_m|^2 \omega_m(\bx)\intd \mu(\bx)\\
        \leq & \sum_{m:\|\bx-\bc_m\|_2\leq \delta} \delta ^2 \bu_m  \leq \delta ^2.
    \end{aligned}
    \]
    Setting $\delta = \frac{\epsilon}{2L_f}$ and using the Lipschitz property of $f$, we have:
    \[
    \begin{aligned}
    |f(\bx , \mu)- f(\bx,\mu_\omega) |\leq & L_f  W_1(\mu,\mu_\omega)\\        
    \leq& L_f  W_2(\mu,\mu_\omega) \\
    \leq& \frac{\epsilon}{2}.
    \end{aligned}
    \]
    We also have that $C_\Omega\leq  C \epsilon^{-d}$.
    We next define a function $g:\mathbb R^d\times \mathbb R^{C_\Omega}\to \mathbb R$ such that $g(\bx,\bu) = f(\bx,\mu_\omega)$. We claim that $g$ is Lipschitz in the following sense: for any $\bx,\by\in \mathbb R^d$, and for any $\mu,\nu\in \mathcal P_2(\Omega)$, define $\mu_\omega$ and $\nu_\omega$ as before with coefficients $\bu$ and $\bv$ respectively. We have:
    \[
    \begin{aligned}
        |g(\bx,\bu)-g(\by,\bv)| =&  |f(\bx,\mu_\omega)-f(\by,\nu_\omega)|\\
        \leq & L_f\left(|\bx-\by|_2 + W_1(\mu_\omega,\nu_\omega )  \right).
     \end{aligned}
    \]
    Notice that 
    \[
    \begin{aligned}
      W_1(\mu_\omega, \nu_\omega) =& \sup_{\mathrm{Lip}(t) \le 1} \left\{ \int_{\Omega} t(\bx) \intd\mu_\omega(\bx) - \int_{\Omega} t(\bx) \intd\nu_\omega(\bx) \right\} \\
    = & \sup_{\mathrm{Lip}(t) \le 1} \left\{\int_{\Omega} t(\bx) \sum_{m=1}^{C_\Omega} (\bu_m -\bv_m )\delta_{\bc_m}(\intd \bx)\right\}\\
    = & \sup_{\mathrm{Lip}(t) \le 1} \left\{ \sum_{m=1}^{C_\Omega}  t(\bc_m)  (\bu_m -\bv_m )\right\}.
    \end{aligned}
    \]
    Letting $\bw^+_m = \max\{\bu_m - \bv_m,0\}$ and $\bw^-_m=-\min\{\bu_m-\bv_m,0\}$ yields: $$
    W_1(\mu_\omega, \nu_\omega) = \sup_{\mathrm{Lip}(t) \le 1} \left\{ \sum_{m=1}^{C_\Omega}  t(\bc_m) \bw^+_m-t(\bc_m) \bw^-_m\right\}.
    $$
    Next, define $M = \sum_{m=1}^{C_\Omega} \bw^+_m = \sum_{m=1}^{C_\Omega} \bw^-_m = \frac{1}{2} |\bu - \bv|_1$ and thus the following bound holds:
    \[
    \begin{aligned}
        W_1(\mu_\omega, \nu_\omega) \leq & \sup_{\mathrm{Lip}(t) \le 1} \left\{ \sum_{m=1}^{C_\Omega}  t_{\max} \bw^+_m-t_{\min} \bw^-_m\right\}\\
    \leq &  M (t_{\max}-t_{\min})\leq 2MR\sqrt{d}\leq 2R\sqrt{d}|\bu-\bv|_1\leq 2R\sqrt{{C_\Omega}d}|\bu-\bv|_2 .
    \end{aligned}
    \]
Therefore, we have: 
\[
\begin{aligned}
|g(\bx,\bu)-g(\by,\bv)| \leq &  L_f\left(|\bx-\by|_2 + 2R\sqrt{{C_\Omega}d}\|\bu-\bv\|_2   \right)\\
\leq & 2R\sqrt{{C_\Omega}d}L_f \sqrt{|\bx-\by|_2^2 +|\bu-\bv|^2_2 }.
\end{aligned}
\]
Therefore, by Theorem \ref{theorem:Lip_depend_scaleing}, for $\epsilon>0$, we set
$
H = N^{d_1},
$
where $N =   \lfloor \frac{4\sqrt{d_1}L_U\gamma_1 }{\epsilon}\rfloor +1 $ and $L_U = 2R\sqrt{{C_\Omega}d}L_f, d_1=d+C_\Omega$, $\gamma_1 = \max\{1,R\}$.
Thus: 
\[
H \;\lesssim\; \left(\frac{C\,\gamma_1\,\sqrt{d+C_\Omega}\,R\sqrt{C_\Omega d}\,L_f}{\epsilon}\right)^{d+C_\Omega}\leq C(C_\Omega+d)^{\frac{C_\Omega+d}{2}}(C_\Omega d)^{\frac{C_\Omega+d}{2}}\epsilon^{-C_\Omega-d},
\]
where $C$ depends on $R,L_f$. Then there exist $H$ neural networks $\{q_k\}_{k=1}^H \in \mathcal F_{NN}(d+C_\Omega,1,L,p,K,\kappa,R)  $ with \[
\begin{aligned}
L  & = O((d+C_\Omega)^2 \log(\epsilon^{-1}) + (d+C_\Omega)^2 \sqrt{{C_\Omega}d}L_f + (d+C_\Omega)^2 \log(d+C_\Omega))\\
p &= O(1)\\
K &= O((d+C_\Omega)^2 \log(\epsilon^{-1}) + (d+C_\Omega)^2\sqrt{{C_\Omega}d}L_f + (d+C_\Omega) ^2 \log(d+C_\Omega) )\\
\kappa &=  O\left((d+C_\Omega)^{\frac{d+C_\Omega}{2}+1}\epsilon^{-d-C_\Omega-1}L_U^{d+C_\Omega}\right)\\
R & = O(1),
\end{aligned}
\] such that 
\[
\sup_{\mu\in U ,\bx\in \Omega }\left|g(\bx,\bu)-\sum_{k=1}^H a_k q_k(\bx,\bu)\right|\leq \frac{\epsilon}{2},
\]
where $a_k$ are constants dependent on $f$. We have 
\[
\begin{aligned}
    \sup_{\mu\in U ,\bx\in \Omega }\left|f(\bx,\mu) - \sum_{k=1}^H a_k q_k(\bx,\bu)\right|\leq &\sup_{\mu\in U ,\bx\in \Omega }\left|f(\bx,\mu) - g(\bx,\bu)\right|+ \sup_{\mu\in U ,\bx\in \Omega }\left|g(\bx,\bu)-\sum_{k=1}^H a_k q_k(\bx,\bu)\right|\\
= & \sup_{\mu}\left|f(\bx,\mu) - f(\bx,\mu_\omega)\right|+ \sup_{\mu}\left|g(\bx,\bu)-\sum_{k=1}^H a_k q_k(\bx,\bu)\right|\\
\leq & \epsilon.
\end{aligned}
\]
\end{proof}

\section{Proof of Theorem \ref{thm:mvnn-approximation}}\label{app:proof-approximation}

\begin{proof}
    By Assumption~\ref{assum:finite-dimensional} and Remark~\ref{remark:G}, there exists an $L_G$-Lipschitz function 
    \(
        G : \mathbb{R}^d \times \mathbb{R}^r \to \mathbb{R}^d
    \)
    such that
    \(
        \bb^\star(\bX, \mu) = G(\bX, \langle \bg, \mu \rangle),
    \)
    where $\bg = (g_1,\dots,g_r)$ is the vector of $L_{\bg}$-Lipschitz feature functions.
    Since $\bg$ is Lipschitz and the domain $[0,1]^d$ is compact, each $g_j$ is bounded. Let:
    \[
        M_{\bg} := \sup_{\bY \in [0,1]^d} \|\bg(\bY)\|_{\infty} < \infty,
    \]
    then, for any probability measure $\mu$ supported on $[0,1]^d$:
    \[
        \|\langle \bg, \mu \rangle\|_{\infty} 
        \le \int \|\bg(\bY)\|_{\infty} \,\mu(d\bY) 
        \le M_{\bg}.
    \]
    Our goal is to construct an MVNN of the form:
    \(        \bb_\theta(\bX, \mu) = \varphi_{\mathrm{int}}\bigl(\bX, \langle \varphi_{\mathrm{e}}, \mu \rangle\bigr)
    \)
    such that:
    \[
        \sup_{\bX \in [0,1]^d,\,\mu \in U} 
        \bigl\|\bb^\star(\bX,\mu) - \bb_\theta (\bX,\mu)\bigr\|_{\mathbb{R}^d} \le \epsilon.
    \]
    For brevity, denote
    \(  \bz_{\bg}(\mu) := \langle \bg, \mu \rangle \in \mathbb{R}^r,
        \bz_{\mathrm{e}}(\mu) := \langle \varphi_{\mathrm{e}}^{(1:r)}, \mu \rangle \in \mathbb{R}^r,
    \)  where $\varphi_{\mathrm{e}}^{(1:r)}$ denotes the first $r$ coordinates of $\varphi_{\mathrm{e}}$.
    We decompose the total error as:
    \[
    \begin{aligned}
    &\sup_{\bX,\mu} 
           \bigl\|G(\bX, \langle \bg, \mu \rangle) 
                 - \varphi_{\mathrm{int}}(\bX,  \langle \varphi_{\mathrm{e}}^{(1:r)}, \mu \rangle)\bigr\|_{\mathbb{R}^d} \\
        &\le \sup_{\bX,\mu} 
           \bigl\|G(\bX, \langle \bg, \mu \rangle) 
                 - G(\bX,  \langle \varphi_{\mathrm{e}}^{(1:r)}, \mu \rangle)\bigr\|_{\mathbb{R}^d} \\
        &\quad + \sup_{\bX,\mu} 
           \bigl\|G(\bX,  \langle \varphi_{\mathrm{e}}^{(1:r)}, \mu \rangle) 
                 - \varphi_{\mathrm{int}}(\bX,  \langle \varphi_{\mathrm{e}}^{(1:r)}, \mu \rangle)\bigr\|_{\mathbb{R}^d}.
    \end{aligned}
    \]
    For the first term, since $G$ is $L_G$-Lipschitz in its second argument, we have:
    \[
    \begin{aligned}
        & \sup_{\bX,\mu} 
           \bigl\|G(\bX, \langle \bg, \mu \rangle) 
                 - G(\bX,  \langle \varphi_{\mathrm{e}}^{(1:r)}, \mu \rangle)\bigr\|_{\mathbb{R}^d}\\
                 \le &
        L_G \sup_{\mu \in U} 
        \bigl\| \langle \bg, \mu \rangle  - \langle \varphi_{\mathrm{e}}^{(1:r)}, \mu \rangle \bigr\|_{\mathbb{R}^r}\\ 
        =& L_G \left\| \int \bigl(\bg(\bY) - \varphi_{\mathrm{e}}^{(1:r)}(\bY)\bigr) \,\mu(d\bY)\right\|_{\mathbb{R}^r} \\
        \le& L_G\int \bigl\|\bg(\bY) - \varphi_{\mathrm{e}}^{(1:r)}(\bY)\bigr\|_{\mathbb{R}^r} \,\mu(d\bY) \\
        \le & L_G \sup_{\bY \in [0,1]^d} 
            \bigl\|\bg(\bY) - \varphi_{\mathrm{e}}^{(1:r)}(\bY)\bigr\|_{\mathbb{R}^r}.
        \end{aligned}
    \]
    We will enforce:
    \[
        \sup_{\bY \in [0,1]^d} 
        \bigl\|\bg(\bY) - \varphi_{\mathrm{e}}^{(1:r)}(\bY)\bigr\|_{\mathbb{R}^r}
        \le \epsilon_{\mathrm{e}} := \frac{\epsilon}{2 L_G},
    \]
    so that \begin{equation}\label{equ:e1}
    \sup_{\bX,\mu} 
           \bigl\|G(\bX, \langle \bg, \mu \rangle) 
                 - G(\bX,  \langle \varphi_{\mathrm{e}}^{(1:r)}, \mu \rangle)\bigr\|_{\mathbb{R}^d} \le \epsilon/2.
                 \end{equation}
                 
     By Theorem~1 in \cite{yarotsky2017error}, for any $\epsilon_{\mathrm{e}} > 0$, there exists a deep ReLU network 
    \[
        \varphi_{\mathrm{e}} : \mathbb{R}^d \to \mathbb{R}^k, \quad k \ge r,
    \]
    with depth at most $O(\log(\epsilon_{\mathrm{e}}^{-1})) = O(\log(\epsilon^{-1}))$ and width at most $O\bigl(r \cdot \epsilon_{\mathrm{e}}^{-d} \bigr)
        = O\bigl(r \cdot \epsilon^{-d}\bigr)$ and the total number of parameters at most $O(r \cdot  \epsilon^{-d} \log(\epsilon^{-1}) )$, such that
    \[
        \sup_{\bY \in [0,1]^d} 
        \bigl\|\bg(\bY) - \varphi_{\mathrm{e}}^{(1:r)}(\bY)\bigr\|_{\mathbb{R}^r}
        \le \epsilon_{\mathrm{e}},
    \]
    where the constant depends on $L_G,d$.
    For the second term, consider the compact domain
    \[
        \mathcal{D} := [0,1]^d \times [-R,R]^r \subset \mathbb{R}^{d+r},
    \] where \( R := M_{\bg} + \epsilon_{\mathrm{e}}\)
    with both $\langle \bg, \mu \rangle$ and $\langle \varphi_{\mathrm{e}}^{(1:r)}, \mu \rangle$ lying in the box $[-R,R]^r$ for all $\mu \in U$. Again by Theorem~1 in \cite{yarotsky2017error}, for any $\epsilon_{\mathrm{i}} > 0$ there exists a deep ReLU network:
    \[
        \varphi_{\mathrm{int}} : \mathbb{R}^{d+r} \to \mathbb{R}^d
    \]
    with depth $L_{\mathrm{int}}$ and width $W_{\mathrm{int}}$ at most:
      \[
        L_{\mathrm{int}} =  O(\log(\epsilon^{-1})),
        \qquad
        W_{\mathrm{int}} =  O\bigl(d\cdot\epsilon^{-(d+r)} \bigr),
    \]
    such that:
     \begin{equation}\label{equ:e2}
        \sup_{(\bX,\bz) \in \mathcal{D}} 
        \bigl\|G(\bX,\bz) - \varphi_{\mathrm{int}}(\bX,\bz)\bigr\|_{\mathbb{R}^d}
        \le \epsilon/2.
    \end{equation}
  
Combining the two estimates in Equations \eqref{equ:e1} and \eqref{equ:e2} gives
\[
\sup_{\bX\in[0,1]^d,\ \mu\in U}
\|\bb^\star(\bX,\mu)-\bb_\theta(\bX,\mu)\|_{\mathbb{R}^d}
\le \frac{\epsilon}{2}+\frac{\epsilon}{2}
=\epsilon.
\]
This completes the proof.
\end{proof}
\end{document}